\documentclass[paper]{imsart}
\usepackage{appendix}
\usepackage{mathtools}
\usepackage[T1]{fontenc}
\usepackage{amsfonts}
\usepackage{amsmath}
\usepackage{amssymb}
\usepackage{amsthm}
\usepackage{thmtools}
\usepackage{bbm}
\usepackage{comment}
\usepackage{bm}
\usepackage{mathrsfs}
\usepackage{color}
\usepackage{pdfsync}

\usepackage{enumitem}
\usepackage{booktabs,threeparttable,tabularx,siunitx}
\usepackage{subcaption}
\usepackage{amsfonts,bbm,bm}

\newcommand{\cB}{\mathcal{B}}
\newcommand{\cC}{\mathcal{C}}

\newcommand{\cL}{\mathcal{L}}
\newcommand{\cM}{\mathcal{M}}

\newcommand{\cT}{\mathcal{T}}

\newcommand{\cW}{\mathcal{W}}

\newcommand{\N}{\mathbb{N}} 
\newcommand{\BB}{\mathbb{B}}

\newcommand{\vae}{\varepsilon}

\newcommand{\rom}[1]{%
  \textup{\uppercase\expandafter{\romannumeral#1}}%
}

\newcommand{\inner}[2]{\langle #1,#2\rangle}

\definecolor{darkgreen}{RGB}{0,100,0}

\newcommand*{\rd}{\mathrm{d}}
\newcommand*{\dd}{\, \rd}

\newcommand{\bigslant}[2]{{\raisebox{.2em}{$#1$}\left/\raisebox{-.2em}{$#2$}\right.}}

\usepackage{natbib}

\usepackage{xcolor} 
\definecolor{darkgreen}{rgb}{0,0.5,0} 
\definecolor{darkbkue}{rgb}{1,0,0} 

\makeatletter
\renewcommand\@makefnmark{%
  \hbox{\textsuperscript{\textcolor{darkgreen}{\@thefnmark}}}}
\makeatother
\usepackage{tikz}
\usetikzlibrary{patterns}

\DeclareMathOperator*{\diam}{diam}

\newcommand{\RR}{\mathbb{R}}

\newcommand{\R}{\RR}

\newcommand{\E}{\mathbb{E}}
\newcommand{\PP}{\mathbb{P}}
\newcommand{\NN}{\mathbb{N}}

\newcommand{\eps}{ \varepsilon}
\newcommand{\diff}{\mathrm{d}}

\newcommand{\cP}{\mathcal{P}}

\newcommand{\med}{\operatorname{med}}

\newcommand{\mykill}[1]{}

\usepackage[capitalize, nameinlink]{cleveref}

\newlist{hypotheses}{enumerate}{1}
\setlist[hypotheses]{
    label=\textcolor{darkgreen}{(H\arabic*)},
    ref=(H\arabic*)
}
\theoremstyle{plain}
\newtheorem{theorem}{Theorem}[section]
\newtheorem{proposition}[theorem]{Proposition}
\newtheorem{lemma}[theorem]{Lemma}
\newtheorem{corollary}[theorem]{Corollary}
\theoremstyle{definition}
\newtheorem{remark}[theorem]{Remark}

\newtheorem{assumption}[theorem]{Assumption}

\newlist{myenum}{enumerate}{3}
\setlist[myenum,1]{label={\rm (H\arabic*)},
                   ref  ={\rm (H\arabic*)}}
\crefname{myenumi}{property}{properties}

\begin{document}

\begin{frontmatter}

\title{Empirical optimal transport potentials: fast rates\\ and a functional central limit theorem}
\runtitle{Empirical optimal transport potentials}

\begin{aug}
\author[A]{\fnms{Alberto}~\snm{Gonz\'{a}lez-Sanz}\ead[label=e1]{ag4855@columbia.edu}}
\author[B]{\fnms{Gilles}~\snm{Mordant}\ead[label=e2]{gilles.mordant@yale.edu}}
\author[A]{\fnms{Shunan}~\snm{Sheng}\ead[label=e3]{ss6574@columbia.edu}}

\address[A]{Department of Statistics,
Columbia University,
\printead[presep={\\}]{e1,e3}}

\address[B]{Department of Applied and Computational Mathematics,
Yale University,
\printead[presep={\\}]{e2}}
\end{aug}

\begin{abstract}
Optimal transport potentials are fundamental objects in statistics, economics, and machine learning: their gradients generate optimal transport maps, while the potentials themselves act as location-dependent dual prices and sensitivity variables. We study the estimation of the quadratic optimal transport potential when a fixed absolutely continuous reference distribution $\mu$ is transported to an unknown distribution $\nu$, accessed to via its empirical measure. Our main ingredient is a stability inequality that controls the $L^1(\mu)$ distance, modulo additive constants, between a strongly convex potential $\varphi$ and a convex potential $\widetilde\varphi$ by a weak dual norm of $(\nabla\widetilde\varphi)_\#\mu-(\nabla\varphi)_\#\mu,$
together with a second-order Wasserstein remainder of logarithmic type. This separation between the leading empirical-process term and the Wasserstein remainder yields faster convergence for potentials than for the corresponding transport maps.

Under smoothness and uniform convexity assumptions, the exact semidiscrete Brenier potential converges in $L^1(\mu)$ at rate $n^{-1/2}$ for $d\leq3$, at rate $n^{-1/2}(\log n)^{5/2}$ for $d=4$, and at rate
$n^{-2/d}(\log n)^{(d+2)/d}$ for $d\geq5$. The polynomial exponents are sharp. In dimensions $d\leq3$, we further establish a nondegenerate function-space central limit theorem and prove consistency of the nonparametric bootstrap. These results yield joint root-$n$ inference for every fixed finite collection of normalization-invariant weighted contrasts of the potential, including regional shadow premia in reference-based risk problems. Finally, we prove matching upper and lower bounds of order
$\varepsilon\log(1/\varepsilon)$ for the normalization-invariant sum of the entropic dual potentials.
\end{abstract}



\end{frontmatter}

\section{Introduction}

Optimal transport (OT) has become a common framework for comparing,
matching, and transforming probability distributions across statistics,
machine learning, economics, and nonlinear partial differential equations
\citep{PanaretosZemel2019,peyre2019computational,Galichon2016,
BenamouBrenier1998}.  Its dual potentials play a central role: for
quadratic transport, their gradients generate the optimal assignment,
while the potentials themselves record the local value of the marginal
constraints.  They therefore retain information that is lost when OT is
reduced to a single distance or cost.

Turning to applications, potential-level quantities arise as integrated multivariate quantiles and
risk or inequality indices
\citep{EkelandGalichonHenry2012,BercuBigotThurin2024,HallinMordant2025},
equilibrium payoff and price schedules
\citep{ChiapporiMcCannNesheim2010,Galichon2016,GalichonSalanie2022},
and learned convex functions in generative modeling
\citep{MakkuvaTaghvaeiOhLee2020,PatyDAspremontCuturi2020}.  Physical
examples include semi-geostrophic geopotentials
\citep{BenamouBrenier1998,BourneEganPelloniWilkinson2022} and
Monge--Amp\`ere--Kantorovich reconstruction of the primordial
gravitational potential from galaxy catalogues
\citep{LevyMohayaeeVonHausegger2021,
NikakhtarPadmanabhanLevyShethMohayaee2023}, see below for a more detailed exposition. 

In these domains, one distribution is often fixed while the other is learned
from data.  In practice, the potentials are typically computed after
replacing the unknown distribution by an empirical or other plug-in
estimate, often with entropic regularization.  This raises the following
statistical question:
\begin{center}
\emph{How well can we estimate the optimal transport potential as a function
of the sample size?}
\end{center}

We now introduce the setting.  We consider quadratic transport from a fixed
absolutely continuous reference distribution $\mu$ to an unknown target
distribution $\nu$.
By Brenier's theorem \citep{Brenier1991}, the optimal transport map is the
gradient of a convex potential.  Thus, writing
$\varphi:=\varphi_\nu$, $(\nabla\varphi)_\#\mu=\nu,$ 
with $\varphi$ unique up to an additive constant.  In the empirical setting,
let $X_1,\ldots,X_n$ be i.i.d.\ with law $\nu$, and write $\widehat\nu_n=\frac1n\sum_{i=1}^n\delta_{X_i}.$ 
Our estimator is the exact unregularized semidiscrete potential
$\widehat\varphi_n$ satisfying
$(\nabla\widehat\varphi_n)_\#\mu=\widehat\nu_n$.  The question above asks
for the rate and limiting distribution of
$\widehat\varphi_n-\varphi$ as a random function modulo constants.

A bound obtained indirectly from existing results on the transport map does
not reveal the correct scale.  Combining map stability with Poincar\'e's
inequality gives only
\[
\E\inf_a\|\widehat\varphi_n-\varphi-a\|_{L^1(\mu)}
\lesssim \E W_2(\widehat\nu_n,\nu),
\]
which is essentially $n^{-1/d}$ in high dimension.  Whether the potential
itself can be estimated faster is therefore not answered by the available
map bounds.

\subsection{Contributions}

Our first result is a stability inequality, which is the main building block
for the subsequent contributions.

\vspace{10pt}
\noindent
\textbf{Informal theorem (\cref{theorem:main})}.
For a strongly convex $\mathcal{C}^{3,\alpha}$ function $\varphi$, assume
that $\mu$ is regular enough.  Then there exists a constant $C>0$ such that,
for every lower semicontinuous convex function $\widetilde\varphi$,
\begin{equation}
     \| \varphi - \widetilde \varphi- \med(\varphi - \widetilde \varphi) \|_{L^{1}(\mu)} \leq C\  \mathfrak{d}(\widetilde{\nu},\nu), 
\end{equation}
where $\nu=(\nabla\varphi)_\#\mu$ and
$\widetilde\nu=(\nabla\widetilde\varphi)_\#\mu$.  Here $\med$ is the
$L^1(\mu)$-median and $\mathfrak d$, made explicit in
\cref{theorem:main}, is the sum of two terms: the norm of
$\widetilde\nu-\nu$ in the dual of $\cC^{1,\mathrm{LogLip}}$, and a term
of order $\cW_2^2(\widetilde\nu,\nu)$ up to a logarithmic factor.
\medskip

The two terms play different roles.  The first is a weak dual norm and is
the leading term.  The Wasserstein distance, which controls the transport
maps, appears only through a second-order remainder.  This is why the
potential can be estimated faster than the map.

This stability inequality yields our main statistical result.

\vspace{10pt}
\noindent
\textbf{Informal theorem (\Cref{Thm:rates-and-CLT})}.
 Let $\varphi$ and $\mu$ be as in \cref{theorem:main}.  Let
 $\widehat{\varphi}_n$ be a convex potential such that
 $\nabla\widehat{\varphi}_n$ pushes $\mu$ forward to the empirical measure
 $\widehat{\nu}_n$.  Then
    \begin{enumerate}
    \item $ \E\big[ \inf_a\|  \widehat{\varphi}_n- \varphi-a \|_{L^{1}(\mu)}\big]  \lesssim  \begin{cases}
    n^{-\frac{1}{2}}  & \text{if } d=1,2,3,\\
         n^{-\frac{1}{2}} (\log(n))^{\frac{5}{2}}  & \text{if } d=4,\\
       n^{-2/d}\left( \log (n)\right) ^{(d+2)/d}  & \text{if } d\geq 5.
\end{cases}$
     \item  For $d=1,2,3$,   $p>\max\{d, \frac{2d}{4-d}\}$ and
     $\gamma_p=1-\frac{d}{p}$,
    $\sqrt{n} ( \widehat{\varphi}_n -\varphi )$ converges to a
    nondegenerate centered Gaussian random element in
    $L^{p'}(\Omega)/\langle 1\rangle$, where $p'$ is the conjugate exponent
    of $p$.
    \end{enumerate}
\medskip

The polynomial exponents in the first conclusion are sharp.  For
$d\leq3$, the nondegenerate functional central limit theorem rules out
$o(n^{-1/2})$ convergence.  For $d\geq4$, the empirical potential is
polyhedral with at most $n$ active affine pieces, and our approximation
lower bound shows that a strongly convex potential cannot be approximated
in $L^1$ by such functions faster than $n^{-2/d}$.  Thus, only logarithmic
gaps remain.

Let us compare this conclusion with the rate obtained from the standard Poincaré argument: Poincaré's inequality and the stability estimate for optimal transport maps \cite[Theorem~6]{Manole.et.al.AoS.2024}---note that $\varphi$ is uniformly convex and $\cC^{3,\alpha}$---give
\[
    \E\big[\inf_{a\in\R}\|\widehat{\varphi}_n-\varphi-a\|_{L^1(\mu)}\big]
    \lesssim
    \E[\|\nabla \widehat{\varphi}_n-\nabla \varphi\|_{L^2(\mu)}]
    \lesssim
    \E[\cW_2(\widehat{\nu}_n,\nu)]\lesssim \alpha(n,d)^{1/2}.
\]
Thus, for $d\geq 4$, this gives only the usual transport-map scale, essentially
$n^{-1/d}$, whereas \cref{Thm:rates-and-CLT} gives the sharper potential rate,
essentially $n^{-2/d}$ up to logarithmic factors. Our present result does not contradict the
minimax optimality of the rate for
$\E\|\nabla\widehat\varphi_n-\nabla\varphi\|_{L^2(\mu)}^2$ since our target is at level of potentials rather than transport maps.

The function-space formulation also yields joint root-$n$ inference for
every fixed finite collection of normalization-invariant weighted
contrasts of the potential;
 see \cref{cor:fixed-weighted-contrasts}.  This
includes the regional shadow premium studied in
\cref{sec:applications}.  Because the theorem treats the potential as a
random function rather than a preselected scalar summary, the same limit
can be projected onto new admissible functionals without repeating the
first-order analysis.

Two distinct mechanisms explain the restriction to $d\leq3$.  The dual
$\cC^{1,\mathrm{LogLip}}$ norm of
$\nu-\widehat{\nu}_n$ has parametric order precisely when the corresponding
entropy integral converges, which happens below dimension four; see
\cref{Lemma:Control-of-entropy}.  Separately, the second-order remainder is
negligible at the $n^{-1/2}$ scale only when $d\leq3$. Furthermore, in \Cref{Thm:rates-and-CLT-bootstrap} we prove that the bootstrap is consistent for these cases. This allows for nonparametric inference of the OT potentials in low dimension.

Finally, our techniques shed light on entropy-regularized optimal transport and provide new stability estimates.   As the regularization parameter
$\varepsilon$ vanishes, \cref{Thm:Sinkhorn} gives matching upper and lower
bounds of order $\varepsilon\log(1/\varepsilon)$ for the
normalization-invariant sum of the two entropic dual potentials.

\subsection{Relationship with previous work and literature review}
\label{sec: PreviousWork}

\subsubsection{Potentials as objects of interest.}
The transport value, map, and potential answer different questions.  The
value is one number, and the map records the assignment.  The potential is
the dual optimizer: its gradient generates the quadratic transport map,
and its averages give the shadow values of changes in the marginal
constraints.  Thus, the potential links the optimal value, the assignment,
and the effect of local changes in a distribution.

This role is explicit in several applications.  In matching models,
potentials are equilibrium payoff and price schedules
\citep{ChiapporiMcCannNesheim2010,Galichon2016,GalichonSalanie2022}.  In
multivariate risk, quantiles, and inequality measurement, integrated or
averaged potentials are the reported quantities
\citep{EkelandGalichonHenry2012,BercuBigotThurin2024,HallinMordant2025}.
In machine learning, convex potentials are learned and differentiated to
produce transport maps
\citep{MakkuvaTaghvaeiOhLee2020,PatyDAspremontCuturi2020}.  Potentials also
appear as geopotentials in semi-geostrophic models
\citep{BenamouBrenier1998,BourneEganPelloniWilkinson2022}.

Our limit theorem applies directly when the reference distribution is
fixed and smooth, the target distribution is sampled, and the quantity of
interest is a fixed normalization-invariant weighted average of the
potential.  It therefore gives joint root-$n$ inference, in dimensions
$d\leq3$, for regional shadow premia and local distribution-shift audits,
as developed in \cref{sec:applications}.  It also covers smooth-reference
versions of integrated-quantile and inequality summaries.  For learned
transport potentials, it describes the sampling error of the exact
semidiscrete optimizer; approximation and optimization errors require
separate control.

Cosmological reconstruction gives the potential a direct physical
meaning.  In the idealized real-space Zel'dovich model, an initial uniform
mass field is transported to the present matter distribution and
\[
\Phi(q)=\frac12\|q\|^2-\phi_I(q),
\]
where $\phi_I$ is the primordial gravitational potential
\citep{LevyMohayaeeVonHausegger2021}.   Real surveys also involve biased or weighted tracers, unobserved
mass, survey geometry, and redshift-space effects
\citep{NikakhtarPadmanabhanLevyShethMohayaee2023}.  This makes cosmology a
clear target for extensions of our results.

\subsubsection{Empirical costs and maps.}
The statistical OT literature first concentrated on the optimal value.
Limit laws for empirical transport costs are available in discrete,
semidiscrete, and smooth settings; see
\citet{delBarrioLoubes2019,delBarrioGonzalezSanzLoubes2024,SommerfeldMunk2018} and the review
of \citet{del2025distributional}.  A cost is one scalar and therefore does
not recover the spatially indexed shadow values studied here.  A separate
literature estimates the Brenier map: \citet{HutterRigollet2021} establish
minimax rates, \citet{Manole.et.al.AoS.2024} study
plug-in map estimators, and
\citet{ManoleBalakrishnanNilesWeedWasserman2023clt} prove pointwise limit
theorems for smoothed map estimators on the flat torus.
\citet{CazellesPauwelsPortales2026} construct maps and couplings from
discrete Brenier potentials.  These results target $\nabla\varphi$ or its
induced coupling.  Our loss is instead imposed directly on the smoother
primitive $\varphi$, modulo constants, and our limit is function-valued;
neither result implies the other.

\subsubsection{Stability and convergence rates.}
\citet{DelalandeMerigot2023} control Brenier maps, and consequently
potentials through Poincar\'e's inequality, by a Wasserstein discrepancy.
Extensions cover log-concave sources and John domains
\citep{letrouit2024gluing}, general power costs
\citep{MischlerTrevisan2024}, and squared Riemannian distance
\citep{kitagawa2025stability}.  In the smooth--smooth regime,
\citet{CajaLopezDelgadinoKitagawa2026} obtain Lipschitz stability and
linear response when both marginals vary; that regime excludes the atomic
empirical target considered here.

Applied to an empirical target, the available Wasserstein-based potential
bound gives essentially $n^{-1/d}$ in high dimension.  Our stability
inequality replaces its leading term by the empirical-process norm dual to
$\cC^{1,\mathrm{LogLip}}$ and retains Wasserstein distance only in a
second-order remainder.  It yields $n^{-1/2}$ for $d\leq3$ and
$n^{-2/d}$, up to logarithms, for $d\geq5$.  The polynomial exponents are
sharp: the low-dimensional lower bound follows from our nondegenerate
functional limit, while the high-dimensional lower bound follows from
approximating a strongly convex function by a polyhedral function with at
most $n$ active pieces.  Thus the comparison is not only an improved upper
bound; it identifies the correct polynomial scale for the empirical
potential.

\subsubsection{Limit theory for potentials.}
\citet{delBarrioGonzalezSanzLoubes2024} prove a central limit theorem for
semidiscrete dual weights when the target support is fixed and finite.
Randomness there changes a finite vector of weights.  Here the empirical
target has $n$ random sites, so the number and locations of the Laguerre
cells both change with $n$.  For fixed entropic regularization,
\citet{GonzalezSanzLoubesNilesWeed2022,Goldfeld.et.al.2024.EJS} obtain a functional central limit
theorem for smooth Schr\"odinger potentials, while
\citet{Mordant2024Entropic} lets the regularization decrease and balances
regularization bias against stochastic error in pointwise and $L^2$
limits.  Our theorem instead treats the exactly unregularized
semidiscrete estimator and gives convergence in
$L^{p'}(\Omega)/\langle1\rangle$ for $d\leq3$.  It therefore yields joint
inference for fixed $L^p$-weighted contrasts, but deliberately does not
claim pointwise inference or a limit theorem for the map.

\subsubsection{Entropic regularization.}
Strong $L^1$ convergence without a rate is proved by
\citet{NutzWiesel2022}.
\citet{LopezRivera2026} derives local uniform rates for potentials and
their gradients under convexity assumptions, whereas
\citet{Malamut-Maxime.SIMA.2025} gives sharp asymptotics for regularized
values and plans.  Statistical limit laws for fixed regularization are
available for the value \citep{MenaNilesWeed2019} and the potentials
\citep{GonzalezSanzLoubesNilesWeed2022,Goldfeld.et.al.2024.EJS}.  Our result answers a different
bias question: for the normalization-invariant sum of the two potentials,
\cref{Thm:Sinkhorn} proves matching upper and lower
$L^1$ bounds of order
$\varepsilon\log(1/\varepsilon)$.  The upper bound comes from our
stability inequality and the lower bound from the value asymptotics of
\citet{Malamut-Maxime.SIMA.2025}.

\paragraph{Organization of the paper}
Section~\ref{sec:main-results} states the stability theorem, derives the
rates and functional central limit theorem for the empirical potential as well as the consistency of the bootstrap, and
establishes the bounds for entropy-regularized potentials.
Section~\ref{sec:applications} develops regional shadow-premium inference
for maximal-correlation risk and its distribution-shift interpretation.
Section~\ref{sec: Proofs} contains the proofs of the main results.
Appendices~\ref{sec: Elliptic}--\ref{sec:green-function} collect the
elliptic background and establish the endpoint estimate and Green-function
bounds for the linearized Monge--Amp\`ere equation. Appendix~\ref{appendix-tecnical-lemmas} contains two technical lemmas used in the elliptic analysis:
the embedding
$W^{2,\mathrm{BMO}}(\Omega)\hookrightarrow
\mathcal C^{1,\mathrm{LogLip}}(\Omega)$
and a compactness result for weak-* convergent measures. Appendix~\ref{sect:omitted-proofs} contains the proofs omitted from the main text.

\section{Main results}
\label{sec:main-results}

\subsection{Notation}
Let $\Omega \subset \R^d$ be open with $\mathcal{C}^{2,\alpha}$ boundary. The supremum norm in $\Omega$ is denoted as $|\cdot|_{\infty,\Omega}$.  Following the style of \cite{GilbargTrudinger.Book}, for $k \in \mathbb{N}_0=\{0, 1,\dots\}$ and $0 < \alpha \le 1$, 
we define the Hölder seminorm
\[
[f]_{k,\alpha;\Omega}
:=
\sup_{|\beta| = k}
\sup_{\substack{x,y \in \Omega \\ x \neq y}}
\frac{|D^{\beta} f(x) - D^{\beta} f(y)|}
{|x-y|^{\alpha}},
\]
where $\beta=(\beta_1, \dots, \beta_d) \in \NN_0^d$, $|\beta|=\sum_{i=1}^d \beta_i$ and  $D^{\beta} f=\frac{\partial^{|\beta|} f}
{\partial x_1^{\beta_1} \cdots \partial x_d^{\beta_d}}$ denotes the mixed partial derivative of order $|\beta|$.
The Hölder space $\mathcal{C}^{k,\alpha}(\Omega)$ is
\[
\mathcal{C}^{k,\alpha}(\Omega)
:=
\left\{
f \in C^{k}(\Omega)
:\;
[f]_{k,\alpha;\Omega} < \infty
\right\}.
\]
The associated norm is
$\|f\|_{k,\alpha;\Omega}
:=
\sum_{|\beta| \le k}
\|D^{\beta} f\|_{\infty,\Omega}
+
[f]_{k,\alpha;\Omega}.$
 Define the log-Lipschitz modulus of continuity
 $\omega_{\rm LogLip}(t) = t (1+ |\log(t)|)
 $ and the space $\cC^{1,\mathrm{LogLip}}(\Omega)$ of $\mathcal{C}^1$ functions with $\nabla f$ having finite log-Lipschitz modulus of continuity. Define the norm 
 \[
\|f\|_{1,\mathrm{LogLip},\Omega}= \|f\|_{\infty,\Omega}
+\|\nabla f\|_{\infty,\Omega}+ \sup_{\substack{x,y \in \Omega \\ 0 < \|x-y\| \le 1}}
\frac{\|\nabla f(x)-\nabla f(y)\|}
{\omega_{\rm LogLip}(\|x-y\|)}. 
\]
For a function space $F$, denote its dual by $F^*$. 
Define the dual norms of $\cC^{1,\mathrm{LogLip}}(\Omega)$ and $\cC^{1,\alpha}(\Omega)$ as $$|\gamma|_{1,\mathrm{LogLip},\Omega}'= \sup_{\|f\|_{1,\mathrm{LogLip},\Omega}\leq 1}|\gamma(f)| \quad \text{and}\quad  |\gamma|_{1,\alpha,\Omega}'= \sup_{\|f\|_{1,\alpha,\Omega}\leq 1}|\gamma(f)|.$$ 
We say that  $\mu\in 
\mathcal{P}^{\geq \lambda}(\Omega)$ if $\mu\in \mathcal{P}(\Omega)$ has density (with respect to the Lebesgue measure) lower  bounded by $\lambda$ in $\Omega$. 
$W^{2, p}(\Omega)$ is the space of functions in $L^p(\Omega)$ whose weak derivatives up to order 2 exist and are in $L^p(\Omega)$.
The space $L_0^p(\Omega)$ is the subspace of $L^p(\Omega)$ of functions integrating to 0. We use the relative-ball definition of BMO on $\Omega$.
For $f\in L^1(\Omega)$ and $D_{x,r}:=B_r(x)\cap\Omega$, write
$f_{x,r}=|D_{x,r}|^{-1}\int_{D_{x,r}}f$.  We say that
$f \in {\rm BMO}(\Omega)$ if
\[
\|f\|_{{\rm BMO}}
:= \sup_{\substack{x\in\Omega\\0<r\leq\diam(\Omega)}}
\frac{1}{|D_{x,r}|}\int_{D_{x,r}} |f-f_{x,r}|\,dx
< \infty,
\]
On a bounded $\mathcal C^{2,\alpha}$ domain this is equivalent to the
interior-cube definition and admits a bounded extension to
${\rm BMO}(\mathbb R^d)$; see \citet{Jones1980BMO}.
Define the norm 
$$ \| f \|_{W^{2,BMO}(\Omega)} = \|  f \|_{\mathrm{W}^{2,2}(\Omega)} +  \sum_{i,j}\| \partial_{ij}  f \|_{{\rm BMO}(\Omega)} $$
and the space 
$W^{2,BMO}(\Omega)$ of functions  $f\in {W^{2,2}(\Omega)} $ with finite $\| f \|_{W^{2,BMO}(\Omega)} $. Further, let
\begin{equation}
    \label{eq: XSpace}
\mathcal{X}:=W^{2,BMO}(\Omega)/\langle 1\rangle 
\end{equation}
be the quotient of $W^{2,BMO}(\Omega)$ by the subspace of constant functions. The quotient norm is denoted by $ \|\cdot\|_\mathcal{X}$.

\subsection{Stability of the potential}
Recall that  the notation $\mu\in 
\mathcal{P}^{\geq \lambda}(\Omega)$ means that $\mu\in \mathcal{P}(\Omega)$ has density lower  bounded by $\lambda$ in $\Omega$.

\begin{theorem}[Weak-norm stability of Brenier potentials]
\label{theorem:main}
Let $0<\alpha<1$, let $\Omega\subset\mathbb R^d$ be bounded and
connected with $\mathcal C^{2,\alpha}$ boundary, and fix $R,\kappa>0$.
Let $\varphi:\mathbb R^d\to\mathbb R$ be
$\kappa$-strongly convex and $\mathcal C^{3,\alpha}$ on a neighborhood
of $\overline\Omega$.  Assume that $\mu$ has a
$\mathcal C^{1,\alpha}(\overline\Omega)$ density, also denoted by
$\mu$, bounded below by $\lambda>0$. Then there exists a constant
$C>0$, depending on
$d,\Omega,\alpha,R,\kappa,\lambda,\|\mu\|_{1,\alpha;\Omega}$ and
$\|\varphi\|_{3,\alpha;\Omega}$, such that, for every l.s.c.~convex
function $\widetilde \varphi$ with
$(\partial \widetilde \varphi(\Omega))\cup
(\nabla \varphi(\Omega))\subset \mathbb{B}_R(0)$,
\begin{equation}
    \label{eq:for-L1-main-th}
     \min_a\| \widetilde \varphi -\varphi -a \|_{L^{1}(\mu)} \leq C \left( |\widetilde{\nu} -\nu| _{1,\mathrm{LogLip},\mathbb{B}_R}' +   \cW_2(  \widetilde{\nu}, \nu) \omega_{{\rm LogLip}}\left(\cW_2(  \widetilde{\nu}, \nu) \right)\right) , 
\end{equation}
where $\nu=(\nabla \varphi)_{\# } \mu $ and $\widetilde{\nu}=(\nabla \widetilde\varphi)_{\# } \mu $. Furthermore, for every
$q\in(1,\infty)$ when $d=1$, and every
$q\in(1,d/(d-1))$ when $d\geq2$, there exists a constant $C_q>0$,
depending further on $q$, such that
\begin{equation}
     \label{eq:for-Lp-main-th-stability}
     \min_{a\in \mathbb{R}}
     \| \widetilde{\varphi}- \varphi-a \|_{L^{q}(\mu)}
     \leq
     C_q\left(
     |\widetilde{\nu} -\nu|_{1,\gamma,\mathbb{B}_R}'
     +
     \mathcal{W}_2(\widetilde{\nu}, \nu)^{1+\gamma}
     \right),
\end{equation}
where $\gamma=1-\frac{d(q-1)}{q}$.
\end{theorem}

The proof proceeds by linearizing the push-forward relation around the
reference map $\nabla\varphi$, a strategy also that is also useful to understand the influence function of transport-based quantiles \citet{gonzalez2026influence}. For a smooth test function $g$, a Taylor
expansion gives
\begin{equation}
\label{eq:Taylor-expansion-intro}
(\widetilde\nu-\nu)(g\circ\nabla\varphi^*)
=
\int_\Omega
\left\langle
[\nabla^2\varphi]^{-1}\nabla g,
\nabla(\widetilde\varphi-\varphi)
\right\rangle
\dd\mu
+\mathrm{Rem}(g),
\end{equation}
where
\[
|\mathrm{Rem}(g)|
\lesssim
\|g\|_{W^{2,\mathrm{BMO}}(\Omega)}
W_2(\widetilde\nu,\nu)
\omega_{\mathrm{LogLip}}
\bigl(W_2(\widetilde\nu,\nu)\bigr);
\]
see~\cref{Prop:bound-in-terms-calpha}. To turn this identity into an estimate on
$\widetilde\varphi-\varphi$, fix an arbitrary
$h\in L^\infty_0(\Omega)$ and let $u_h$ solve the linearized Neumann
problem
\begin{equation}
\label{eq:PDE-intro}
\begin{cases}
\operatorname{div}\bigl(
\mu[\nabla^2\varphi]^{-1}\nabla u_h
\bigr)=h
& {\rm in}\ \Omega,\\[0.3em]
\left\langle
\mu[\nabla^2\varphi]^{-1}\nabla u_h,
\mathbf{n}_\Omega
\right\rangle=0
& {\rm on}\ \partial\Omega.
\end{cases}
\end{equation}
By~\cref{theorem-Solvable}, $\|[u_h]\|_{W^{2,\mathrm{BMO}}(\Omega)/\langle1\rangle}
\lesssim
\|h\|_\infty.$ 
The embedding $W^{2,\mathrm{BMO}}(\Omega)
\hookrightarrow
\cC^{1,\mathrm{LogLip}}(\Omega)$ 
allows us to take $g=u_h$ in
\eqref{eq:Taylor-expansion-intro}. Moreover, integration by parts and
the boundary condition in~\eqref{eq:PDE-intro} give
\[
\int_\Omega
\left\langle
[\nabla^2\varphi]^{-1}\nabla u_h,
\nabla(\widetilde\varphi-\varphi)
\right\rangle
\dd\mu
=
-\int_\Omega
h(\widetilde\varphi-\varphi)\,\dd x.
\]
Consequently,
\[
\left|
\int_\Omega
h(\widetilde\varphi-\varphi)\,\dd x
\right|
\lesssim
\|h\|_\infty
\left(
|\widetilde\nu-\nu|_{1,\mathrm{LogLip}}'
+
W_2(\widetilde\nu,\nu)
\omega_{\mathrm{LogLip}}
\bigl(W_2(\widetilde\nu,\nu)\bigr)
\right).
\]
Taking the supremum over $\left\{
h\in L^\infty_0(\Omega):
\|h\|_\infty\leq1
\right\}$ 
yields the desired $L^1(\Omega)$ estimate modulo additive constants.
Since the density of $\mu$ is bounded above and below, this is
equivalent to the corresponding $L^1(\mu)$ estimate.

The $L^q$ estimate follows from the same argument. In this case, one
takes $h\in L^{q'}_0(\Omega)$, where $q'=q/(q-1)>d$, and uses the
standard $W^{2,q'}$ elliptic estimate in place of the endpoint
$W^{2,\mathrm{BMO}}$ estimate.

\subsection{Sample complexity and central limit theorems}
\label{sec:clt}

Let $X_1,\dots,X_n$ be an i.i.d.~sample from
$\nu=(\nabla\varphi)_\#\mu$.  Let
$\widehat{\nu}_n=n^{-1}\sum_{i=1}^n\delta_{X_i}$ and let
$\widehat\varphi_n$ be the Brenier potential from $\mu$ to
$\widehat\nu_n$.  Then it is well known that (see e.g., \cite[Theorem~1.1]{del2025distributional} and \cite[Theorem~6]{Manole.et.al.AoS.2024})
\begin{equation}\label{eq:W2}
     \E[\mathcal{W}_2^2(  \nu ,\widehat{\nu}_n)] \lesssim \alpha(n,d):=\begin{cases}
    n^{-1}  & \text{if } d=1,\\
      n^{-1}\log(n)  & \text{if } d=2,\\
       n^{-\frac{2}{d}}  & \text{if } d\geq 3, 
\end{cases}\qquad \E[\| \nabla \varphi-\nabla \widehat{\varphi}_n  \|_{L^2(\mu)}^2]  \lesssim  \alpha(n,d).
\end{equation}
These map rates are sharp for some distributions
(see, e.g., \citep{JonathanFrancis2019}); in particular, a
parametric $L^2$ limit for the empirical transport map is unavailable in
general.  The potential nevertheless converges faster.  We start with a
complexity estimate; see its proof in~\cref{sec:proof-sample-complexity}.

\begin{lemma}\label{Lemma:Control-of-entropy}
The following estimates hold:
\[
\E\!\left[
|\nu-\widehat\nu_n|_{1,\mathrm{LogLip},\mathbb B_R}'
\right]
\lesssim \beta(n,d):=
\begin{cases}
    n^{-\frac{1}{2}}  & \text{if } d=1,2,3,\\
         n^{-\frac{1}{2}} (\log(n))^{\frac{5}{2}}  & \text{if } d=4,\\
       n^{-2/d}\left( \log (n)\right) ^{(d+2)/d}  & \text{if } d\geq 5.
\end{cases}
\]
\end{lemma}

\begin{lemma}[Empirical CLT in a negative Sobolev space]
\label{lem:empirical-sobolev-clt}
Assume that $d<4$ and let
\[
p>\max\left\{d,\frac{2d}{4-d}\right\}.
\]
Let $D\subset\mathbb R^d$ be bounded with Lipschitz boundary, and let
$\widehat\nu_n$ be the empirical measure associated with an i.i.d.~sample
from a probability measure $\nu$ supported on $D$.
Then
\[
\sqrt n(\nu-\widehat\nu_n)
\rightsquigarrow
\mathbb G_\nu
\qquad\text{in}\qquad
\bigl(W^{2,p}(D)/\langle1\rangle\bigr)^*.
\]
The limit $\mathbb G_\nu$ is a tight centered Gaussian
random element satisfying
\[
\E[\mathbb G_\nu(f)\mathbb G_\nu(g)]
=
\operatorname{Cov}_\nu(f,g),
\]
for every
$f,g\in W^{2,p}(D)/\langle1\rangle$.
\end{lemma}

Denote $\Omega'=\nabla\varphi(\Omega)$. Uniform convexity implies that
$\nabla\varphi:\overline\Omega\to\overline{\Omega'}$ is a
$C^{2,\alpha}$ diffeomorphism with inverse $\nabla\varphi^*$; in
particular, $\Omega'$ has $C^{2,\alpha}$ boundary.
In what follows, $\mathbb G_\nu$ denotes the Gaussian limit furnished by
\cref{lem:empirical-sobolev-clt} with $D=\Omega'$, so that
\[
\mathbb G_\nu
\in
\bigl(W^{2,p}(\Omega')/\langle1\rangle\bigr)^*
\qquad\text{almost surely}.
\] 
Define
\[
\mathcal{T}_{\nabla \varphi}  :
\left(W^{2,p}(\Omega')/\langle 1\rangle\right)^*\to
\left(W^{2,p}(\Omega) / \langle 1\rangle\right)^*
\]
by
$\mathcal{T}_{\nabla \varphi}(\gamma)(f)
=\gamma(f\circ\nabla\varphi^*)$.
The Sobolev chain rule and change of variables give
\[
\|f\circ\nabla\varphi^*\|_{W^{2,p}(\Omega')}
\leq C\|f\|_{W^{2,p}(\Omega)},
\]
so $\mathcal T_{\nabla\varphi}$ is bounded. Finally, define
\[
(\mathcal{L}_A^{-1})^*:
\left(W^{2,p}(\Omega)/\langle 1\rangle\right)^*
\to
L^{p'}(\Omega)/\langle 1\rangle
\]
as the adjoint of the operator $\mathcal{L}_A^{-1}$ mapping
$g\in L^p_0(\Omega)=\{f\in L^p(\Omega):\int_\Omega f=0\}$ to the
unique solution $u\in W^{2,p}(\Omega)/\langle1\rangle$ of
\[
\begin{cases}
{\rm div}\bigl(\mu[\nabla^2\varphi]^{-1}\nabla u\bigr)=g
& {\rm in}\ \Omega,\\
\langle
\mu[\nabla^2\varphi]^{-1}\nabla u,\mathbf{n}_\Omega
\rangle=0
& {\rm on}\ \partial\Omega.
\end{cases}
\]
Here $p'$ denotes the conjugate exponent of $p$, i.e.,
$1=\frac1p+\frac1{p'}$. Then, we arrive at the following result.

\begin{theorem}[Rates and functional limit for empirical potentials]
\label{Thm:rates-and-CLT}
Let $\varphi$ and $\mu$ be as in \cref{theorem:main}. Let
$\widehat{\varphi}_n$ be a convex potential such that
$\nabla\widehat{\varphi}_n$ pushes forward $\mu$ to the empirical
measure $\widehat{\nu}_n$. Then 
\begin{enumerate}
\item $\E\left[
\inf_a
\|\widehat{\varphi}_n-\varphi-a\|_{L^1(\mu)}
\right]
\lesssim
\beta(n,d).$ 
\item For $d=1,2,3$ and
$p>\max\{d,\frac{2d}{4-d}\}$,
$\sqrt n(\widehat{\varphi}_n-\varphi)$ converges in distribution to $(\mathcal L_A^{-1})^*
\bigl(
\mathcal T_{\nabla\varphi}
(\mathbb G_\nu)
\bigr)$ 
in $L^{p'}(\Omega)/\langle1\rangle$.
\end{enumerate}
\end{theorem}

The limit distribution is nondegenerate as the following result shows.

\begin{lemma}[Nondegeneracy of the Gaussian limit]
\label{lem:nondegenerate-limit}
Assume the setting of \Cref{Thm:rates-and-CLT}~(ii) holds. Let
\[
Z=(\mathcal L_A^{-1})^*
\mathcal T_{\nabla\varphi}
(\mathbb G_\nu).
\]
For every nonzero $h\in L^p_0(\Omega)$, with
$u_h=\mathcal L_A^{-1}h$,
\[
\operatorname{Var}\!\left(\int_\Omega hZ\,\rd x\right)
=
\operatorname{Var}_{Y\sim\nu}
\{u_h(\nabla\varphi^*(Y))\}
=
\operatorname{Var}_{X\sim\mu}\{u_h(X)\}>0.
\]
\end{lemma}

\begin{proof}
The two identities follow from the definition of the adjoint and
$(\nabla\varphi^*)_\#\nu=\mu$. If the last variance were zero,
$u_h$ would be constant $\mu$-almost everywhere. The positive
density of $\mu$ and continuity of $u_h$ would make it constant on
the connected domain $\Omega$, contradicting
$\mathcal L_Au_h=h\ne0$.
\end{proof}
\subsubsection{Sharpness of \Cref{Thm:rates-and-CLT}.}

We will now see that the bounds in \Cref{Thm:rates-and-CLT} are sharp up to logarithmic factors. To this end, we begin by proving two auxiliary results.
\begin{lemma}[Affine approximation of a strongly convex function]
\label{lem:local-affine-lower}
Let $K\subset\mathbb R^d$ be a bounded convex Borel set with
$|K|>0$, and let $f:K\to\mathbb R$ be $\kappa$-strongly convex.
Then, for every affine function $\ell$,
\[
\int_K|f(x)-\ell(x)|\,\rd x
\geq c_d\kappa |K|^{1+2/d},
\qquad
c_d=\frac{d\,\omega_d^{-2/d}}
{4(d+2)(1+2^d)},
\]
where $\omega_d$ is the volume of the Euclidean unit ball.
\end{lemma}

\begin{proposition}[Polyhedral lower bound]
\label{prop:polyhedral-lower}
Let $\Omega$ be convex, let the density of $\mu$ be bounded below by
$\lambda>0$, and let $f$ be $\kappa$-strongly convex on $\Omega$.
For every convex polyhedral function
\[
g(x)=\max_{1\leq j\leq N}\{\langle a_j,x\rangle+b_j\},
\qquad N\leq n,
\]
\[
\inf_{c\in\mathbb R}\|g-f-c\|_{L^1(\mu)}
\geq
\lambda c_d\kappa|\Omega|^{1+2/d}n^{-2/d}.
\]
\end{proposition}

\begin{figure}[!t]
    \centering
    \includegraphics[width=\linewidth]{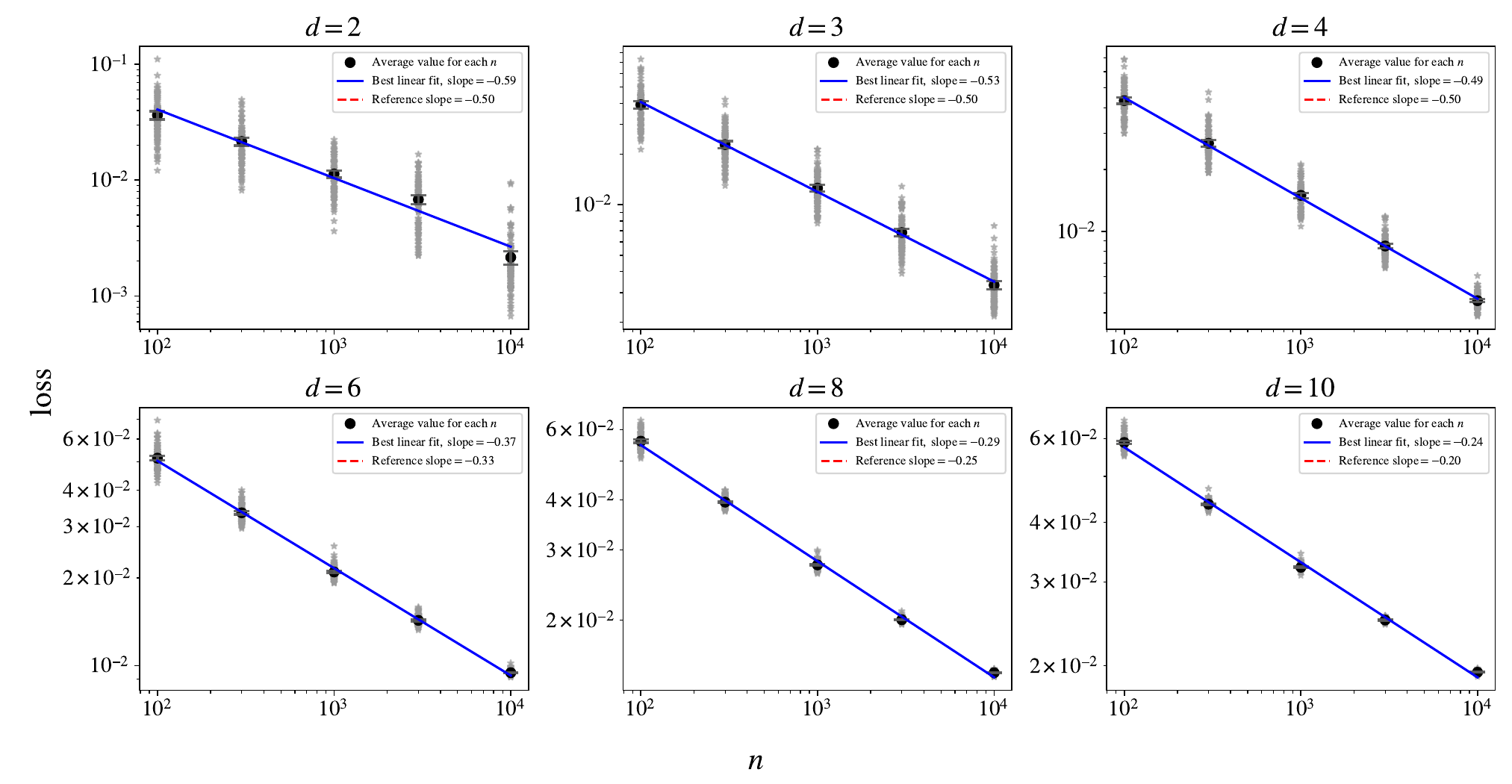}
    \caption{Log-log plot of $\E[\inf_{a\in\R}\|\widehat{\varphi}_n-\varphi-a\|_{L^1(\mu)}]$ for dimensions $d \in \{2,3,4,6,8,10\}$. The black dots report averages over $100$ independent Monte Carlo trials for each $n \in \{100,300,1000,3000, 10000\}$. The semi-discrete Brenier potentials are computed by constructing the associated Laguerre cells, following the procedure described in~\cite[Section~5.2]{peyre2019computational}. The $L^1(\mu)$ integral is approximated by averaging over $10000$ samples from $\mu$.}
    \label{fig:loglog-ot-potential-unit-ball}
\end{figure}

\begin{remark}[Sharpness of the polynomial rates]
\label{rem:sharpness}
For $d\leq3$, choose any nonzero
$h\in L^\infty_0(\Omega)$.  By
\cref{Thm:rates-and-CLT,lem:nondegenerate-limit},
\[
\sqrt n\int_\Omega h(\widehat\varphi_n-\varphi)\,\rd x
\rightsquigarrow N(0,\sigma_h^2),
\qquad \sigma_h^2>0.
\]
Since constants are suppressed by $h$ and the density of $\mu$ is
bounded below, this functional is bounded by a constant times
$\sqrt n\inf_a\|\widehat\varphi_n-\varphi-a\|_{L^1(\mu)}$.
By the continuous mapping theorem (additionally using Fatou's lemma and a Skorokhod representation),
\[
\liminf_{n\to\infty}
\mathbb E\left|
\sqrt n\int_\Omega h(\widehat\varphi_n-\varphi)\,\rd x
\right|
\geq \mathbb E|N(0,\sigma_h^2)|>0.
\]
Consequently,
\[
\liminf_{n\to\infty}\sqrt n\,
\mathbb E\inf_a
\|\widehat\varphi_n-\varphi-a\|_{L^1(\mu)}>0,
\]
so the parametric rate is exact.

For $d\geq5$, take $\Omega$ convex,
$\mu=\nu$ with density bounded below, and
$\varphi(x)=\|x\|^2/2$.  Every potential transporting $\mu$ to an
$n$-atomic measure is the maximum of at most $n$ affine functions.
\Cref{prop:polyhedral-lower} gives, for every sample,
\[
\inf_a\|\widehat\varphi_n-\varphi-a\|_{L^1(\mu)}
\geq c n^{-2/d}.
\]
Thus the exponent $2/d$ is sharp for $d\geq5$.  The same argument at
$d=4$ gives a deterministic $n^{-1/2}$ lower bound, so the polynomial
exponent is sharp there as well.  Dudley-type lower bounds for
$W_2^2(\widehat\nu_n,\nu)$
\citep{Dudley1968,JonathanFrancis2019} suggest the same dimension
threshold, but
they do not logically yield a potential lower bound: the Poincar\'e
comparison is one-sided and cannot be reversed.  The polyhedral
argument supplies the required direct obstruction.  
\end{remark}

\Cref{Thm:rates-and-CLT} gives a potential-level CLT in $L^{p'}(\Omega)/\langle 1\rangle$. This is consistent with the obstruction to an
$H^1$-valued first-order limit in the periodic setting noted in
\cite[Theorem~7]{ManoleBalakrishnanNilesWeedWasserman2023clt}. Indeed, when $d=3$, the condition
$p>\frac{2d}{4-d}=6$ gives $p'=\frac{p}{p-1}<\frac{6}{5}$.
Thus the convergence obtained here is only in the much weaker space
$L^{p'}(\Omega)/\langle1\rangle$, and does not imply convergence in
$H^1$ or any control of the transport maps
$\nabla\widehat{\varphi}_n-\nabla\varphi$ in $L^2(\mu)$.

We further illustrate the rates in~\cref{Thm:rates-and-CLT}~(i) through numerical experiments. 
Obtaining the potentials requires solving a semi-discrete optimal transport problem, which we do via a L-BFGS algorithm. In~\cref{fig:loglog-ot-potential-unit-ball}, we take
$\mu=\nu={\rm Unif}(\mathbb{B}(0,1))$. In this case, the optimal transport map is the identity, and the corresponding Brenier potential is
$\varphi(x)=\frac12\|x\|^2$. Observe that for $d\in \{2,3,4\}$, the observed rate is approximately $n^\delta$ with $\delta \approx -1/2$. In contrast, for $d\in \{6,8,10\}$, the rates exhibit the curse of dimensionality, with $\delta \approx -2/d$.

\subsection{Bootstrap consistency}\label{Sect:bootstrap}

We now show that the usual nonparametric bootstrap consistently
estimates the limiting distribution of the empirical potential. Let
$X_1^*,\dots,X_n^*$ be conditionally i.i.d.~with common distribution
$\widehat\nu_n$, and define
\[
\widehat\nu_n^*
=
\frac1n\sum_{i=1}^n\delta_{X_i^*}.
\]
We first control the distance between $\widehat\nu_n^*$ and $\nu$. 

\begin{lemma}\label{lem:bootstrap-W2}
Let $\widehat\nu_n^*$ be the bootstrap empirical measure of
$\widehat\nu_n$. Then
\[
\E\!\left[
\mathcal W_2^2(\widehat\nu_n^*,\nu)
\right]
\leq
4\E\!\left[
\mathcal W_2^2(\widehat\nu_n,\nu)
\right]
\lesssim
\alpha(n,d).
\]
\end{lemma}

Set $\mathbb G_n^*
=
\sqrt n(\widehat\nu_n^*-\widehat\nu_n).$
We say that $\mathbb G_n^*$ converges conditionally in distribution
to $\mathbb G_\nu$ in $\bigl(W^{2,p}(\Omega')/\langle1\rangle\bigr)^*,$ 
and write $\mathbb G_n^*
\rightsquigarrow_{\mathbb P}
\mathbb G_\nu,$ 
if
\[
\sup_{
H\in{\rm BL}_1(
(W^{2,p}(\Omega')/\langle1\rangle)^*
)
}
\left|
\E^*H(\mathbb G_n^*)
-
\E H(\mathbb G_\nu)
\right|
\overset{\mathbb P}{\longrightarrow}0.
\]
Here ${\rm BL}_1(M)$ denotes, for a metric space $(M,d_M)$, the
collection of functions $H:M\to\mathbb R$ satisfying
\[
\sup_{x\in M}|H(x)|\leq1,
\qquad
|H(x)-H(y)|\leq d_M(x,y),
\quad x,y\in M.
\]
The dual space is endowed with its operator norm. 

\begin{theorem}[Bootstrap consistency]
\label{Thm:rates-and-CLT-bootstrap}
Fix $d\leq3$. Let $\varphi$ and $\mu$ be as in
\cref{theorem:main}. Let $\widehat\varphi_n$
(resp.~$\widehat\varphi_{n,*}$) be a convex potential such that
$\nabla\widehat\varphi_n$
(resp.~$\nabla\widehat\varphi_{n,*}$) pushes forward $\mu$ to
$\widehat\nu_n$ (resp.~to $\widehat\nu_n^*$). Fix
\[
p>\max\left\{d,\frac{2d}{4-d}\right\}.
\]
Then
\[
\sqrt n
\bigl(
\widehat\varphi_{n,*}-\widehat\varphi_n
\bigr)
\rightsquigarrow_{\mathbb P}
(\mathcal L_A^{-1})^*
\bigl(
\mathcal T_{\nabla\varphi}(\mathbb G_\nu)
\bigr)
\qquad\text{in}\qquad
L^{p'}(\Omega)/\langle1\rangle.
\]
\end{theorem}

For applications, the following corollary will be used. 
\begin{corollary}[Fixed weighted contrasts]
\label{cor:fixed-weighted-contrasts} Assume the setting of \Cref{Thm:rates-and-CLT-bootstrap}~(ii) holds. 
 Fix
$J<\infty$ and $h_1,\ldots,h_J\in L^p_0(\Omega)$.   For
\[
\Theta_j([\psi])=\int_\Omega h_j(x)\psi(x)\,\rd x,
\qquad j=1,\ldots,J,
\]
the vector
\[
\sqrt n\big(
\Theta_j([\widehat\varphi_n])-\Theta_j([\varphi])
\big)_{j=1}^J
\]
converges to the centered Gaussian vector
$\big(\int_\Omega h_j Z\,\rd x\big)_{j=1}^J$, where $Z$ is the
limit in \cref{Thm:rates-and-CLT}~(ii).  Furthermore, conditionally on the sample,
\[
\sqrt n
\left(
\Theta_j([\widehat\varphi_{n,*}])
-
\Theta_j([\widehat\varphi_n])
\right)_{j=1}^J
\rightsquigarrow_{\mathbb P}
\left(
\int_\Omega h_j(x)Z(x)\dd x
\right)_{j=1}^J
\qquad\text{in }\mathbb R^J. 
\]
\end{corollary}
\begin{proof}
Each $\Theta_j$ is a continuous linear functional on
$L^{p'}(\Omega)/\langle1\rangle$.  The conclusion follows from the 
continuous  mapping theorem and its bootstrap version 
\citep[Theorem~10.8]{Kosorok.2008.book}. 
\end{proof}

\subsection{Entropy regularized optimal transport}
Let $\mu\in\mathcal P_2(\Omega)$ and
$\nu\in\mathcal P_2(\Omega')$, where
$\Omega'=\nabla\varphi(\Omega)$.  For $\varepsilon>0$, the
entropy--regularized optimal transport (EOT) plan with squared
Euclidean cost is
\[
\pi_\varepsilon
\in
\operatorname*{argmin}_{\pi \in \Pi(\mu,\nu)}
\left\{
\int_{\Omega\times\Omega'}
\frac12 \|x-y\|^2 \, \pi(\rd x,\rd y)
+
\varepsilon
H\big(\pi \mid \mu\otimes\nu\big)
\right\},
\]
where the relative entropy is defined by
\[
H(\pi \mid \mu\otimes\nu)
=
\begin{cases}
\displaystyle
\int_{\Omega\times\Omega'}
\log\!\left(\frac{\dd\pi}{\dd(\mu\otimes\nu)}\right)
\dd\pi,
& \text{if } \pi \ll \mu\otimes\nu, \\[1.2ex]
+\infty,
& \text{otherwise.}
\end{cases}
\]
The dual formulation of the entropy--regularized optimal transport problem is
\begin{multline*}
    \sup_{f,g}
\int_\Omega f(x)\, \mu(\rd x)
+
\int_{\Omega'} g(y)\, \nu(\rd y)
\\-
\varepsilon
\int_{\Omega\times\Omega'}
\exp\!\left(
\frac{f(x)+g(y)-\frac{1}{2}\|x-y\|^2}{\varepsilon}
\right)
\mu(\rd x)\nu(\rd y)
+
\varepsilon. 
\end{multline*}
A pair $(f_\varepsilon,g_\varepsilon)$ attaining this supremum
is called a pair of Schr\"odinger potentials. They are characterized (up to additive constants) by the optimality system
\begin{equation}\label{eq:EOT-system}
    \begin{cases}
\displaystyle
\int_{\Omega'}
\exp\!\left(
\frac{f_\varepsilon(x)+g_\varepsilon(y)-\frac{1}{2}\|x-y\|^2}{\varepsilon}
\right)
\nu(\rd y)
= 1
& \text{for } \mu\text{-a.e. } x, \\[1.2ex]
\displaystyle
\int_\Omega
\exp\!\left(
\frac{f_\varepsilon(x)+g_\varepsilon(y)-\frac{1}{2}\|x-y\|^2}{\varepsilon}
\right)
\mu(\rd x)
= 1
& \text{for } \nu\text{-a.e. } y.
\end{cases}
\end{equation}
Moreover, the optimal plan $\pi_\varepsilon$ is
\[
\frac{\dd\pi_\varepsilon}{\dd \mu\otimes \nu}(x,y)
=
\exp\!\left(
\frac{f_\varepsilon(x)+g_\varepsilon(y)-\frac{1}{2}\|x-y\|^2}{\varepsilon}
\right).
\]
Define the pair
\[
(\varphi_\eps,\psi_\eps)
=\left(\frac12\|\cdot\|^2-f_\eps,
\frac12\|\cdot\|^2-g_\eps\right).
\]
It satisfies
\[
\begin{aligned}
\nabla\varphi_\eps(x)
&=\int y\,\frac{\dd\pi_\varepsilon}{\dd\mu\otimes\nu}(x,y)\,\nu(\rd y),\\
\nabla\psi_\eps(y)
&=\int x\,\frac{\dd\pi_\varepsilon}{\dd\mu\otimes\nu}(x,y)\,\mu(\rd x).
\end{aligned}
\]
Combining \cref{theorem:main} with
\cite[Theorem~3.7]{Malamut-Maxime.SIMA.2025} yields the sharp rate of
convergence of $\varphi_\eps\oplus\psi_\eps$ to
$\varphi\oplus\varphi^*$, where
$(f\oplus g)(x,y)=f(x)+g(y)$.  Recall that
$\varphi^*(x^*)
:=
\sup_{x \in \R^d}
\big\{
\langle x^*, x \rangle - \varphi(x)
\big\}.$

\begin{theorem}[Rates for EOT potentials]\label{Thm:Sinkhorn}
Under the setting of~\cref{theorem:main}, there exists a constant $C$ and $\eps_0>0$ such that, for every $\eps\in (0, \eps_0)$,
    $$ \frac{1}{C}\eps \log\left(1/\eps\right)\leq   \|\varphi_\eps\oplus\psi_\eps - \varphi\oplus \varphi^*\|_{L^1(\mu\otimes \nu)} \leq C \eps \log(1/\eps). $$
\end{theorem}

With these results at hand, we now turn to an application.

\section{Potentials as shadow values: an application to regional risk premia}
\label{sec:applications}

Let $Y\in\mathbb R^d$ be a vector of losses with law $\nu$, and let
$\mu$ be a fixed absolutely continuous baseline distribution.  The
maximal-correlation risk functional of
\citet{EkelandGalichonHenry2012} is
\begin{equation}
\label{eq:maxcorr-risk}
\rho_\mu(\nu)
:=
\sup_{\pi\in\Pi(\mu,\nu)}
\int \langle u,  y \rangle \,\pi(\rd u,\rd y)
=
\inf_{\psi\ {\rm convex}}
\left\{\int\psi\,\rd\mu+\int\psi^*\,\rd\nu\right\}.
\end{equation}
The reference law $\mu$ is part of the definition of the risk
functional, whereas $\nu$ is the law of the observed loss vector.
This is exactly the one-sample setting of
\cref{Thm:rates-and-CLT}.  \citet{EkelandGalichonHenry2012} also give a
general-equilibrium interpretation: $\varphi(u)$ is the indirect
utility of a reference type $u$, $\varphi^*(y)$ is the price of a
loss bundle $y$, and $\nabla\varphi$ is the equilibrium assignment.

The source potential also has a direct interpretation as a marginal
value.  Let $\Omega_a,\Omega_b\subset\Omega$ be disjoint and have positive
$\mu$-probability, and put
\begin{equation}
\label{eq:potential-contrast}
w_{a,b}(u)
:=
\frac{\mathbbm 1_{\Omega_a}(u)}{\mu({\Omega_a})}
-
\frac{\mathbbm 1_{\Omega_b}(u)}{\mu({\Omega_b})},
\qquad
\theta_{a,b}
:=
\int_\Omega\varphi(u)w_{a,b}(u)\,\mu(\rd u).
\end{equation}

For $t\in(-\mu({\Omega_a}),\mu({\Omega_b}))$, define
\[
\mu_t(\rd u)=\{1+t w_{a,b}(u)\}\mu(\rd u).
\]
This perturbation reallocates probability mass $t$ from $\Omega_b$  to $\Omega_a$.

\begin{proposition}[Regional shadow premium]
\label{prop:shadow-premium}
Suppose that the optimizer in \eqref{eq:maxcorr-risk} is unique up to
an additive constant and that normalized dual optimizers for $\mu_t$
converge to the normalized optimizer for $\mu$ as $t\to0$.  Then
\[
\left.\frac{\rd}{\rd t}\rho_{\mu_t}(\nu)\right|_{t=0}
=\theta_{a,b}.
\]
\end{proposition}
Thus, $\theta_{a,b}$ is the marginal change in aggregate
maximal-correlation risk when one infinitesimal unit of reference mass is shifted
from $\Omega_b$  to $\Omega_a$.  We call it a \emph{regional shadow premium}.

This regional contrast measures the first-order sensitivity of maximal-correlation risk to a marginal tilt of the reference distribution from central ranks toward ranks emphasizing jointly large losses.
The
proposition also explains why inference on the potential is not
reducible to inference on the scalar value or on the transport map.
More generally, replacing $w_{a,b}$ by any fixed bounded
zero-$\mu$-mean weight gives a local distribution-shift audit: the
resulting potential contrast is the first-order change in the OT
objective under that prespecified reweighting.  Unlike a global
distance, it identifies which part of the reference population drives
the change.

If $\mu$ has density $m$, the functional in
\eqref{eq:potential-contrast} is represented on
$L^{p'}(\Omega)/\langle1\rangle$ by
\[
g_{a,b}=m\ w_{a,b}\in L^p_0(\Omega).
\]
Set
$\widehat\theta_{a,b}=\int_\Omega\widehat\varphi_nw_{a,b}\,\rd\mu$.
Let $u_{a,b}=\mathcal L_A^{-1}g_{a,b}$ and
$h_{a,b}=u_{a,b}\circ\nabla\varphi^*$.  The continuous mapping
theorem and \cref{Thm:rates-and-CLT} give, for $d\leq3$,
\begin{equation}
\label{eq:contrast-clt}
\sqrt n(\widehat\theta_{a,b}-\theta_{a,b})
=
\frac1{\sqrt n}\sum_{i=1}^n
\left\{\mathbb E h_{a,b}(Y)-h_{a,b}(Y_i)\right\}
+o_{\mathbb P}(1)
\rightsquigarrow N(0,\sigma_{a,b}^2),
\end{equation}
where $\sigma_{a,b}^2
=
{\rm Var}_\nu\{u_{a,b}(\nabla\varphi^*(Y))\}.$ 
Fixed positive-mass regional contrasts therefore admit root-$n$
Gaussian limits.

We illustrate the regional shadow premium with the Loss--ALAE data of
\citet{FreesValdez1998}, distributed with the \texttt{evd} package for the R software.
The data contain 1,500 general-liability claims.  \emph{Loss} is the
indemnity payment and \emph{ALAE} is the allocated loss-adjustment
expense, including investigation and legal costs.  We omit the 34
claims whose indemnity payments were recorded at a policy limit.  For
the remaining $n=1{,}466$ claims, we use the fixed transformation
\[
Y_i=
\left(
\log_{10}\{1+\mathrm{Loss}_i/1000\},
\log_{10}\{1+\mathrm{ALAE}_i/1000\}
\right).
\]

The compact-support, smooth-geometry, and uniform-convexity conditions
of \cref{Thm:rates-and-CLT} are maintained population assumptions in
this illustration.  For the baseline law, let
\[
U=\mathbf 1+Z,\qquad
Z\sim{\rm Unif}\{\mathbb B(0,1)\subset\mathbb R^2\},
\qquad \mathbf 1=(1,1)^\top.
\]
Its support is contained in the positive orthant and
$\mathbb E U=(1,1)^\top$, so the risk functional is componentwise
monotone and has the usual coordinatewise translation normalization.
We display reference ranks in the centered coordinate
$z=u-\mathbf 1$.  In that coordinate, define
\[
\begin{aligned}
\Omega_a^0&=
\left\{z:\|z\|\geq0.70,\
\frac{\pi}{8}\leq\arg(z)\leq\frac{3\pi}{8}\right\},
&
\Omega_b^0 &=\{z:\|z\|\leq0.35\},\\
\Omega_a&=\mathbf 1+\Omega_a^0,
&
\Omega_b&=\mathbf 1+\Omega_b^0.
\end{aligned}
\]
{Region $\Omega_a$ is the outer northeast reference sector.  Under the
empirical transport, its source ranks are assigned to claims with
jointly large indemnity and adjustment costs;} $\Omega_b$ is the central reference region.  Their
baseline probabilities are
$\mu(\Omega_a)=0.06375$ and $\mu(\Omega_b)=0.1225$.

Translating the source disk does not change the optimal assignment,
the regional contrast, or the displayed centered potential.  Solving
the unregularized semidiscrete problem gives
\[
\widehat\theta_{a,b}=1.529.
\]
On the transformed loss scale, shifting one unit of baseline mass
from the central region to the outer joint-loss sector therefore has
an estimated marginal value of $1.529$.  This is an indirect-utility
or shadow-value contrast; it cannot be thought as premium expressed in dollars.    In
contrast, the scalar $\rho_\mu(\widehat\nu_n)$ aggregates the whole
loss distribution and does not reveal which reference region carries
the marginal value.

{Let $\widehat\theta_n^*$ denote the same contrast recomputed after
drawing $n$ claims with replacement from $\widehat\nu_n$.  Since
$g_{a,b}\in L_0^p(\Omega)$, \cref{cor:fixed-weighted-contrasts} gives
\[
 \sqrt n(\widehat\theta_n^*-\widehat\theta_n)
 \rightsquigarrow_{\mathbb P}N(0,\sigma_{a,b}^2).
\]
Thus the usual $n$-out-of-$n$ bootstrap applies directly.  With
$B=1{,}000$ bootstrap samples, its estimated standard error is $0.0192$;
the basic 95\% bootstrap interval is $[1.4906,1.5664]$.  The corresponding
normal interval is $[1.4910,1.5662]$.  Numerical details and sensitivity
checks are given in Appendix~\ref{app:risk-application}.}

\begin{figure}[!t]
\centering
\includegraphics[width=\linewidth]{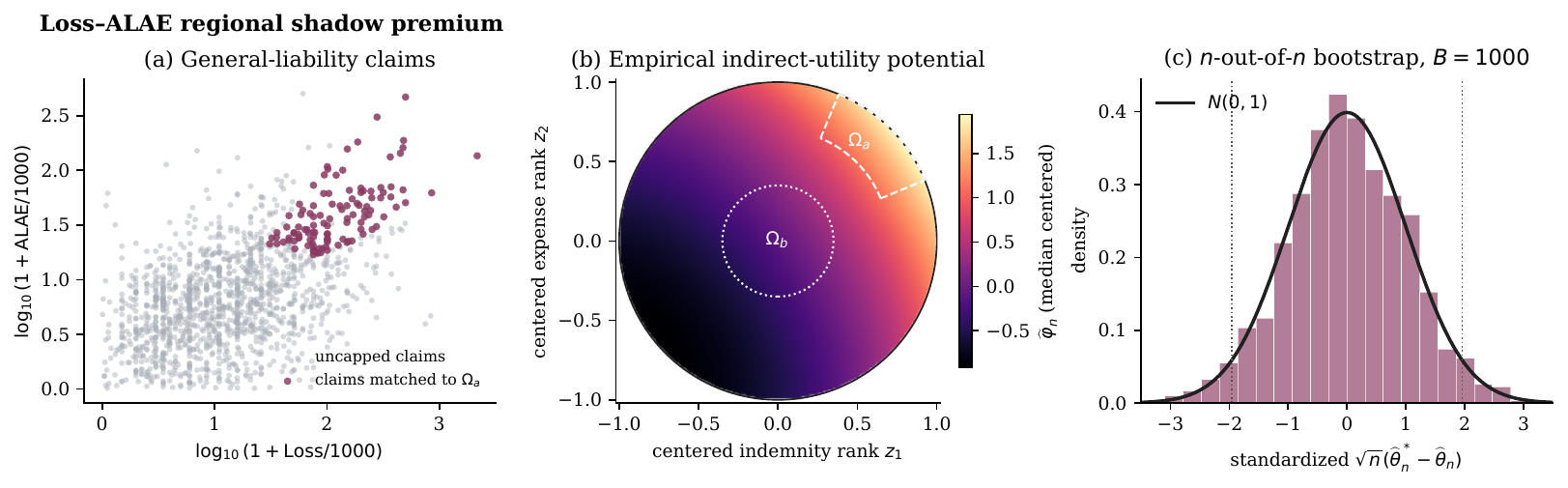}
\caption{Regional shadow premia for the Loss--ALAE claims.
(a) Log-transformed indemnity payments and allocated loss-adjustment
expenses for the 1,466 uncapped claims; {highlighted claims have empirical
Laguerre cells intersecting the outer joint-loss reference sector
$\Omega_a$.}  (b) The median-centered empirical Brenier potential,
displayed in the centered reference coordinate $z=u-\mathbf 1$;
$\Omega_b$  is the central disk.  {(c) Conditional $n$-out-of-$n$ bootstrap
distribution of $\sqrt n(\widehat\theta_n^*-\widehat\theta_n)$, based on
$B=1{,}000$ draws and standardized by its bootstrap standard deviation,
together with the standard normal density.}}
\label{fig:application-risk}
\end{figure}
\medskip

\section{Proofs of the main results}
\label{sec: Proofs}

\subsection{Proof of \cref{theorem:main}}
\begin{proposition}\label{Prop:bound-in-terms-calpha} Under the setting of~\cref{theorem:main},  the following hold.
  
    \begin{enumerate}
    \item There exists a constant $C$ such that, for every $g\in W^{2,BMO}(\Omega)$, 
\begin{multline}\label{eq:development-BMO}
      \left|(\widetilde{\nu} - \nu)(g\circ \nabla \varphi^*)-   \int   \langle [\nabla^2 \varphi]^{-1} \nabla g , \nabla ( \widetilde \varphi - \varphi)\rangle \dd \mu\right| \\
     \leq C\cdot \|g\|_{W^{2,BMO}(\Omega)} \cW_2(  \widetilde{\nu}, \nu)  {\omega}_{{\rm LogLip}}(\cW_2(  \widetilde{\nu}, \nu) ), 
\end{multline}
where we recall that ${\omega}_{{\rm LogLip}}(t) = t(1+|\log(t)|)$.
\item  For every $p>d$,  there exists a constant $C=C(p)$ such that, for every function $g\in \mathcal{W}^{2,p}(\Omega)$,
$$ \left|(\widetilde{\nu} - \nu)(g\circ \nabla \varphi^*)-   \int   \langle [\nabla^2 \varphi]^{-1} \nabla g , \nabla ( \widetilde \varphi - \varphi)\rangle \dd \mu\right| 
     \leq C\cdot  \|g\|_{\mathcal{W}^{2,p}(\Omega)} (\cW_2(  \widetilde{\nu}, \nu) )^{1+\alpha} , $$
where $\alpha=1-\frac{d}{p}$. 
\end{enumerate}
\end{proposition}
\begin{proof}  
Recall that $\nabla \varphi(\Omega) \cup \nabla \widetilde \varphi(\Omega) \subset \BB_R$. Let $f \in \cC^{1}(\BB_R)$. We have
\[
(\widetilde{\nu} - \nu)(f) = \int  f \dd ( (\nabla \widetilde\varphi)_{\# } \mu - (\nabla  \varphi)_{\# }  \mu) 
= \int \bigl(f(\nabla \widetilde \varphi)-f(\nabla \varphi)\bigr) \dd\mu.
\]
Therefore, by the mean-value theorem,
\begin{equation}\label{eq:tylor}
    \begin{aligned}
 &(\widetilde{\nu} - \nu)(f)  -  \int  \langle \nabla f(\nabla \varphi), \nabla ( \widetilde \varphi - \varphi)\rangle \dd \mu \\
 =& \int \int_0^1 \inner{\nabla f(\nabla \varphi + t\nabla (\widetilde \varphi - \varphi)) -\nabla f(\nabla \varphi)}{\nabla (\widetilde \varphi - \varphi)} \dd t \dd \mu.
\end{aligned}
\end{equation}

\textit{Proof of (i).} Assume further that  $f \in \cC^{1,{\rm LogLip}}(\BB_R)$. By definition and as $t(1+|\log(t)|) \leq 1$ for $t\in [0,1]$, \eqref{eq:tylor} is upper bounded by
\begin{equation}\label{eq:loglip}
     \|f\|_{1,{\rm LogLip}, \BB_R}\int \|\nabla \widetilde \varphi-\nabla \varphi\| \omega_{{\rm LogLip}}(\|\nabla \widetilde \varphi-\nabla \varphi\|)  \dd\mu.
\end{equation}
It remains to show that 
\begin{equation}\label{eq:Jensen-LogLip}
     \int \|\nabla \widetilde \varphi-\nabla \varphi\| \omega_{{\rm LogLip}}(\|\nabla \widetilde \varphi-\nabla \varphi\|)  \dd\mu\lesssim   \|\nabla \widetilde \varphi-\nabla \varphi\|_{L^2(\mu)}  {\omega}_{{\rm LogLip}}(\|\nabla \widetilde \varphi-\nabla \varphi\|_{L^2(\mu)}). 
\end{equation}
Indeed, denote by  $\Omega':=\nabla \varphi(\Omega)$ and take $f=g\circ \nabla \varphi^* : \Omega' \to \R $. Then, $\nabla f(\nabla \varphi)= \nabla^2 \varphi^*(\nabla \varphi) \nabla g = [\nabla^2\varphi]^{-1} \nabla g$. Moreover, by~\cite[Theorem~2]{Fefferman2009Extension}, $f$ admits a $\cC^{1,\mathrm{LogLip}}$ extension to $\BB_R$ with  $\|f\|_{1,{\rm LogLip},\mathbb{B}_R(0)}\lesssim \|f\|_{1,{\rm LogLip},\Omega'} $. In addition, since $\Omega'$ is a $\cC^{2,\alpha}$ domain, we conclude that $\|f\|_{1,{\rm LogLip},\mathbb{B}_R(0)} \lesssim \|g\|_{1,{\rm LogLip},\Omega} \lesssim  \|g\|_{W^{2,BMO}(\Omega)}$.
Here, the second-last inequality follows by the $\mathcal{C}^{3,\alpha}$ regularity of $\varphi^*$ and the last inequality holds because $W^{2,BMO}(\Omega) \hookrightarrow  \cC^{1,\mathrm{LogLip}}(\Omega)$ by~\cref{lemma:BMO-inclusion}. Finally, we recall from \cite{Manole.et.al.AoS.2024} that 
\begin{equation}
    \label{eq:TudorsBound}
    \|\nabla \widetilde \varphi-\nabla \varphi\|_{L^2(\mu)}  \lesssim  \cW_2(  \widetilde{\nu}, \nu).  
\end{equation}
Then, combining with~\eqref{eq:tylor} and~\eqref{eq:Jensen-LogLip}, we conclude that
\begin{align*}
    &\left|(\widetilde{\nu} - \nu)(f)  -  \int  \langle \nabla f(\nabla \varphi), \nabla ( \widetilde \varphi - \varphi)\rangle \dd \mu \right|\\
    \lesssim &  \|f\|_{1,{\rm LogLip}, \BB_R} \|\nabla \widetilde \varphi-\nabla \varphi\|_{L^2(\mu)}  {\omega}_{{\rm LogLip}}(\|\nabla \widetilde \varphi-\nabla \varphi\|_{L^2(\mu)})\\
     \lesssim &\|g\|_{W^{2,BMO}(\Omega)} \cW_2(  \widetilde{\nu}, \nu)  {\omega}_{{\rm LogLip}}(\cW_2(  \widetilde{\nu}, \nu)).
\end{align*}
We will now show~\eqref{eq:Jensen-LogLip}. Call $\xi=\widetilde{\varphi}-\varphi$ and set
$
    Z:= \|\nabla\xi\|.
$
Since $\nabla\varphi(\Omega)\cup\nabla\widetilde\varphi(\Omega)\subset\BB_R$, we have
$0\leq \|\nabla\xi\|\leq 2R$. We use the following elementary estimate: if $0\leq \|\nabla\xi\|\leq M$ and
$m:=\int Z^2\dd\mu$, then
\begin{equation}\label{eq:truncated-log-moment}
    \int \|\nabla\xi\|^2(1+|\log \|\nabla\xi\||)\dd\mu
    \leq C_M m\bigl(1+|\log\sqrt{m}|\bigr).
\end{equation}
The assertion is trivial when $m=0$. Otherwise, set $A:=\{\|\nabla\xi\|\leq1\}$ and
$m_0:=\int_A \|\nabla\xi\|^2\dd\mu$. The function
\[
    H(s):=s\left(1-\frac12\log s\right),\qquad s\in(0,1],
\]
extended by $H(0)=0$, is increasing and concave. Thus, if $\mu(A)>0$, conditional
Jensen's inequality gives
\[
    \int_A \|\nabla\xi\|^2(1+|\log \|\nabla\xi\||)\dd\mu
    =\int_A H(\|\nabla\xi\|^2)\dd\mu
    \leq \mu(A)H\left(\frac{m_0}{\mu(A)}\right)
    \leq H(m_0).
\]
The same estimate is immediate when $\mu(A)=0$. If $m\leq1$, it follows that
\[
    \int_A \|\nabla\xi\|^2(1+|\log \|\nabla\xi\||)\dd\mu
    \leq H(m)
    =m\bigl(1+|\log\sqrt m|\bigr).
\]
Moreover,
\[
    \int_{A^c}\|\nabla\xi\|^2(1+|\log \|\nabla\xi\||)\dd\mu
    \leq \bigl(1+\max\{0,\log M\}\bigr)m.
\]
Since $m\leq m(1+|\log\sqrt m|)$ when $m\leq1$, this proves
\eqref{eq:truncated-log-moment} in this case. If $m>1$, then
\[
    \int_A H(\|\nabla\xi\|^2)\dd\mu\leq1\leq m,
\]
while the same estimate on $A^c$ applies. Since
$m\leq m(1+|\log\sqrt m|)$, \eqref{eq:truncated-log-moment} follows in all cases.

Applying~\eqref{eq:truncated-log-moment} with $M=2R$, we obtain
\begin{align*}
    \int \|\nabla\xi\|\omega_{{\rm LogLip}}(\|\nabla\xi\|)\dd\mu
    &=\int Z^2(1+|\log Z|)\dd\mu\\
    &\lesssim m\bigl(1+|\log\sqrt m|\bigr)=\|\nabla\xi\|_{L^2(\mu)}
      \omega_{{\rm LogLip}}\bigl(\|\nabla\xi\|_{L^2(\mu)}\bigr),
\end{align*}
which proves~\eqref{eq:Jensen-LogLip}. 

\textit{Proof of (ii).} The proof follows from an argument similar to that for part~(i) after replacing $\|f\|_{1,{\rm LogLip}, \BB_R}$ in~\eqref{eq:loglip} by $\|f\|_{1,\alpha, \BB_R}$ and noting that $\mathcal{W}^{2,p}(\Omega) \hookrightarrow  \cC^{1,\alpha}(\Omega)$ for $p>d$ and $\alpha=1-\frac{d}{p}$ by the Sobolev embedding theorem. In addition, an analogue of~\eqref{eq:Jensen-LogLip} can be derived directly from Jensen's inequality since the function $t\mapsto t^{\frac{1+\alpha}{2}}$ is concave.
\end{proof}

We are now ready to prove~\cref{theorem:main}. 
\begin{proof}[Proof of \cref{theorem:main}]
We will present the proof for \eqref{eq:for-L1-main-th} as \eqref{eq:for-Lp-main-th-stability} follows via a similar argument and~\cref{Prop:bound-in-terms-calpha}~(ii). Note that \cref{Prop:bound-in-terms-calpha}~(i) yields that, for every $g\in W^{2,BMO}(\Omega)$,  
\begin{multline}\label{eq:development-BMO-2} 
     \left|  \int   \langle [\nabla^2 \varphi]^{-1} \nabla g , \nabla ( \widetilde \varphi  -\varphi)\rangle \dd \mu\right| \\
     \lesssim \|g\|_{W^{2,BMO}(\Omega)} \left(  |\widetilde{\nu} -\nu| _{1,\mathrm{LogLip},\mathbb{B}_R}'+ \cW_2(  \widetilde{\nu}, \nu)  {\omega}_{{\rm LogLip}}(\cW_2(  \widetilde{\nu}, \nu) ) \right).
\end{multline}
By \cref{theorem-Solvable}, for any $h\in L^\infty_0(\Omega)$,  there exists a unique solution $u\in  \bigslant{W^{2,BMO}(\Omega)}{\langle 1\rangle}$ to $$\begin{cases}
    {\rm div}([ \nabla^2 \varphi]^{-1} \mu \nabla u)= h  & \text{in } \Omega,\\
     \langle \nabla u, [ \nabla^2 \varphi]^{-1} \nu_{\Omega} \rangle=0& \text{on } \partial \Omega,
\end{cases} $$
and $\|u\|_{W^{2,BMO}(\Omega)} \lesssim \|h\|_{\infty,\Omega}$.  Hence, take $g=u$ in \eqref{eq:development-BMO-2}, there exists a constant $C$ such that for any $h\in L^\infty_0(\Omega)$, 
\begin{align*}
     C\, \|h\|_{\infty,\Omega}  &\left(  |\widetilde{\nu} -\nu| _{1,\mathrm{LogLip},\mathbb{B}_R}'+ \cW_2(  \widetilde{\nu}, \nu)  {\omega}_{{\rm LogLip}}(\cW_2(  \widetilde{\nu}, \nu) ) \right)\\
     &\geq  \left|  \int   \langle [\nabla^2 \varphi]^{-1} \nabla u , \nabla ( \widetilde \varphi  -\varphi)\rangle \dd \mu\right| =  \left|  \int_\Omega h ( \widetilde \varphi  -\varphi) \dd x \right|.
\end{align*}
To conclude the proof, taking the supremum over $\{ h\in L^\infty_0(\Omega): \|h\|_{\infty,\Omega}\leq 1\}$ and noting that
\[
\sup_{h\in L^\infty_0(\Omega),\, 
\|h\|_{\infty,\Omega}\leq 1}\left|  \int_\Omega h ( \widetilde \varphi  -\varphi) \dd x \right| =  \min_{a\in \R} \|\widetilde\varphi - \varphi -a\|_{L^1(\Omega)}.
\]
The desired result follows since $\mu\in \cP^{\geq \lambda}(\Omega)\cap \cC^{1,\alpha}(\Omega)$ so that $L^1(\mu)$ and $L^1(\Omega)$ are equivalent norms on $\Omega$.
\end{proof}

\subsection{Proof of~\cref{Thm:rates-and-CLT}}\label{sec:proof-sample-complexity}

\cref{theorem:main}, \eqref{eq:W2} and
\cref{Lemma:Control-of-entropy} imply the first claim
once the logarithmic remainder is controlled. Since both
$\widehat\nu_n$ and $\nu$ are supported in $\mathbb B_R$, we have
$0\leq\cW_2(\widehat\nu_n,\nu)\leq2R$. Applying
\eqref{eq:truncated-log-moment} gives
\[
\begin{aligned}
&\E\left[
\cW_2^2(\widehat\nu_n,\nu)
\left(
1+
\left|
\log\cW_2(\widehat\nu_n,\nu)
\right|
\right)
\right]
\\
&\qquad\lesssim
\E\left[
\cW_2^2(\widehat\nu_n,\nu)
\right]
\left(
1+
\left|
\log
\sqrt{
\E\left[
\cW_2^2(\widehat\nu_n,\nu)
\right]
}
\right|
\right)
\\
&\qquad\lesssim
\alpha(n,d)(1+\log n)
\lesssim
\beta(n,d).
\end{aligned}
\]
Note that
$\E[\cW_2^2(\widehat\nu_n,\nu)]>0$ and the last inequality
follows by considering separately the cases in \eqref{eq:W2}. Combining
this estimate with \cref{Lemma:Control-of-entropy} proves the first
claim.

We now prove the second claim. Recall that we use the same notation
for $\mu$ and its density. From \cref{theorem:solvable-lp}, we know
that for every $g\in L^p_0(\Omega)$, there exists a unique solution
$u\in W^{2,p}(\Omega)/\langle 1\rangle$ to
\[
\begin{cases}
{\rm div}\bigl(\mu[\nabla^2\varphi]^{-1}\nabla u\bigr)=g
&\text{in }\Omega,\\
\left\langle
\mu[\nabla^2\varphi]^{-1}\nabla u,\nu_{\Omega}
\right\rangle=0
&\text{on }\partial\Omega.
\end{cases}
\]
Furthermore, the operator $\mathcal{L}_A^{-1}$ with
$A=\mu[\nabla^2\varphi]^{-1}$ is bounded so that the adjoint
\[
(\mathcal{L}_A^{-1})^*:
\left(W^{2,p}(\Omega)/\langle1\rangle\right)^*
\to L^{p'}(\Omega)/\langle1\rangle
\]
is also bounded. Recall that the adjoint satisfies
\[
\int_\Omega g(\mathcal{L}_A^{-1})^*(\gamma)\dd x
=
\gamma(\mathcal{L}_A^{-1}(g)),
\quad
\text{for all }
\gamma\in
\left(W^{2,p}(\Omega)/\langle1\rangle\right)^*,
\quad
g\in L^p_0(\Omega),
\]
and the bounded operator
\begin{align*}
\mathcal{T}_{\nabla\varphi}:
\left(W^{2,p}(\Omega')/\langle1\rangle\right)^*
&\to
\left(W^{2,p}(\Omega)/\langle1\rangle\right)^*,
\\
\gamma
&\mapsto
\left(
f\mapsto
\mathcal{T}_{\nabla\varphi}(\gamma)(f)
=
\gamma(f\circ\nabla\varphi^*)
\right).
\end{align*}
Therefore, \cref{Prop:bound-in-terms-calpha}~(ii) implies the
following: for every $g\in L^p_0(\Omega)$,
\begin{align*}
C&\cdot\|g\|_{L^p(\Omega)}
\bigl(\cW_2(\widehat{\nu}_n,\nu)\bigr)^{1+\alpha}
\\
&\geq
C\cdot
\|\mathcal{L}_A^{-1}(g)\|_{W^{2,p}(\Omega)}
\bigl(\cW_2(\widehat{\nu}_n,\nu)\bigr)^{1+\alpha}
\\
&\geq
\left|
(\widehat{\nu}_n-\nu)
\bigl(\cL_A^{-1}(g)\circ\nabla\varphi^*\bigr)
-
\int
\left\langle
[\nabla^2\varphi]^{-1}
\nabla(\mathcal{L}_A^{-1}(g)),
\nabla(\widehat{\varphi}_n-\varphi)
\right\rangle
\dd\mu
\right|
\\
&=
\left|
\cT_{\nabla\varphi}(\widehat{\nu}_n-\nu)
\bigl(\cL_A^{-1}(g)\bigr)
+
\int_\Omega g(\widehat{\varphi}_n-\varphi)\dd x
\right|
\\
&=
\left|
\int_\Omega
g\left\{
(\mathcal{L}_A^{-1})^*
\mathcal{T}_{\nabla\varphi}(\widehat{\nu}_n-\nu)
+\widehat{\varphi}_n-\varphi
\right\}
\dd x
\right|.
\end{align*}
Here, the second-to-last line follows from integration by parts
and the last one from the definition of the adjoint.
Taking the supremum on
$\{g\in L^p_0(\Omega):\|g\|_{L^p(\Omega)}\leq1\}$ for both sides
yields a constant $C=C(p)$ such that
\[
\left\|
(\mathcal{L}_A^{-1})^*
\mathcal{T}_{\nabla\varphi}(\widehat{\nu}_n-\nu)
+\widehat{\varphi}_n-\varphi
\right\|_{L^{p'}(\Omega)/\langle1\rangle}
\leq
C\bigl(\cW_2(\widehat{\nu}_n,\nu)\bigr)^{1+\alpha},
\]
for
$\alpha=1-\frac{(p'-1)d}{p'}=1-\frac{d}{p}$. Since
\[
\E[\cW_2^2(\widehat{\nu}_n,\nu)]
\lesssim
\alpha(n,d)
=
\begin{cases}
n^{-1} & \text{if }d=1,\\
n^{-1}\log(n) & \text{if }d=2,\\
n^{-\frac{2}{d}} & \text{if }d\geq3,
\end{cases}
\]
using Jensen's inequality and the concavity of the map
$t\mapsto t^{1-\frac{d}{2p}}$, we can derive that
\[
\sqrt{n}\,
\E\left[
\cW_2(\widehat{\nu}_n,\nu)^{2-\frac{d}{p}}
\right]
\longrightarrow0
\]
for $d=1,2,3$ and any
$p>\max\{d,\frac{2d}{4-d}\}$ as $n\to\infty$, and thus
\[
\sqrt{n}\,
\cW_2(\widehat{\nu}_n,\nu)^{2-\frac{d}{p}}
=o_{\PP}(1)
\]
by Markov's inequality.
Since $1+\alpha=2-\frac{d}{p}$, the previous estimate gives
\[
\sqrt{n}(\widehat{\varphi}_n-\varphi)
=
(\mathcal{L}_A^{-1})^*
\mathcal{T}_{\nabla\varphi}
\bigl(\sqrt{n}(\nu-\widehat{\nu}_n)\bigr)
+o_{\PP}(1)
\]
in $L^{p'}(\Omega)/\langle1\rangle$. By
\cref{lem:empirical-sobolev-clt},
\[
\sqrt{n}(\nu-\widehat{\nu}_n)
\rightsquigarrow
\mathbb{G}_\nu
\quad\text{in}\quad
\left(W^{2,p}(\Omega')/\langle1\rangle\right)^*.
\]
The boundedness of $\mathcal{T}_{\nabla\varphi}$ and
$(\mathcal{L}_A^{-1})^*$, together with the continuous mapping
theorem, yields
\[
(\mathcal{L}_A^{-1})^*
\mathcal{T}_{\nabla\varphi}
\bigl(\sqrt{n}(\nu-\widehat{\nu}_n)\bigr)
\rightsquigarrow
(\mathcal{L}_A^{-1})^*
\mathcal{T}_{\nabla\varphi}
(\mathbb{G}_\nu)
\]
in $L^{p'}(\Omega)/\langle1\rangle$. The result now follows from
Slutsky's lemma.
\qed

\subsection{Proof of \cref{Thm:Sinkhorn}}\label{Sect:proof-EOT}
Let $g\in W^{2,\mathrm{BMO}}(\Omega)$. 
Since $\pi_\eps\in \Pi(\mu,\nu)$, we have 
$0= \int (g(x)-g(\nabla \varphi^*(y)))\pi_\eps(\rd x,\rd y)$, and a Taylor development of $g\circ \nabla \varphi^*$ around $\nabla \varphi(x)$ yields  
\begin{multline*}
    \left\vert \int \langle [\nabla^2 \varphi(x)]^{-1} \nabla g(x),
    \nabla \varphi(x) - y \rangle \pi_\eps(\rd x,\rd y) \right\vert \\
    \leq C \cdot \|g\|_{1,{\rm LogLip},\Omega}
    \int \|\nabla \varphi(x)-y\|
    \omega_{{\rm LogLip}}(\|\nabla \varphi(x)-y\|)\pi_\eps(\rd x,\rd y),
\end{multline*}
for some constant $C$ independent of $g$. Arguing as in the proof of \cref{Prop:bound-in-terms-calpha}, we get 
\begin{multline*}
    \left\vert \int \langle [\nabla^2 \varphi(x)]^{-1} \nabla g(x),
    \nabla \varphi(x) - y \rangle \pi_\eps(\rd x,\rd y) \right\vert \\
    \leq C \cdot \|g\|_{1,{\rm LogLip},\Omega}
    \|\nabla \varphi(X)-Y\|_{L^2(\mathbb{P})}
    \omega_{{\rm LogLip}}
    (\|\nabla \varphi(X)-Y\|_{L^2(\mathbb{P})}),
\end{multline*}
where $(X,Y)\sim \pi_\eps$ are random variables defined on a common probability space. As $\mu,\nu$ are supported on compact sets and have densities that are bounded from above and below, Remark~3.10 in \cite{Malamut-Maxime.SIMA.2025} shows the estimate
$\|\nabla \varphi(X)-Y\|_{L^2(\mathbb{P})}\lesssim \sqrt{\eps}$ for $\eps$ small.
As a consequence,
\[
\left\vert
\int \langle [\nabla^2 \varphi(x)]^{-1}\nabla g(x),
\nabla \varphi(x)-y\rangle\pi_\eps(\rd x,\rd y)
\right\vert
\leq
C\cdot\|g\|_{1,{\rm LogLip},\Omega}
\cdot\eps\log(1/\eps).
\]
Since
$\nabla \varphi_\eps(x)
=\int y \frac{\dd\pi_\varepsilon}{\dd\mu\otimes\nu}(x,y)\dd\nu(y)$,
we have
\begin{align*}
&\left\vert
\int
\left\langle
[\nabla^2\varphi(x)]^{-1}\nabla g(x),
\nabla\varphi(x)-\nabla\varphi_\eps(x)
\right\rangle
\mu(\rd x)
\right\vert
\\
&\qquad=
\left\vert
\int
\left\langle
[\nabla^2\varphi(x)]^{-1}\nabla g(x),
\nabla\varphi(x)-y
\right\rangle
\pi_\eps(\rd x,\rd y)
\right\vert
\\
&\qquad\leq
C\cdot\|g\|_{1,{\rm LogLip},\Omega}
\cdot\eps\log(1/\eps).
\end{align*}
This allows to conclude that
\[
\min_{a\in\R}
\|\varphi-\varphi_\eps-a\|_{L^1(\mu)}
\lesssim
\eps\log(1/\eps), 
\]
by imitating the proof of \cref{theorem:main}. Naturally, owing to symmetry, the other dual potential satisfies the same estimate
\begin{equation}
\label{eq:estimates-EOT-pot}
\min_{a\in\R}
\|\varphi^*-\psi_\eps-a\|_{L^1(\nu)}
\lesssim
\eps\log(1/\eps).
\end{equation} 
We now relate the sum of the potentials to the regularized value. Let
${\rm EOT}_\eps$ be the optimal value of the entropy-regularized
problem with entropy relative to $\mu\otimes\nu$, and let ${\rm OT}$ be
the unregularized optimal value. By the Schr\"odinger normalization \eqref{eq:EOT-system},
\[
{\rm EOT}_\eps
=
\int f_\eps\dd\mu+\int g_\eps\dd\nu.
\]
On the other hand, Kantorovich duality gives
\[
{\rm OT}
=
\int\left(\frac12\|x\|^2-\varphi(x)\right)\mu(\rd x)
+
\int\left(\frac12\|y\|^2-\varphi^*(y)\right)\nu(\rd y).
\]
Consequently, 
\begin{equation}\label{eq:value-potential-identity}
\int
\bigl[
(\varphi_\eps-\varphi)
\oplus
(\psi_\eps-\varphi^*)
\bigr]
\dd(\mu\otimes\nu)
=
-\bigl({\rm EOT}_\eps-{\rm OT}\bigr).
\end{equation}
This identity is independent of the admissible additive
normalizations of the potentials. 
We now compare our normalization with the one used in
\cite[Theorem~1]{pal2019difference}. Let
\[
p_\eps(x,y)
=
(2\pi\eps)^{-d/2}
\exp\left(-\frac{\|x-y\|^2}{2\eps}\right)
\]
and denote by
\[
K_\eps
:=
\inf_{\pi\in\Pi(\mu,\nu)}
H\left(
\pi\,\middle|\,
\mu(\rd x)p_\eps(x,y)\rd y
\right)
\]
the corresponding entropic cost. For every
$\pi\in\Pi(\mu,\nu)$, since the second marginal of $\pi$ is $\nu$,
\begin{align*}
H\left(
\pi\,\middle|\,
\mu(\rd x)p_\eps(x,y)\rd y
\right)
={}&
H(\pi\mid\mu\otimes\nu)
+
H(\nu\mid\mathcal L^d)
\\
&+
\frac{1}{2\eps}
\int\|x-y\|^2\pi(\rd x,\rd y)
+
\frac d2\log(2\pi\eps).
\end{align*}
Taking the infimum over $\pi\in\Pi(\mu,\nu)$ yields
\begin{equation}\label{eq:EOT-Pal-normalization}
{\rm EOT}_\eps
=
\eps K_\eps
-
\eps H(\nu\mid\mathcal L^d)
-
\frac d2\,\eps\log(2\pi\eps).
\end{equation}
The entropies of $\mu$ and $\nu$ relative to Lebesgue measure are
finite under our assumptions. Moreover, the quadratic-cost case of
\cite[Theorem~1]{pal2019difference} gives
\[
K_\eps-\frac{{\rm OT}}{\eps}
=
\frac12
\left\{
H(\nu\mid\mathcal L^d)
-
H(\mu\mid\mathcal L^d)
\right\}
+o(1).
\]
Combining this expansion with
\eqref{eq:EOT-Pal-normalization}, we obtain
\begin{align*}
{\rm EOT}_\eps-{\rm OT}
={}&
\frac d2\,\eps\log(1/\eps)
-\frac d2\,\eps\log(2\pi)
-
\frac{\eps}{2}
\left\{
H(\mu\mid\mathcal L^d)
+
H(\nu\mid\mathcal L^d)
\right\}
+o(\eps)
\\
=&
\frac d2\,\eps\log(1/\eps)+O(\eps).
\end{align*}
Therefore, by~\eqref{eq:value-potential-identity}, for $\eps$
sufficiently small,
\begin{equation}\label{eq:rate-eot}
\frac1C\eps\log(1/\eps)
\leq
\left|
\int
\bigl[
(\varphi_\eps-\varphi)
\oplus
(\psi_\eps-\varphi^*)
\bigr]
\dd(\mu\otimes\nu)
\right|
\leq
C\eps\log(1/\eps).
\end{equation}
The previous estimates on the two potentials give
\begin{align*}
\inf_{b\in\R}
\|(\varphi_\eps-\varphi)
\oplus
(\psi_\eps-\varphi^*)-b\|_{L^1(\mu\otimes\nu)}
&\leq
\inf_{a\in\R}
\|\varphi_\eps-\varphi-a\|_{L^1(\mu)}
+
\inf_{b\in\R}
\|\psi_\eps-\varphi^*-b\|_{L^1(\nu)}
\\
&\lesssim
\eps\log(1/\eps).
\end{align*}
Choose $b_\eps\in\R$ such that
\[
\|(\varphi_\eps-\varphi)
\oplus
(\psi_\eps-\varphi^*)-b_\eps\|_{L^1(\mu\otimes\nu)}
\leq 2 \cdot \inf_{b\in\R}
\|(\varphi_\eps-\varphi)
\oplus
(\psi_\eps-\varphi^*)-b\|_{L^1(\mu\otimes\nu)}.
\]
Then, by Jensen's inequality and \eqref{eq:rate-eot},
\begin{align*}
|b_\eps|
&\leq
\left|
\int
\bigl[
(\varphi_\eps-\varphi)
\oplus
(\psi_\eps-\varphi^*)
\bigr]
\dd(\mu\otimes\nu)
\right| +
\left|
\int
\bigl[
(\varphi_\eps-\varphi)
\oplus
(\psi_\eps-\varphi^*)-b_\eps
\bigr]
\dd(\mu\otimes\nu)
\right|
\\
&\leq
\left|
\int
\bigl[
(\varphi_\eps-\varphi)
\oplus
(\psi_\eps-\varphi^*)
\bigr]
\dd(\mu\otimes\nu)
\right| +
2 \cdot \inf_{b\in\R}
\|(\varphi_\eps-\varphi)
\oplus
(\psi_\eps-\varphi^*)-b\|_{L^1(\mu\otimes\nu)}
\\
&\lesssim
\eps\log(1/\eps).
\end{align*}
It follows that
\[
\|(\varphi_\eps-\varphi)
\oplus
(\psi_\eps-\varphi^*)\|_{L^1(\mu\otimes\nu)}
\leq
\|(\varphi_\eps-\varphi)
\oplus
(\psi_\eps-\varphi^*)-b_\eps\|_{L^1(\mu\otimes\nu)}
+
|b_\eps|
\lesssim
\eps\log(1/\eps),
\]
which is the upper bound in \cref{Thm:Sinkhorn}. Finally,
\eqref{eq:rate-eot} and Jensen's inequality give
\[
\|(\varphi_\eps-\varphi)
\oplus
(\psi_\eps-\varphi^*)\|_{L^1(\mu\otimes\nu)}
\geq
\left|
\int
\bigl[
(\varphi_\eps-\varphi)
\oplus
(\psi_\eps-\varphi^*)
\bigr]
\dd(\mu\otimes\nu)
\right|
\gtrsim
\eps\log(1/\eps),
\]
which proves the lower bound.
\bibliographystyle{imsart-nameyear}
\bibliography{ref}

@article{gonzalez2026influence,
  title={The Influence Function of Transport-based Quantiles},
  author={Gonz{\'a}lez-Sanz, Alberto and Sheng, Shunan and Wu, Bohan and Medina, Marco Avella},
  journal={arXiv preprint arXiv:2607.19080},
  year={2026}
}

@article{letrouit2024gluing,
  title={Gluing methods for quantitative stability of optimal transport maps},
  author={Letrouit, Cyril and M{\'e}rigot, Quentin},
  journal={arXiv preprint arXiv:2411.04908},
  year={2024}
}

@article {pal2019difference,
    AUTHOR = {Pal, Soumik},
     TITLE = {On the difference between entropic cost and the optimal
              transport cost},
   JOURNAL = {Ann. Appl. Probab.},
  FJOURNAL = {The Annals of Applied Probability},
    VOLUME = {34},
      YEAR = {2024},
    NUMBER = {1B},
     PAGES = {1003--1028},
      ISSN = {1050-5164,2168-8737},
   MRCLASS = {49Q22 (46N10 60F10 91G10 94A17)},
  MRNUMBER = {4700251},
       DOI = {10.1214/23-aap1983},
       URL = {https://doi.org/10.1214/23-aap1983},
}

@article{Goldfeld.et.al.2024.EJS,
  title = {Limit theorems for entropic optimal transport maps and Sinkhorn divergence},
  volume = {18},
  journal = {Electron. J. Statist. },
  author = {Goldfeld,  Ziv and Kato,  Kengo and Rioux,  Gabriel and Sadhu,  Ritwik},
  year = {2024},
  pages={980-1041},
}

@article{GalichonSalanie2022,
  author  = {Galichon, Alfred and Salani{\'e}, Bernard},
  title   = {Cupid's Invisible Hand: Social Surplus and Identification in Matching Models},
  journal = {The Review of Economic Studies},
  year    = {2022},
  volume  = {89},
  number  = {5},
  pages   = {2600--2629}
}

@article{ChiapporiMcCannNesheim2010,
  author  = {Chiappori, Pierre-Andr{\'e} and McCann, Robert J. and Nesheim, Lars P.},
  title   = {Hedonic Price Equilibria, Stable Matching, and Optimal Transport: Equivalence, Topology, and Uniqueness},
  journal = {Economic Theory},
  year    = {2010},
  volume  = {42},
  number  = {2},
  pages   = {317--354},
  doi     = {10.1007/s00199-009-0455-z}
}

@article{EkelandGalichonHenry2012,
  author  = {Ekeland, Ivar and Galichon, Alfred and Henry, Marc},
  title   = {Comonotonic Measures of Multivariate Risks},
  journal = {Mathematical Finance},
  year    = {2012},
  volume  = {22},
  number  = {1},
  pages   = {109--132},
  doi     = {10.1111/j.1467-9965.2010.00453.x}
}

@article{BercuBigotThurin2024,
  author  = {Bercu, Bernard and Bigot, J{\'e}r{\'e}mie and Thurin, Gauthier},
  title   = {{M}onge--{K}antorovich Superquantiles and Expected Shortfalls with Applications to Multivariate Risk Measurements},
  journal = {Electronic Journal of Statistics},
  year    = {2024},
  volume  = {18},
  number  = {2},
  pages   = {3461--3496},
  doi     = {10.1214/24-EJS2279}
}

@book{Galichon2016,
  author    = {Galichon, Alfred},
  title     = {Optimal Transport Methods in Economics},
  publisher = {Princeton University Press},
  address   = {Princeton, NJ},
  year      = {2016}
}

@inproceedings{MakkuvaTaghvaeiOhLee2020,
  author    = {Makkuva, Ashok Vardhan and Taghvaei, Amirhossein and Oh, Sewoong and Lee, Jason D.},
  title     = {Optimal Transport Mapping via Input Convex Neural Networks},
  booktitle = {Proceedings of the 37th International Conference on Machine Learning},
  series    = {Proceedings of Machine Learning Research},
  volume    = {119},
  pages     = {6672--6681},
  year      = {2020},
  url       = {https://proceedings.mlr.press/v119/makkuva20a.html}
}

@inproceedings{PatyDAspremontCuturi2020,
  author    = {Paty, Fran{\c{c}}ois-Pierre and d'Aspremont, Alexandre and Cuturi, Marco},
  title     = {Regularity as Regularization: Smooth and Strongly Convex {B}renier Potentials in Optimal Transport},
  booktitle = {Proceedings of the Twenty Third International Conference on Artificial Intelligence and Statistics},
  series    = {Proceedings of Machine Learning Research},
  volume    = {108},
  pages     = {1222--1232},
  year      = {2020},
  url       = {https://proceedings.mlr.press/v108/paty20a.html}
}

@article{Brenier1991,
  author  = {Brenier, Yann},
  title   = {Polar Factorization and Monotone Rearrangement of Vector-Valued Functions},
  journal = {Communications on Pure and Applied Mathematics},
  year    = {1991},
  volume  = {44},
  number  = {4},
  pages   = {375--417},
  doi     = {10.1002/cpa.3160440402}
}

@article{BenamouBrenier1998,
  author  = {Benamou, Jean-David and Brenier, Yann},
  title   = {Weak Existence for the Semigeostrophic Equations Formulated as a Coupled {M}onge--{A}mp{\`e}re/Transport Problem},
  journal = {SIAM Journal on Applied Mathematics},
  year    = {1998},
  volume  = {58},
  number  = {5},
  pages   = {1450--1461},
  doi     = {10.1137/S0036139995294111}
}

@article{BourneEganPelloniWilkinson2022,
  author  = {Bourne, David P. and Egan, Charlie P. and Pelloni, Beatrice and Wilkinson, Mark},
  title   = {Semi-Discrete Optimal Transport Methods for the Semi-Geostrophic Equations},
  journal = {Calculus of Variations and Partial Differential Equations},
  year    = {2022},
  volume  = {61},
  number  = {1},
  pages   = {39},
  doi     = {10.1007/s00526-021-02133-z}
}

@article{LevyMohayaeeVonHausegger2021,
  author  = {L{\'e}vy, Bruno and Mohayaee, Roya and von Hausegger, Sebastian},
  title   = {A Fast Semidiscrete Optimal Transport Algorithm for a Unique Reconstruction of the Early Universe},
  journal = {Monthly Notices of the Royal Astronomical Society},
  year    = {2021},
  volume  = {506},
  number  = {1},
  pages   = {1165--1185},
  doi     = {10.1093/mnras/stab1676}
}

@article{NikakhtarPadmanabhanLevyShethMohayaee2023,
  author  = {Nikakhtar, Farnik and Padmanabhan, Nikhil and L{\'e}vy, Bruno and Sheth, Ravi K. and Mohayaee, Roya},
  title   = {Optimal Transport Reconstruction of Biased Tracers in Redshift Space},
  journal = {Physical Review D},
  year    = {2023},
  volume  = {108},
  number  = {8},
  pages   = {083534},
  doi     = {10.1103/PhysRevD.108.083534}
}

@article{delBarrioLoubes2019,
  author  = {del Barrio, Eustasio and Loubes, Jean-Michel},
  title   = {Central Limit Theorems for Empirical Transportation Cost in General Dimension},
  journal = {The Annals of Probability},
  year    = {2019},
  volume  = {47},
  number  = {2},
  pages   = {926--951},
  doi     = {10.1214/18-AOP1275}
}

@inproceedings{MenaNilesWeed2019,
  author    = {Mena, Gonzalo and Niles-Weed, Jonathan},
  title     = {Statistical Bounds for Entropic Optimal Transport: Sample Complexity and the Central Limit Theorem},
  booktitle = {Advances in Neural Information Processing Systems},
  volume    = {32},
  year      = {2019}
}

@article{HutterRigollet2021,
  author  = {H{\"u}tter, Jan-Christian and Rigollet, Philippe},
  title   = {Minimax Estimation of Smooth Optimal Transport Maps},
  journal = {The Annals of Statistics},
  year    = {2021},
  volume  = {49},
  number  = {2},
  pages   = {1166--1194}
}

@book{Kosorok.2008.book,
  author    = {Kosorok, Michael R.},
  title     = {Introduction to Empirical Processes and Semiparametric Inference},
  series    = {Springer Series in Statistics},
  publisher = {Springer},
  address   = {New York},
  year      = {2008},
  doi       = {10.1007/978-0-387-74978-5}
}

@unpublished{ManoleBalakrishnanNilesWeedWasserman2023clt,
  author = {Manole, Tudor and Balakrishnan, Sivaraman and Niles-Weed, Jonathan and Wasserman, Larry},
  title  = {Central Limit Theorems for Smooth Optimal Transport Maps},
  note   = {arXiv:2312.12407},
  year   = {2023}
}

@article{HallinMordant2025,
  author  = {Hallin, Marc and Mordant, Gilles},
  title   = {Multiple-Attribute {L}orenz Functions and {G}ini Indices: A Measure Transportation Approach},
  journal = {Journal of Business \& Economic Statistics},
  year    = {2025},
  volume  = {43},
  number  = {4},
  pages   = {1092--1104},
  doi     = {10.1080/07350015.2025.2475964}
}

@article{SommerfeldMunk2018,
  title = {Inference for Empirical Wasserstein Distances on Finite Spaces},
  volume = {80},
  number = {1},
  journal = {Journal of the Royal Statistical Society Series B: Statistical Methodology},
  author = {Sommerfeld,  Max and Munk,  Axel},
  year = {2018},
  pages = {219–238}
}

@article {Fefferman2009Extension,
    AUTHOR = {Fefferman, Charles},
     TITLE = {Extension of {$C^{m,\omega}$}-smooth functions by linear
              operators},
   JOURNAL = {Rev. Mat. Iberoam.},
  FJOURNAL = {Revista Matem\'atica Iberoamericana},
    VOLUME = {25},
      YEAR = {2009},
    NUMBER = {1},
     PAGES = {1--48},
      ISSN = {0213-2230,2235-0616},
   MRCLASS = {46E15 (26E10 41A05)},
  MRNUMBER = {2514337},
MRREVIEWER = {M.\ Laczkovich},
       DOI = {10.4171/RMI/568},
       URL = {https://doi.org/10.4171/RMI/568},
}

@book{peyre2019computational,
  title={Computational optimal transport: With applications to data science},
  author={Peyr{\'e}, Gabriel and Cuturi, Marco},
  year={2019},
  publisher={Now Foundations and Trends}
}

@article{PanaretosZemel2019,
  author  = {Panaretos, Victor M. and Zemel, Yoav},
  title   = {Statistical Aspects of {W}asserstein Distances},
  journal = {Annual Review of Statistics and Its Application},
  year    = {2019},
  volume  = {6},
  pages   = {405--431},
  doi     = {10.1146/annurev-statistics-030718-104938}
}

@article {JonathanFrancis2019,
    AUTHOR = {Weed, Jonathan and Bach, Francis},
     TITLE = {Sharp asymptotic and finite-sample rates of convergence of
              empirical measures in {W}asserstein distance},
   JOURNAL = {Bernoulli},
  FJOURNAL = {Bernoulli. Official Journal of the Bernoulli Society for
              Mathematical Statistics and Probability},
    VOLUME = {25},
      YEAR = {2019},
    NUMBER = {4A},
     PAGES = {2620--2648},
      ISSN = {1350-7265,1573-9759},
   MRCLASS = {60B10 (62G30)},
  MRNUMBER = {4003560},
MRREVIEWER = {Aihua\ Xia},
       DOI = {10.3150/18-BEJ1065},
       URL = {https://doi.org/10.3150/18-BEJ1065},
}

@article {Dudley1968,
    AUTHOR = {Dudley, R. M.},
     TITLE = {The speed of mean {G}livenko-{C}antelli convergence},
   JOURNAL = {Ann. Math. Statist.},
  FJOURNAL = {Annals of Mathematical Statistics},
    VOLUME = {40},
      YEAR = {1968},
     PAGES = {40--50},
      ISSN = {0003-4851},
   MRCLASS = {60.30},
  MRNUMBER = {236977},
MRREVIEWER = {G.\ Schay},
       DOI = {10.1214/aoms/1177697802},
       URL = {https://doi.org/10.1214/aoms/1177697802},
}

@article{kitagawa2025stability,
  title={Stability of optimal transport maps on Riemannian manifolds},
  author={Kitagawa, Jun and Letrouit, Cyril and M{\'e}rigot, Quentin},
  journal={arXiv preprint arXiv:2504.05412},
  year={2025}
}

@article{luxburg2004distance,
  title={Distance-based classification with Lipschitz functions},
  author={Luxburg, Ulrike von and Bousquet, Olivier},
  journal={Journal of Machine Learning Research},
  volume={5},
  number={Jun},
  pages={669--695},
  year={2004}
}

@book{brezis2011functional,
  title={Functional analysis, Sobolev spaces and partial differential equations},
  author={Br{\'e}zis, Haim},
  volume={2},
  year={2011},
  publisher={Springer}
}

@article {JohnNirenberg1961,
    AUTHOR = {John, F. and Nirenberg, L.},
     TITLE = {On functions of bounded mean oscillation},
   JOURNAL = {Comm. Pure Appl. Math.},
  FJOURNAL = {Communications on Pure and Applied Mathematics},
    VOLUME = {14},
      YEAR = {1961},
     PAGES = {415--426},
      ISSN = {0010-3640,1097-0312},
   MRCLASS = {26.00},
  MRNUMBER = {131498},
MRREVIEWER = {L.\ C.\ Young},
       DOI = {10.1002/cpa.3160140317},
       URL = {https://doi.org/10.1002/cpa.3160140317},
}

@article{Jones1980BMO,
  author  = {Jones, Peter W.},
  title   = {Extension Theorems for {BMO}},
  journal = {Indiana University Mathematics Journal},
  year    = {1980},
  volume  = {29},
  number  = {1},
  pages   = {41--66}
}

@book {HanLinPDE2011Notes,
    AUTHOR = {Han, Qing and Lin, Fanghua},
     TITLE = {Elliptic partial differential equations},
    SERIES = {Courant Lecture Notes in Mathematics},
    VOLUME = {1},
   EDITION = {Second},
 PUBLISHER = {Courant Institute of Mathematical Sciences, New York; American
              Mathematical Society, Providence, RI},
      YEAR = {2011},
     PAGES = {x+147},
      ISBN = {978-0-8218-5313-9},
   MRCLASS = {35Jxx (35-01 35B50)},
  MRNUMBER = {2777537},
}

@article {DelalandeMerigot2023,
    AUTHOR = {Delalande, Alex and M\'erigot, Quentin},
     TITLE = {Quantitative stability of optimal transport maps under
              variations of the target measure},
   JOURNAL = {Duke Math. J.},
  FJOURNAL = {Duke Mathematical Journal},
    VOLUME = {172},
      YEAR = {2023},
    NUMBER = {17},
     PAGES = {3321--3357},
      ISSN = {0012-7094,1547-7398},
   MRCLASS = {49Q22 (30L99 49K40)},
  MRNUMBER = {4688680},
MRREVIEWER = {Lukas\ Koch},
       DOI = {10.1215/00127094-2022-0106},
       URL = {https://doi.org/10.1215/00127094-2022-0106},
}

@book{wainwright2019high,
  title={High-dimensional statistics: A non-asymptotic viewpoint},
  author={Wainwright, Martin J},
  volume={48},
  year={2019},
  publisher={Cambridge university press}
}

@article{del2025distributional,
  title={Distributional limit theory for optimal transport},
  author={{del Barrio}, Eustasio and Gonz{\'a}lez-Sanz, Alberto and Loubes, Jean-Michel and Rodr{\'\i}guez-V{\'\i}tores, David},
  journal={arXiv preprint arXiv:2505.19104},
  year={2026}
}

@incollection{van1996weak,
  title={Weak convergence},
  author={{van der Vaart}, Aad W and Wellner, Jon A},
  booktitle={Weak convergence and empirical processes: with applications to statistics},
  pages={16--28},
  year={1996},
  publisher={Springer}
}

@book{miranda1970partial,
  title={Partial Differential Equations of Elliptic Type},
  author={Miranda, Carlo},
  edition={2nd Revised},
  year={1970},
  publisher={
Springer-Verlag
},
  address={Berlin, Heidelberg, New York},
  translator={Motteler, Zane C.},
  series={Ergebnisse der Mathematik und ihrer Grenzgebiete},
  volume={2},
  isbn={978-3-540-04804-6}
}

@book {GilbargTrudinger.Book,
    AUTHOR = {Gilbarg, David and Trudinger, Neil S.},
     TITLE = {Elliptic partial differential equations of second order},
    SERIES = {Grundlehren der mathematischen Wissenschaften [Fundamental
              Principles of Mathematical Sciences]},
    VOLUME = {224},
   EDITION = {Second},
 PUBLISHER = {Springer-Verlag, Berlin},
      YEAR = {1983},
     PAGES = {xiii+513},
      ISBN = {3-540-13025-X},
   MRCLASS = {35Jxx (35-01)},
  MRNUMBER = {737190},
MRREVIEWER = {O.\ John},
       DOI = {10.1007/978-3-642-61798-0},
       URL = {https://doi.org/10.1007/978-3-642-61798-0},
}

@book{Taira.2024.book,
  title = {Real Analysis Methods for Markov Processes: Singular Integrals and Feller Semigroups},
  ISBN = {9789819736591},
  url = {http://dx.doi.org/10.1007/978-981-97-3659-1},
  DOI = {10.1007/978-981-97-3659-1},
  publisher = {Springer Nature Singapore},
  author = {Taira,  Kazuaki},
  year = {2024}
}

@article{cianchi1996continuity,
  title={Continuity properties of functions from Orlicz-Sobolev spaces and embedding theorems},
  author={Cianchi, Andrea},
  journal={Annali della Scuola Normale Superiore di Pisa-Classe di Scienze},
  volume={23},
  number={3},
  pages={575--608},
  year={1996}
}

@article {Manole.et.al.AoS.2024,
    AUTHOR = {Manole, Tudor and Balakrishnan, Sivaraman and Niles-Weed,
              Jonathan and Wasserman, Larry},
     TITLE = {Plugin estimation of smooth optimal transport maps},
   JOURNAL = {Ann. Statist.},
  FJOURNAL = {The Annals of Statistics},
    VOLUME = {52},
      YEAR = {2024},
    NUMBER = {3},
     PAGES = {966--998},
      ISSN = {0090-5364,2168-8966},
   MRCLASS = {62G05 (49Q22 62C20 62G07 62G20)},
  MRNUMBER = {4784066},
MRREVIEWER = {Dimitris\ Vartziotis},
       DOI = {10.1214/24-aos2379},
       URL = {https://doi.org/10.1214/24-aos2379},
}

@article {Malamut-Maxime.SIMA.2025,
    AUTHOR = {Malamut, Hugo and Sylvestre, Maxime},
     TITLE = {Convergence rates of the regularized optimal transport:
              disentangling suboptimality and entropy},
   JOURNAL = {SIAM J. Math. Anal.},
  FJOURNAL = {SIAM Journal on Mathematical Analysis},
    VOLUME = {57},
      YEAR = {2025},
    NUMBER = {3},
     PAGES = {2533--2558},
      ISSN = {0036-1410,1095-7154},
   MRCLASS = {49Q22 (49K40 94A17)},
  MRNUMBER = {4907179},
MRREVIEWER = {Lukas\ Koch},
       DOI = {10.1137/23M1591554},
       URL = {https://doi.org/10.1137/23M1591554},
}

@article{delBarrioGonzalezSanzLoubes2024,
  author  = {del Barrio, Eustasio and Gonz{\'a}lez-Sanz, Alberto and Loubes, Jean-Michel},
  title   = {Central limit theorems for semidiscrete {W}asserstein distances},
  journal = {Bernoulli},
  year    = {2024},
  volume  = {30},
  number  = {1},
  pages   = {554--580},
  doi     = {10.3150/23-BEJ1608}
}

@misc{MischlerTrevisan2024,
  author        = {Mischler, Octave and Trevisan, Dario},
  title         = {Quantitative stability in optimal transport for general power costs},
  year          = {2024},
  eprint        = {2407.19337},
  archivePrefix = {arXiv},
  primaryClass  = {math.OC}
}

@article{NutzWiesel2022,
  author  = {Nutz, Marcel and Wiesel, Johannes},
  title   = {Entropic optimal transport: convergence of potentials},
  journal = {Probability Theory and Related Fields},
  year    = {2022},
  volume  = {184},
  pages   = {401--424},
  doi     = {10.1007/s00440-021-01096-8}
}

@misc{GonzalezSanzLoubesNilesWeed2022,
  author        = {Gonz{\'a}lez-Sanz, Alberto and Loubes, Jean-Michel and Niles-Weed, Jonathan},
  title         = {Weak limits of entropy regularized optimal transport: potentials, plans and divergences},
  year          = {2022},
  eprint        = {2207.07427},
  archivePrefix = {arXiv},
  primaryClass  = {math.PR}
}

@unpublished{Mordant2024Entropic,
  author = {Mordant, Gilles},
  title  = {The Entropic Optimal (Self-)Transport Problem: Limit Distributions for Decreasing Regularization with Application to Score Function Estimation},
  note   = {arXiv:2412.12007, revised April 2026},
  year   = {2024}
}

@article{FreesValdez1998,
  author  = {Frees, Edward W. and Valdez, Emiliano A.},
  title   = {Understanding Relationships Using Copulas},
  journal = {North American Actuarial Journal},
  year    = {1998},
  volume  = {2},
  number  = {1},
  pages   = {1--25},
  doi     = {10.1080/10920277.1998.10595667}
}

@misc{CazellesPauwelsPortales2026,
  author        = {Cazelles, Elsa and Pauwels, Edouard and Portales, L{\'e}o},
  title         = {Statistical Estimation of {M}onge Transport Maps via {B}renier Potentials},
  year          = {2026},
  eprint        = {2604.22366},
  archivePrefix = {arXiv},
  primaryClass  = {math.OC}
}

@misc{CajaLopezDelgadinoKitagawa2026,
  author        = {Caja-Lopez, F.-U. and Delgadino, Matias G. and Kitagawa, Jun},
  title         = {Stability of optimal transport maps and second variation of the 2-{M}onge--{K}antorovich distance},
  year          = {2026},
  eprint        = {2605.24232},
  archivePrefix = {arXiv},
  primaryClass  = {math.AP}
}

@misc{LopezRivera2026,
  author        = {L{\'o}pez-Rivera, Pablo},
  title         = {A uniform rate of convergence for the entropic potentials in the quadratic {E}uclidean setting},
  year          = {2026},
  note          = {To appear in ESAIM: Control, Optimisation and Calculus of Variations},
  eprint        = {2502.00084},
  archivePrefix = {arXiv},
  primaryClass  = {math.CA}
}

\newpage
\appendix

\section{Numerical details for the risk application}
\label{app:risk-application}

For a target measure $\sum_j b_j\delta_{Y_j}$, we minimize the
continuous-source semidual
\[
 v\longmapsto
 \mathbb E_{U\sim\mu}\max_j
 \left\{\langle U,Y_j\rangle-\frac12\|Y_j\|^2+v_j\right\}
 -\sum_jb_jv_j
\]
by averaged stochastic gradient descent.  Each fit uses $10n=14{,}660$
independent draws from the uniform disk and the step size
$0.01/\sqrt t$.  The same disk draw is used for the original fit and all
bootstrap fits.  A $2{,}400$-point discrete solve supplies only the initial
dual weights; the reported potentials are the averaged iterates for the
continuous semidual.  The two duplicated claim vectors are consolidated and
given their combined empirical weights.

We use seed $20260726$ for $B=1{,}000$ ordinary bootstrap samples of size
$n=1{,}466$.  Repeated bootstrap observations are represented by their exact
multinomial weights.  Contrasts are evaluated independently of the SGD draw,
using $32{,}768$ quasi-Monte Carlo points in each of $\Omega_a$ and $\Omega_b$
for every bootstrap fit and $131{,}072$ points per region for the original
fit.  The algorithm has a fixed draw budget rather than a stopping tolerance.
All $1{,}000$ requested fits returned finite values; no run was discarded or
replaced.

The basic interval uses the empirical $0.025$ and $0.975$ quantiles of
$\sqrt n(\widehat\theta_n^*-\widehat\theta_n)$.  Across the first
$250,500,750$, and $1{,}000$ bootstrap draws, the estimated root-$n$ standard
deviations were $0.746$, $0.728$, $0.741$, and $0.735$, respectively.  Over
twenty bootstrap fits repeated with an independent disk draw, the largest
absolute change in the contrast was $4.5\times10^{-4}$ and the root mean
squared change was $1.9\times10^{-4}$.  Finally, using between $8{,}192$ and
$131{,}072$ evaluation points per region kept the original-sample estimate in
$[1.52861,1.52866]$.  The code, complete bootstrap output, software versions,
and data checksum accompany the paper.

\section{Elliptic equations}
\label{sec: Elliptic}
The present section and~\cref{sec:green-function} are devoted to establishing the BMO estimates; see~\cref{theorem-Solvable}. We denote by $|\cdot|$ the Lebesgue measure. Let $\Omega\subset \mathbb R^d$ be a bounded domain with $\mathcal{C}^{2,\alpha}$ boundary $\partial\Omega$, and let $\mathbf{n}_\Omega$ denote the outward unit normal vector field on $\partial\Omega$. We consider the divergence-form operator
\[
    \mathcal{L}_A u := \operatorname{div}(A\nabla u),
\]
together with the boundary operator
\[
    \mathcal{B}_A u := \langle A\nabla u,\mathbf{n}_\Omega\rangle .
\]

\begin{assumption}
\label{ass:RegPDE}
The matrix-valued coefficient $A:\Omega\to\mathbb R^{d\times d}$ satisfies the following conditions.
\begin{enumerate}
    \item \emph{Ellipticity.} The matrix $A(x)$ is symmetric and positive definite for every $x\in\Omega$. Moreover, there exists $\lambda_A\in(0,1]$ such that
    \[
        \lambda_A \|\xi\|^2
        \leq
        \langle A(x)\xi,\xi\rangle
        \leq
        \lambda_A^{-1}\|\xi\|^2,
        \qquad x\in\Omega,\ \xi\in\mathbb R^d.
    \]
    \item \emph{Regularity of coefficients.} The entries of $A$ belong to $\mathcal{C}^{1,\alpha}(\Omega)$.
\end{enumerate}
\end{assumption}

Under~\cref{ass:RegPDE}, the operator $\cL_A$ is uniformly elliptic, and the boundary condition is uniformly oblique. In our setting, we take
\[
    A=\mu[\nabla^2\varphi]^{-1},
\]
where we identify $\mu$ with its density. Under the assumptions of~\cref{theorem:main}, this density is bounded from below by $\lambda$ and from above by $\|\mu\|_{1,\alpha,\Omega}$. Moreover, there exists a constant $C_0>0$ such that $\frac{1}{C_0}I_d \preceq \nabla^2\varphi \preceq C_0 I_d$ on $\Omega$.
Consequently, $A$ satisfies the uniform ellipticity condition in~\cref{ass:RegPDE}, with ellipticity constant $\lambda_A$ depending only on
$\lambda$, $\|\mu\|_{1,\alpha,\Omega}$, and $C_0$.
The main theorem of this section is the following.

\begin{theorem}\label{theorem-Solvable}
Under~\cref{ass:RegPDE}, for any $f\in L_0^\infty(\Omega)$, there exists a unique solution
$[u]\in\mathcal{X}=W^{2,BMO}(\Omega)/\langle1\rangle$ of 
\begin{equation}
\label{eq:Elliptic-system}
\begin{cases}
    \mathcal{L}_A(u) = f & {\rm in}\ \Omega,\\
    \mathcal{B}_A(u) = 0 &{\rm on}\ \partial \Omega.
\end{cases}
\end{equation}
Furthermore, there exists a constant $C>0$, independent of $f$, such that
\[
    \|[u]\|_{\mathcal{X}}
    \leq
    C\|f\|_{\infty,\Omega}.
\] 
\end{theorem}
When $d=1$, the result follows directly. Indeed, $\Omega=(a,b)$ and $A$ is scalar. Taking the mean-zero representative, the equation and the Neumann condition give
\[
    A(x)u'(x)=\int_a^x f(t)\dd t,
\]
where the boundary condition at $b$ follows from $\int_a^b f=0$. Consequently,
\[
    u''(x)
    =
    \frac{f(x)}{A(x)}
    -
    \frac{A'(x)}{A(x)^2}\int_a^x f(t)\dd t.
\]
Uniform ellipticity and $A\in\cC^1(\overline\Omega)$ therefore yield
\[
    \|u'\|_{\infty,\Omega}
    +
    \|u''\|_{\infty,\Omega}
    \lesssim
    \|f\|_{\infty,\Omega}.
\]
Together with the mean-zero normalization, this is stronger than the claimed
$W^{2,BMO}$ estimate. In the Green-function construction below, we may therefore assume
that $d\geq2$.

\begin{remark}[Bibliographical remarks]
The solvability theory for uniformly elliptic equations with Neumann or oblique
boundary conditions is classical. For $1<p<\infty$, estimates of the form
$L^p_0(\Omega)\to W^{2,p}(\Omega)/\langle 1\rangle$ go back to the Schauder and
Calderón--Zygmund theories, together with the regularity theory for oblique
derivative problems; see, for instance, \citep{miranda1970partial,GilbargTrudinger.Book}
and the modern presentation in \citep[Section~16]{Taira.2024.book}.

When $g\in L_0^\infty(\Omega)$, one should not in general expect second
derivatives to be bounded. The natural replacement at this endpoint is the space
$BMO$, introduced by John and Nirenberg \citep{JohnNirenberg1961}. Thus
\Cref{theorem-Solvable} can be viewed as the Neumann analogue of the
Calderón--Zygmund estimate: bounded data yield second derivatives of bounded mean
oscillation (BMO), with uniqueness understood modulo constants because of the Neumann
boundary condition. This BMO estimate is precisely what is needed below,
since $W^{2,BMO}(\Omega)$ embeds into $\cC^{1,{\rm LogLip}}(\Omega)$ 
by~\cref{lemma:BMO-inclusion}.
\end{remark}

\begin{corollary}
Under the assumptions of \cref{theorem-Solvable}, let $u\in \mathcal{X}$ be the unique solution of \eqref{eq:Elliptic-system}. Then
there exists a constant $C>0$, independent of $f$, such that
\[
    \|\nabla u(x)-\nabla u(x')\|
    \leq
    C\|f\|_{\infty,\Omega}
    \|x-x'\|
    \bigl(1+|\log(\|x-x'\|)|\bigr).
\] 
\end{corollary}

\subsection{Existence of Green function}
We start with a known result involving the existence and uniqueness of solutions in $W^{2,p}(\Omega)$. Without loss of generality, we assume that $|\Omega|=1$. The proof can be found in \cite[Section~16]{Taira.2024.book}.

\begin{theorem}\label{theorem:solvable-lp}
Fix $p\in(1,\infty)$. Let $f\in L^p(\Omega)$ be centered, i.e.,
$\int_\Omega f(z)\dd z=0$. There exists a unique solution
$u\in W^{2,p}(\Omega)/\langle1\rangle$ to the system
\begin{equation}\label{eq:Lp-system}
\begin{cases}
    \mathcal{L}_A(u)=f & {\rm in}\ \Omega,\\
    \mathcal{B}_A(u)=0 & {\rm on}\ \partial\Omega.
\end{cases}
\end{equation}
Here the boundary is to be understood in the sense of the trace. Moreover,
\begin{equation}
    \inf_{c\in\R}\|u-c\|_{W^{2,p}(\Omega)}
    \leq
    C\|f\|_{L^p(\Omega)},
\end{equation}
for some constant $C$ depending on $d,\lambda_A,\Omega,p$, and $\|A\|_{1,\alpha;\Omega}$. 
\end{theorem}
For $p\in(1,\infty)$, set
\begin{equation}
    \label{eq:zero-traze-space}
     E_p
    :=
    \left\{
        u\in W^{2,p}(\Omega):
        \mathcal{B}_Au=0\ \text{in the trace sense},
        \quad
        \int_\Omega u=0
    \right\}.
\end{equation} 
By~\cref{theorem:solvable-lp}, the operator
\[
    \mathcal{L}_A:E_p\longrightarrow L_0^p(\Omega)
\]
is bounded and invertible. Identifying $(L_0^p(\Omega))^*$ with
$L_0^{p'}(\Omega)$ by choosing mean-zero representatives, its adjoint
\[
    \mathcal{L}_A^*:L_0^{p'}(\Omega)\longrightarrow E_p^*
\]
is also bounded and invertible.

For $p>d/2$, Morrey's inequality
\cite[Theorem~7.26]{GilbargTrudinger.Book} and the Arzelà--Ascoli theorem give
\begin{equation}
\label{eq:contentionW-C-embedding}
    E_p
    \Subset
    \mathcal{C}^{0,\beta}(\overline\Omega)
    \hookrightarrow
    \mathcal{C}(\overline\Omega),
    \qquad
    0<\beta<\min\left\{1,2-\frac{d}{p}\right\}.
\end{equation}
Call $ \mathcal{M}_0(\Omega)
    :=
    \bigl(\mathcal{C}(\overline\Omega)/\langle1\rangle\bigr)^*$ 
the space of finite Radon measures with zero total mass. The compact embedding in
\eqref{eq:contentionW-C-embedding} has a compact adjoint
$\mathcal{M}_0(\Omega)\to E_p^*$. Hence, by composition,
\[
    (\mathcal{L}_A^*)^{-1}:
    \mathcal{M}_0(\Omega)
    \longrightarrow
    L_0^{p'}(\Omega)
\]
is bounded and compact. For the moment, denote
\[
    G_p(x,\cdot)
    :=
    (\mathcal{L}_A^*)^{-1}(\delta_x-1),
    \qquad x\in\Omega,
\]
where $1$ denotes Lebesgue measure on $\Omega$. This definition is independent of the
choice of $p>d/2$. Indeed, let $p_1,p_2>d/2$ and let
$f\in\mathcal{C}^{\infty}(\overline\Omega)$ be centered. If $u$ is the unique mean-zero
solution of $\mathcal{L}_Au=f$, then, for $i=1,2$,
\[
\begin{aligned}
    \int_\Omega f(y)G_{p_i}(x,y)\dd y
    &=
    \left\langle
        G_{p_i}(x,\cdot),
        \mathcal{L}_Au
    \right\rangle
    \\
    &=
    \left\langle
        \delta_x-1,u
    \right\rangle
    =
    u(x).
\end{aligned}
\]
The mean-zero solution $u$ does not depend on $p_i$ by the uniqueness in
\cref{theorem:solvable-lp}. Since both kernels have zero integral, the same identity holds
for arbitrary smooth test functions after subtracting their mean. Thus
$G_{p_1}(x,\cdot)=G_{p_2}(x,\cdot)$ almost everywhere. We henceforth denote the common
kernel by $G(x,\cdot)$. 

We now present the following proposition.

\begin{proposition}\label{Proposition:green-function-characterization}
Fix $p>\frac d2$ and let $f\in L_0^p(\Omega)$. 
If $u=\mathcal{L}_A^{-1}f\in E_p$, then
\[
    u(x)
    =
    \int_\Omega G(x,y)f(y)\dd y,
    \qquad x\in\Omega.
\]
In particular, the function on the right-hand side is the mean-zero solution of
\[
\begin{cases}
    \mathcal{L}_A(u)=f & {\rm in}\ \Omega,\\
    \mathcal{B}_A(u)=0 & {\rm on}\ \partial\Omega.
\end{cases}
\]
Furthermore, for any finite signed Radon measure
$\mu\in\mathcal{M}(\Omega)$ with $\mu(\Omega)=1$, one has
\[
    (\mathcal{L}_A^*)^{-1}(\mu-1)
    =
    \int G(x,\cdot)\dd\mu(x)
\]
in the following weak sense: for every $f\in L^p(\Omega)$,
\[
    \int_\Omega
    \left(
        \int_\Omega G(x,y)f(y)\dd y
    \right)
    \dd\mu(x)
    =
    \int_\Omega
    f(y)(\mathcal{L}_A^*)^{-1}(\mu-1)(y)\dd y.
\] 
\end{proposition}

\begin{proof} 
Note that $\delta_x-1\in\mathcal{M}_0(\Omega)$ since $|\Omega|=1$. Let
$u=\mathcal{L}_A^{-1}f\in E_p$. Since $p>d/2$, the function $u$ has a continuous
representative. By the definition of $G$,
\begin{align*}
    u(x)
    &=
    \langle\delta_x-1,u\rangle
    \\
    &=
    \left\langle
        \mathcal{L}_A^*G(x,\cdot),u
    \right\rangle
    \\
    &=
    \left\langle
        G(x,\cdot),\mathcal{L}_Au
    \right\rangle
    =
    \int_\Omega G(x,y)f(y)\dd y.
\end{align*}
This proves the first claim.

We now prove the second claim. Let $f\in L^p(\Omega)$ and put
\[
    f_0:=f-\int_\Omega f.
\]
Let $u=\mathcal{L}_A^{-1}f_0\in E_p$. Since
$G(x,\cdot)\in L_0^{p'}(\Omega)$, the first claim gives
\[
    \int_\Omega G(x,y)f(y)\dd y
    =
    \int_\Omega G(x,y)f_0(y)\dd y
    =
    u(x).
\]
Thus the outer integral against $\mu$ is well-defined because $u$ is continuous. Set
\[
    h:=(\mathcal{L}_A^*)^{-1}(\mu-1)\in L_0^{p'}(\Omega).
\]
Using $\int_\Omega u=0$ and $\mu(\Omega)=1$, we obtain
\begin{align*}
    \int_\Omega
    \left(
        \int_\Omega G(x,y)f(y)\dd y
    \right)
    \dd\mu(x)
    &=
    \int_\Omega u(x)\dd\mu(x)
    \\
    &=
    \langle\mu-1,u\rangle
    \\
    &=
    \langle\mathcal{L}_A^*h,u\rangle
    \\
    &=
    \langle h,\mathcal{L}_Au\rangle
    \\
    &=
    \int_\Omega h(y)f_0(y)\dd y
    \\
    &=
    \int_\Omega
    f(y)(\mathcal{L}_A^*)^{-1}(\mu-1)(y)\dd y,
\end{align*}
where the last equality follows from $\int_\Omega h=0$. This concludes the proof. 
\end{proof}

\begin{remark}\label{rmk:green-function}
By definition, $\cL_A^*(G(x,\cdot))=\delta_x-1$. That is, for any
$f\in E_p$, 
we have
\[
    \int_\Omega\mathcal{L}_A(f)(y)G(x,y)\dd y
    =
    \mathcal{L}_A^*(G(x,\cdot))(f)
    =
    f(x)-\int_\Omega f(z)\dd z.
\]
This justifies the use of the terminology \emph{Green's function}. 
Moreover, if $g\in\mathcal{C}_c^\infty(\Omega)$, then
$g-\int_\Omega g\in E_p$ and
\[
    \int_\Omega
    \mathcal{L}_A(g)(y)G(x,y)\dd y
    =
    g(x)-\int_\Omega g(z)\dd z.
\]
Therefore,
\[
    \mathcal{L}_A(G(x,\cdot))
    =
    \operatorname{div}(A\nabla_yG(x,\cdot))
    =
    \delta_x-1
\]
in the sense of distributions. 
\end{remark}

We now study qualitative estimates for the Green function $G$. The observation made in~\cref{rmk:green-function} will be critical for deriving these estimates.

\begin{proposition}\label{Proposition:green-function-qualitative}
The following hold:
\begin{enumerate}
   \item For every $x\in\Omega$, $G(x,\cdot)$ agrees a.e.~on
$\BB\cap\Omega$ with a function in
$\cC^{2,\alpha}(\overline{\BB\cap\Omega})$, for every open ball
$\BB$ satisfying $    \overline{\BB\cap\Omega}\cap\{x\}=\emptyset.$ 
    \item For a.e.~$(x,y)\in\Omega^2$, $G(y,x)=G(x,y)$.
    \item \item For every $y\in\Omega$, $G(\cdot,y)$ agrees a.e.~on
$\BB\cap\Omega$ with a function in
$\cC^{2,\alpha}(\overline{\BB\cap\Omega})$, for every open ball
$\BB$ satisfying $\overline{\BB\cap\Omega}\cap\{y\}=\emptyset.$ 
\end{enumerate} 
Moreover, for every
\[
    1<q<\frac{d}{d-2}
    \quad\text{if }d\geq3,
\]
and for every $1<q<\infty$ if $d=2$, one has
\[
    \sup_{x\in\Omega}\|G(x,\cdot)\|_{L^q(\Omega)}
    +
    \sup_{y\in\Omega}\|G(\cdot,y)\|_{L^q(\Omega)}
    <\infty.
\]
For every $\eta>0$, the $\cC^{2,\alpha}$ bounds in parts~(i)
and~(iii) are uniform over all balls and poles satisfying $\operatorname{dist}
    \bigl(x,\overline{\BB\cap\Omega}\bigr)\geq\eta$
in part~(i), and
$\operatorname{dist}
    \bigl(y,\overline{\BB\cap\Omega}\bigr)\geq\eta$
in part~(iii).
\end{proposition}

\begin{proof}$ $\\
{Proof of \textit{(i).}}
Fix $p>d/2$. 
For clarity, we fix $x_0\in\Omega$ instead of $x$ as in the
statement. Set $    K:=\overline{\BB\cap\Omega}.$ 
Then $K$ is compact and, by assumption, $x_0\notin K$. Hence $d_0:=\operatorname{dist}(x_0,K)>0.$ 
Choose
\[
    0<\gamma<
    \frac12
    \min\left\{
        d_0,\operatorname{dist}(x_0,\partial\Omega)
    \right\}.
\]
Then
\[
    \BB_\gamma(x_0)\Subset\Omega,
    \qquad
    \BB_\gamma(x_0)\cap K=\emptyset.
\]

Let $\{\rho_n\}_n$ be a sequence of $\cC^\infty$ probability densities with support
contained in $\BB_\gamma(x_0)$ and such that
$\rho_n\overset{\ast}{\rightharpoonup}\delta_{x_0}$ in the sense of weak convergence of
probability measures. Since $|\Omega|=1$, we have
$\rho_n-1\in L^p_0(\Omega)$. Fix $f\in L_0^p(\Omega)$. We use $\rho_n-1$ as a test
function to get
\begin{align*}
    \int_{\Omega} f (\cL_A^*)^{-1}(\rho_n-1)\dd x
    &=
    \int_{\Omega} \cL^{-1}_A(f)(\rho_n-1)\dd x
    \\
    &=
    \int_{\Omega}
    \cL^{-1}_A(f)
    \cL_A\cL^{-1}_A(\rho_n-1)\dd x
    \\
    &=
    \int_{\Omega}
    \cL^{-1}_A(f)
    {\rm div}\bigl(A\nabla\cL^{-1}_A(\rho_n-1)\bigr)\dd x
    \\
    &=
    -\int_{\Omega}
    \left\langle
        \nabla\cL^{-1}_A(f),
        A\nabla\cL^{-1}_A(\rho_n-1)
    \right\rangle
    \dd x
    \\
    &\qquad
    +
    \underbrace{
    \int_{\partial\Omega}
    \left\langle
        A\nabla\cL^{-1}_A(\rho_n-1),
        \mathbf{n}_\Omega
    \right\rangle
    \cL^{-1}_A(f)\dd S
    }_{
    =0\ \text{as }\cL_A^{-1}(L_0^p(\Omega))
    \subset\cB_A^{-1}(0)
    }
    \\
    &=
    -\int_{\Omega}
    \left\langle
        A\nabla\cL^{-1}_A(f),
        \nabla\cL^{-1}_A(\rho_n-1)
    \right\rangle
    \dd x
    \qquad\text{(by symmetry of $A$)}
    \\
    &=
    \int_{\Omega}
    {\rm div}\bigl(A\nabla\cL^{-1}_A(f)\bigr)
    \cL^{-1}_A(\rho_n-1)\dd x
    \\
    &\qquad
    -
    \underbrace{
    \int_{\partial\Omega}
    \left\langle
        A\nabla\cL^{-1}_A(f),
        \mathbf{n}_\Omega
    \right\rangle
    \cL^{-1}_A(\rho_n-1)\dd S
    }_{
    =0\ \text{as }\cL_A^{-1}(L_0^p(\Omega))
    \subset\cB_A^{-1}(0)
    }
    \\
    &=
    \int_{\Omega}
    f\cL^{-1}_A(\rho_n-1)\dd x.
\end{align*}
As a consequence,
\[
    u_n
    :=
    \cL^{-1}_A(\rho_n-1)
    =
    (\cL_A^*)^{-1}(\rho_n-1)
\]
by identification. 
By the Schauder regularity for the oblique problem
\cite[Section~16]{Taira.2024.book},
$u_n\in\cC^{2,\alpha}(\overline\Omega)$. 

Next, we show that $u_n$ is uniformly bounded in $\cC^{2,\alpha}$ away from $x_0$.
Since $K$ is compact, it is enough to argue locally and then use a finite covering. Fix $z_0\in K$. Choose $r=r(z_0)>0$ sufficiently small that
\[
    \BB_{4r}(z_0)\cap\BB_\gamma(x_0)=\emptyset,
\]
and such that either
$
    \overline{\BB_{4r}(z_0)}\subset\Omega,
$
or $\BB_{4r}(z_0)\cap\Omega$ is contained in one of the fixed
boundary coordinate neighborhoods of $\partial\Omega$. 
Schauder estimates for oblique derivative problems give
(use Lemma~6.29 in~\cite{GilbargTrudinger.Book} if
$\BB_{4r}(z_0)$ is close to $\partial\Omega$, or Corollary~6.3, ibid.,
if $\BB_{4r}(z_0)$ is far from $\partial\Omega$)
\begin{equation}
    \|u_n\|_{2,\alpha,\BB_r(z_0)\cap\Omega}
    \leq
    C\left(
        \|u_n\|_{\infty,\BB_{2r}(z_0)\cap\Omega}
        +1
    \right),
\end{equation}
where we used the fact that $\cL_Au_n=-1$ on
$\BB_{2r}(z_0)\cap\Omega$, as $\rho_n$ is supported on
$\BB_\gamma(x_0)$. It remains to bound
$\|u_n\|_{\infty,\BB_{2r}(z_0)\cap\Omega}$ uniformly in $n$.

If $\BB_{4r}(z_0)\subset{\rm int}(\Omega)$, then
\cite[Theorem~9.20]{GilbargTrudinger.Book} gives
\[
    \|u_n\|_{\infty,\BB_{2r}(z_0)\cap\Omega}
    \leq
    C(r)\left(
        \|u_n\|_{L^{p'}(\Omega)}
        +1
    \right).
\]
Since
$u_n\to G(x_0,\cdot)$ strongly in $L^{p'}(\Omega)$ by
\cref{lemma:strong-convergence}, the interior part is finished.

We derive the boundary estimates. By making $r$ smaller and flattening the boundary
near $z_0$, we can assume that we are in the setting of
\cite[Lemma~19.1]{Taira.2024.book}. Set
\[
    r_m
    =
    \left(
        3-\sum_{k=1}^m2^{-k}
    \right)r,
    \qquad m\geq0,
\]
so that $r_m<4r$. Then, for every $m\geq0$, we have
\begin{equation}
\label{eq:regularity-green-iteration}
    \|u_n\|_{W^{2,p_m}(\BB_{r_{m+1}}\cap\Omega)}
    \leq
    C(m)
    \left(
        \|u_n\|_{L^{p_m}(\BB_{r_m}\cap\Omega)}
        +1
    \right),
\end{equation}
where $C(m)$ is independent of $n$. Define $p_0=p'$. As long as
$p_m<d/2$, define
\[
    p_{m+1}
    =
    \frac{dp_m}{d-2p_m}
    >
    p_m.
\]
If $p_m=d/2$, the critical Sobolev embedding allows us to choose any
finite $p_{m+1}>d/2$. 

For $m=0$, we use the fact that
$\|u_n\|_{L^{p'}(\Omega)}\leq C$ to derive
\[
    \|u_n\|_{W^{2,p_0}(\BB_{r_1}\cap\Omega)}
    \leq
    C(0).
\]
If $p_0>d/2$, then the Sobolev embedding theorem gives
\[
    \|u_n\|_{\infty,\BB_{r_1}\cap\Omega}
    \leq
    C'(0)
    \|u_n\|_{W^{2,p_0}(\BB_{r_1}\cap\Omega)}
    \leq
    C'(0)C(0).
\]
If $p_0\leq d/2$, then the Sobolev embedding theorem
\cite[Eq.~(7.30)]{GilbargTrudinger.Book} gives
\[
    \|u_n\|_{L^{p_1}(\BB_{r_1}\cap\Omega)}
    \leq
    C'(0)
    \|u_n\|_{W^{2,p_0}(\BB_{r_1}\cap\Omega)}
    \leq
    C'(0)C(0).
\]
If $p_1>d/2$, then applying
\eqref{eq:regularity-green-iteration} once more and using the Sobolev embedding theorem
gives
\[
    \|u_n\|_{\infty,\BB_{r_2}\cap\Omega}\leq C.
\]
If not, we repeat the argument with $p_1$ instead of $p_0$ and $p_2$ instead of $p_1$.
Iterating, since $p_m$ eventually becomes larger than $d/2$, there exists
$m\geq0$ such that $p_m>d/2$, and this gives
\[
    \|u_n\|_{\infty,\BB_{2r}(z_0)\cap\Omega}
    \leq
    C.
\]
Combining this estimate with the Schauder estimate gives
\[
    \|u_n\|_{2,\alpha,\BB_r(z_0)\cap\Omega}
    \leq
    C,
\]
where $C$ is independent of $n$. Since $K$ is compact, there exist
$z_1,\ldots,z_N\in K$ such that
\[
    K\subset
    \bigcup_{j=1}^N\BB_{r_j/2}(z_j),
    \qquad
    \sup_n
    \|u_n\|_{2,\alpha,
    \BB_{r_j}(z_j)\cap\Omega}
    \leq C_j.
\]
Set $  C_0:=\max_{1\leq j\leq N}C_j.$ 
By the Lebesgue number lemma, there exists $\lambda>0$ such that
every pair $x,y\in K$ satisfying $\|x-y\|<\lambda$ is contained in
one of the sets $\BB_{r_j/2}(z_j)$. Therefore, for such $x$ and $y$,
\[
    |D^2u_n(x)-D^2u_n(y)|
    \leq
    C_0\|x-y\|^\alpha.
\]
On the other hand, if $\|x-y\|\geq\lambda$, then
\[
\frac{|D^2u_n(x)-D^2u_n(y)|}{\|x-y\|^\alpha}
\leq
\frac{2\|D^2u_n\|_{\infty,K}}{\lambda^\alpha}
\leq
\frac{2C_0}{\lambda^\alpha}.
\]
Consequently,
\[
    [D^2u_n]_{\alpha;\BB\cap\Omega}
    \leq
    C_0+\frac{2C_0}{\lambda^\alpha}.
\]
The corresponding supremum bounds for $u_n$, $\nabla u_n$, and
$D^2u_n$ follow directly from the finite local cover. Hence
\[
    \sup_n
    \|u_n\|_{2,\alpha,\BB\cap\Omega}
    \leq C,
\]
where $C$ is independent of $n$.

It remains to pass to the limit. Fix $0<\beta<\alpha$. By the
compact embedding $ \cC^{2,\alpha}
    \hookrightarrow
    \cC^{2,\beta}$ 
on each of the finitely many local patches, a subsequence of
$\{u_n\}_n$ converges in $\cC^{2,\beta}$ on every patch to a
function $u$. On the other hand,
\[
    u_n\longrightarrow G(x_0,\cdot)
    \qquad\text{strongly in }L^{p'}(\Omega)
\]
by \cref{lemma:strong-convergence}. Hence
\[
    u=G(x_0,\cdot)
    \qquad\text{a.e.~on }\BB\cap\Omega.
\]
The uniform $\cC^{2,\alpha}$ estimates and lower semicontinuity of
the $\alpha$-Hölder seminorm show that
\[
    u\in
    \cC^{2,\alpha}(\overline{\BB\cap\Omega}).
\]
Thus $G(x_0,\cdot)$ agrees a.e.~on $\BB\cap\Omega$ with a function
in $\cC^{2,\alpha}(\overline{\BB\cap\Omega})$. This proves~(i).

{Proof of \textit{(ii).}}
We first record the consequence of the computation above. For all smooth centered
$f,g\in L_0^p(\Omega)$,
\begin{equation}
\label{eq:self-adjoint-green}
    \int_\Omega\cL_A^{-1}(f)g\dd z
    =
    \int_\Omega f\cL_A^{-1}(g)\dd z.
\end{equation}
Indeed, this follows from the same integration by parts and from the symmetry of $A$.

Fix $x,y\in\Omega$ with $x\neq y$. Choose
$0<\delta<|x-y|/4$ such that
$\BB_\delta(x)\Subset\Omega$ and $\BB_\delta(y)\Subset\Omega$. Let
$\{\rho_n\}_n$ and $\{\mu_n\}_n$ be sequences of smooth probability densities such that,
when $\varepsilon_n\downarrow0$,
\[
    {\rm supp}(\rho_n)\subset\BB_{\varepsilon_n}(x)
    \quad\text{and}\quad
    {\rm supp}(\mu_n)\subset\BB_{\varepsilon_n}(y),
\]
and such that
$\rho_n\overset{\ast}{\rightharpoonup}\delta_x$ and
$\mu_n\overset{\ast}{\rightharpoonup}\delta_y$ in
$\mathcal{M}(\Omega)$. For $n$ large enough, the supports of $\rho_n$ and $\mu_n$ are
separated. Set
\[
    u_n:=\cL_A^{-1}(\rho_n-1),
    \qquad
    v_n:=\cL_A^{-1}(\mu_n-1).
\]
Applying \eqref{eq:self-adjoint-green} with
$f=\rho_n-1$ and $g=\mu_n-1$, we get
\[
    \int_\Omega u_n(\mu_n-1)\dd z
    =
    \int_\Omega v_n(\rho_n-1)\dd z.
\]
Since
$\int_\Omega u_n\dd z=\int_\Omega v_n\dd z=0$, this gives
\begin{equation}\label{eq:LHS=RHS}
    \int_\Omega u_n\mu_n\dd z
    =
    \int_\Omega v_n\rho_n\dd z.
\end{equation}
Arguing as in the proof of part~(i), we have
\[
    u_n\longrightarrow G(x,\cdot)
    \quad\text{in}\quad
    \cC^{2,\beta}(\overline{\BB_\delta(y)})
\]
for every $0<\beta<\alpha$. Hence
\begin{equation}\label{eq:LHS}
    \int_\Omega u_n(z)\mu_n(z)\dd z
    \longrightarrow
    G(x,y).
\end{equation}
Similarly,
\begin{equation}\label{eq:RHS}
    \int_\Omega v_n(z)\rho_n(z)\dd z
    \longrightarrow
    G(y,x).
\end{equation}
Letting $n\to\infty$ in~\eqref{eq:LHS} and~\eqref{eq:RHS}, and using
\eqref{eq:LHS=RHS}, we conclude that
\[
    G(x,y)=G(y,x).
\]
Hence, for the representatives obtained in~(i), the symmetry holds whenever $x\neq y$.
In particular,
\[
    G(x,y)=G(y,x)
    \qquad\text{for a.e. }(x,y)\in\Omega^2.
\]
This proves~(ii). Combining~(i) and~(ii) shows~(iii). 
It remains to prove the uniform estimates. Let
$1<q<d/(d-2)$ when $d\geq3$, or $1<q<\infty$ when $d=2$, and take
$p=q'>d/2$. By the boundedness of $(\mathcal{L}_A^*)^{-1}$,
\[
    \|G(x,\cdot)\|_{L^q(\Omega)}
    \leq
    C_q\|\delta_x-1\|_{E_p^*}.
\]
Morrey's inequality and the mean-zero normalization give
\[
    \sup_{x\in\Omega}\|\delta_x-1\|_{E_p^*}<\infty.
\]
Therefore,
\[
    \sup_{x\in\Omega}
    \|G(x,\cdot)\|_{L^q(\Omega)}
    <\infty.
\]
The corresponding estimate in the first variable follows from the symmetry proved
in~(ii).

Finally, fix $\eta>0$ and assume that
\[
    \operatorname{dist}
    \bigl(x_0,\overline{\BB\cap\Omega}\bigr)
    \geq\eta.
\]
In the local argument above, the radii may then be chosen uniformly
from a fixed finite collection of interior balls and boundary
coordinate neighborhoods, with radii bounded below by a constant
depending only on $\eta$ and $\Omega$. The uniform $L^q$ estimate
proved above supplies a uniform initial bound in the Sobolev
iteration. Consequently, all interior, boundary, Sobolev, and
Schauder constants are uniform in the pole and the ball. This proves
the uniform assertion in part~(i). The assertion in part~(iii)
follows from the symmetry established in part~(ii).
\end{proof}
\section{Estimates of the Green function}\label{sec:green-function}

\subsection{Interior comparison}

Fix $y_0\in {\rm int}(\Omega)$ and write $A_0=A(y_0)$. 
The fundamental solution of
$\mathcal{L}_{0}(u)= \operatorname{div}(A_0\nabla u)$ is 
\begin{equation*}\label{eq:Phi-x0}
    \Phi_{A(y_0)}(z, y_0)=
\begin{cases}
\displaystyle 
\frac{1}{(d-2)\,\mathcal{H}^{d-1}(\mathcal{S}^{d-1})}\,
\frac{1}{\sqrt{\det(A_0)}}\,
\Big( \big\langle  z-y_0, A_0^{-1} ( z-y_0)\big\rangle \Big)^{\frac{2-d}{2}},&
d\ge3, \\[1.4em]
\displaystyle 
-\frac{1}{4\pi\sqrt{\det(A_0)}}\,
\log\!\Big( \big\langle z-y_0, A_0^{-1} ( z-y_0)\big\rangle \Big),
& d=2 .
\end{cases}
\end{equation*}
That is,
$\cL_0(\Phi_{A_0}(\cdot, y_0))=-\delta_{y_0}$.
The following interior comparison lemma holds.

\begin{lemma}[Interior comparison]\label{lemma:interior-comparison} 
Assume that $A$ is $\cC^{1,\alpha}$. There exists a constant $C>0$
such that, whenever
\[
    \BB_{4R}(y)\subset\Omega,
\]
for every multi-index $\beta$ with $|\beta|\in\{1,2\}$ and every
$x$ satisfying $0<\|x-y\|\leq2R$,
\[
    \left|
        \partial_x^\beta\Phi_{A(y)}(x,y)
        +
        \partial_x^\beta G(x,y)
    \right|
    \leq
    \frac{C}
    {\|x-y\|^{d-2+|\beta|-\alpha}}.
\]
The constant depends only on $d,\alpha$, the ellipticity constant,
$\|A\|_{1,\alpha;\Omega}$ and $\Omega$, and is independent of
$x,y$ and $R$. 
\end{lemma}

\begin{proof} 
Apply the interior comparison estimate in
\citet[p.~62]{miranda1970partial} after rescaling
$\BB_{4R}(y)$ onto $\BB_4$. More precisely, the rescaled coefficient
\[
    A_R(z):=A(y+Rz)
\]
has the same ellipticity constant as $A$, and its
$\cC^{1,\alpha}$ norm is bounded in terms of
$\|A\|_{1,\alpha;\Omega}$. The term generated by the normalization
of the Green function is smooth and is absorbed in the estimate.
Scaling back, derivatives of order $|\beta|$ contribute the factor
$R^{2-d-|\beta|}$, while the comparison remainder gains the factor
$\|x-y\|^\alpha$. This gives
\[
    \left|
        \partial_x^\beta\Phi_{A(y)}(x,y)
        +
        \partial_x^\beta G(x,y)
    \right|
    \lesssim
    \|x-y\|^{2-d-|\beta|+\alpha},
\]
which is the desired estimate. 
\end{proof}

\subsection{Boundary comparison}

To derive boundary estimates, we use the classical method of
\emph{freezing coefficients} over the flattened domain. This approach
reduces the local problem to a constant-coefficient Neumann problem
on a half-space, often called the \emph{frozen-coefficients} problem,
whose solution admits an explicit representation. The method goes
back at least to \cite[p.~62]{miranda1970partial}, where the author
derives the interior estimates; it is also commonly viewed as a
\emph{perturbation argument} around the corresponding
constant-coefficient equation
\cite[see, e.g.,][]{HanLinPDE2011Notes,GilbargTrudinger.Book}.

For the sake of completeness, we describe the procedure
for deriving the boundary estimates. Let $x_0\in\partial\Omega$.
Since $\partial\Omega$ is of class $\mathcal C^{2,\alpha}$, there
exists a neighborhood $U(x_0)$ of $x_0$, so that after a rigid change
of coordinates, there exists a $\mathcal C^{2,\alpha}$
diffeomorphism
\[
    \Psi:U(x_0)\cap\Omega\longrightarrow\mathbb B_1^+
\]
with $\mathcal C^{2,\alpha}$ inverse such that
\[
    \Psi(U\cap\partial\Omega)
    =
    \{z_d=0\}\cap\mathbb B_1 .
\]
Write $z=\Psi(x)$. Recall that
$u\in E_p$ is the solution to~\eqref{eq:Lp-system}.
Set $v(z)=u(\Psi^{-1}(z))$ and
$J(x)=\det D\Psi(x)$. Then $v$ satisfies
\begin{equation}\label{eq:flattened-pde-interior}
    \operatorname{div}_z(\widetilde A\nabla_zv)=\widetilde f
    \quad\text{in }\mathbb B_1^+,
\end{equation}
where
\[
\widetilde A(z)
=
\frac{1}{J(\Psi^{-1}(z))}
D\Psi(\Psi^{-1}(z))
A(\Psi^{-1}(z))
D\Psi(\Psi^{-1}(z))^{\mathsf T}
\]
and
\[
\widetilde f(z)
=
\frac{1}{J(\Psi^{-1}(z))}
f(\Psi^{-1}(z)).
\]
The boundary condition becomes
\begin{equation}\label{eq:flattened-pde-boundary}
    \langle\widetilde A\nabla_zv,e_d\rangle=0
    \quad\text{on }\{z_d=0\}\cap\mathbb B_1 .
\end{equation}
Here, as $A$ is of class $\cC^{1,\alpha}$ and
$\Psi,\Psi^{-1}$ are $\cC^{2,\alpha}$,
$\widetilde A$ is also $\cC^{1,\alpha}$.
The system defined in
\eqref{eq:flattened-pde-interior}--\eqref{eq:flattened-pde-boundary}
is the \emph{flattened variable-coefficient problem} on the half-ball
$\mathbb B_1^+$. By the same argument as in
\cref{rmk:green-function},
\[
    \widetilde G(z,y)
    =
    G(\Psi^{-1}(z),\Psi^{-1}(y))
\]
is the Green function associated with the operator
$v\mapsto{\rm div}(\widetilde A\nabla_zv)$. That is,
\begin{equation}\label{eq:green-function-transformed}
    {\rm div}
    \bigl(\widetilde A\nabla_z\widetilde G(\cdot,y)\bigr)
    =
    \delta_y-\frac{1}{J(\Psi^{-1}(z))}
\end{equation}
in the distributional sense.

We then introduce the frozen-coefficients problem. For any
$y\in\BB_1^+$,\footnote{We keep the same notation $y$ to emphasize
that the pole is generated by $y$.}
set $A_0:=\widetilde A(y)$. Define the reflection of $y$ by
\[
    y^\ast
    =
    y-2\,\frac{y_d}{\langle e_d,A_0e_d\rangle}A_0e_d .
\]
Then $(y^\ast)_d=-y_d$ and
\[
    \Gamma_{A_0}(z,y)
    =
    \Phi_{A_0}(z,y)+\Phi_{A_0}(z,y^\ast)
\]
satisfies
\begin{equation}\label{eq:freezing-coeff-PDE}
\begin{cases}
\operatorname{div}_z(A_0\nabla_z\Gamma_{A_0}(\cdot,y))
=
-\delta_y
&\text{in }\BB_1^+,\\[0.3em]
\langle A_0\nabla_z\Gamma_{A_0}(\cdot,y),e_d\rangle=0
&\text{on }\{z_d=0\}\cap\BB_1.
\end{cases}
\end{equation}
We note that
$\operatorname{div}_z(A_0\nabla_z\Phi_{A_0}(\cdot,y))
=-\delta_y$ in $\BB_1^+$ and the reflected part
$\Phi_{A_0}(\cdot,y^\ast)$ is introduced solely to correct the
boundary condition. In what follows, we compare $\widetilde G$ with
$\Gamma_{A_0}(\cdot,y)$.

\begin{lemma}[Boundary comparison]\label{lemma:boundary-comparison}
There exists a constant $C>0$, independent of $x_0$, such that for
any multi-index $\beta$ with $|\beta|\in\{1,2\}$ and any
$(z,y)\in\mathbb B_{1/2}^+\times\mathbb B_{1/2}^+$, $z\neq y$,
\begin{equation}\label{eq:Green-function-boundary}
    \left|
        \partial_z^\beta\Gamma_{\widetilde A(y)}(z,y)
        +
        \partial_z^\beta
        G(\Psi^{-1}(z),\Psi^{-1}(y))
    \right|
    \leq
    \frac{C}{\|z-y\|^{d-2+|\beta|-\alpha}} .
\end{equation}
As a consequence, for any
$x,y\in\Psi^{-1}(\BB_{1/2}^+)$, $x\neq y$,
\begin{equation}\label{eq:Green-function-boundary-pull-back}
    \left|
        \partial_x^\beta
        \Gamma_{\widetilde A(\Psi(y))}
        (\Psi(x),\Psi(y))
        +
        \partial_x^\beta G(x,y)
    \right|
    \leq
    \frac{C_1}{\|x-y\|^{d-2+|\beta|-\alpha}},
\end{equation}
where the constant $C_1>0$ depends on $C$ and $\Omega$.
\end{lemma}

\begin{proof}
Note that~\eqref{eq:Green-function-boundary-pull-back} follows
from~\eqref{eq:Green-function-boundary}, the chain rule and the fact
that $\Psi$ is $\cC^{2,\alpha}$. We now establish
\eqref{eq:Green-function-boundary}.

In the sequel, we use $z$ to denote the variable in the PDE and
$z^*$ to denote the point in the statement. Then
$z^*,y\in\BB_{1/2}^+$ with $z^*\neq y$. Write
\[
    z^*=(z^*_{-d},z_d^*)
    \qquad\text{and}\qquad
\bar z:=(z^*_{-d},0).
\]
Set $R=\|z^*-y\|$. We may also assume that
$R\leq1/4$, in which case $y\in\BB_{3/4}^+$. Otherwise,
by~\cref{Proposition:green-function-qualitative}, the derivatives
of $G(\Psi^{-1}(\cdot),\Psi^{-1}(y))$ are uniformly bounded away
from the pole. The same is true for
$\Gamma_{\widetilde A(y)}(\cdot,y)$ by its explicit formula, and
the result follows immediately.

Define
\[
    \widetilde G(z,y)
    :=
    G(\Psi^{-1}(z),\Psi^{-1}(y))
\]
and
\[
    \xi(z)
    :=
    \Gamma_{\widetilde A(y)}(z,y)
    +
    \widetilde G(z,y).
\]

We first make precise the equations satisfied by these functions.
Let $D\subset\mathbb R^d$ be open, let
$u\in L^1_{\mathrm{loc}}(D)$, and let
$T\in\mathcal D'(D)$, where $\mathcal D'(D)
    :=
    \bigl(\mathcal C_c^\infty(D)\bigr)*$
denotes the space of distributions on $D$. We say that
\[
    \mathcal L_{\widetilde A}u=T
    \qquad\text{in }\mathcal D'(D)
\]
if
\[
    \int_D
    u\,\mathcal L_{\widetilde A}\rho\,\rd z
    =
    T(\rho)
\]
for every $\rho\in\mathcal C_c^\infty(D)$.
In particular, if $T=f+\lambda\delta_y$, where
$f\in L^1_{\mathrm{loc}}(D)$, then this means
\[
    \int_D
    u\,\mathcal L_{\widetilde A}\rho\,\rd z
    =
    \int_D f\rho\,\rd z
    +
    \lambda\rho(y).
\]
With this convention,
\eqref{eq:green-function-transformed} means that
\begin{equation}\label{eq:distribution-G-transformed}
\int_{\BB_1^+}
\widetilde G(z,y)
\mathcal L_{\widetilde A}\rho(z)\,\rd z
=
\rho(y)
-
\int_{\BB_1^+}
\frac{\rho(z)}
{J(\Psi^{-1}(z))}\,\rd z
\end{equation}
for every $\rho\in\mathcal C_c^\infty(\BB_1^+)$. Likewise,
\eqref{eq:freezing-coeff-PDE} means that
\begin{equation}\label{eq:distribution-Gamma}
\int_{\BB_1^+}
\Gamma_{\widetilde A(y)}(z,y)
\mathcal L_{\widetilde A(y)}\rho(z)\,\rd z
=
-\rho(y)
\end{equation}
for every $\rho\in\mathcal C_c^\infty(\BB_1^+)$. Set
\[
    F(z,y)
    :=
    \bigl(\widetilde A(z)-\widetilde A(y)\bigr)
    \nabla_z\Gamma_{\widetilde A(y)}(z,y).
\]
The explicit formula for
$\Gamma_{\widetilde A(y)}$ and the
$\mathcal C^{1,\alpha}$ regularity of $\widetilde A$ give
\[
    |F(z,y)|
    \lesssim
    \|z-y\|^{2-d}.
\]
Thus $F(\cdot,y)\in L^1_{\mathrm{loc}}(\BB_1^+)$. Away from $y$,
define
\[
    c(z,y):={\rm div}_z F(z,y).
\]
A direct computation gives
\[
    |c(z,y)|
    \lesssim
    \|z-y\|^{1-d}.
\]
Moreover, $c$ represents the distributional divergence of $F$ on
all of $\BB_1^+$. Indeed, for
$\rho\in\mathcal C_c^\infty(\BB_1^+)$, integration by parts on
$\BB_1^+\setminus\overline{\BB_\varepsilon(y)}$ gives
\[
\begin{aligned}
-\int_{\BB_1^+\setminus\BB_\varepsilon(y)}
F(z,y)\cdot\nabla\rho(z)\,\rd z
&=
\int_{\BB_1^+\setminus\BB_\varepsilon(y)}
c(z,y)\rho(z)\,\rd z
\\
&\quad+
\int_{\partial\BB_\varepsilon(y)}
\rho(z)F(z,y)\cdot
\frac{z-y}{\|z-y\|}\,\rd S.
\end{aligned}
\]
The last term tends to zero because
\[
\int_{\partial\BB_\varepsilon(y)}
|F(z,y)|\,\rd S
\lesssim
\varepsilon^{2-d}\varepsilon^{d-1}
=
\varepsilon.
\]
Letting $\varepsilon\downarrow0$ therefore yields
\begin{equation}\label{eq:distribution-div-F}
    -\int_{\BB_1^+}
    F(z,y)\cdot\nabla\rho(z)\,\rd z
    =
    \int_{\BB_1^+}
    c(z,y)\rho(z)\,\rd z.
\end{equation}

We now compute the equation satisfied by $\xi$. For every
$\rho\in\mathcal C_c^\infty(\BB_1^+)$,
\begin{align*}
&
\int_{\BB_1^+}
\Gamma_{\widetilde A(y)}(z,y)
\mathcal L_{\widetilde A}\rho(z)\,\rd z
\\
&\quad=
\int_{\BB_1^+}
\Gamma_{\widetilde A(y)}(z,y)
\mathcal L_{\widetilde A(y)}\rho(z)\,\rd z
-
\int_{\BB_1^+}
F(z,y)\cdot\nabla\rho(z)\,\rd z
\\
&\quad=
-\rho(y)
+
\int_{\BB_1^+}
c(z,y)\rho(z)\,\rd z,
\end{align*}
where we used
\eqref{eq:distribution-Gamma} and
\eqref{eq:distribution-div-F}. Combining this identity with
\eqref{eq:distribution-G-transformed}, we obtain
\[
    \int_{\BB_1^+}
    \xi(z)\mathcal L_{\widetilde A}\rho(z)\,\rd z
    =
    \int_{\BB_1^+}
    a(z,y)\rho(z)\,\rd z,
\]
where
\[
    a(z,y)
    :=
    -\frac{1}{J(\Psi^{-1}(z))}
    +
    c(z,y).
\]
Thus
\begin{equation}\label{eq:xi-distribution-equation}
    \mathcal L_{\widetilde A}\xi
    =
    a(\cdot,y)
    \qquad\text{in }\mathcal D'(\BB_1^+).
\end{equation}
In particular, the two Dirac masses cancel at the level of the
distributional identities.

On the flat boundary, both terms defining $\xi$ are smooth away
from the interior point $y$, and their conormal derivatives give
\[
    \left\langle
        \widetilde A(z)\nabla\xi(z),e_d
    \right\rangle
    =
    \left\langle
        \bigl(\widetilde A(z)-\widetilde A(y)\bigr)
        \nabla_z\Gamma_{\widetilde A(y)}(z,y),
        e_d
    \right\rangle
    =:b(z,y).
\]
Consequently,
\begin{equation}\label{eq:growth-control-a-b}
    |a(z,y)|
    \lesssim
    \frac{1}{\|z-y\|^{d-1}},
    \qquad
    |b(z,y)|
    \lesssim
    \frac{1}{\|z-y\|^{d-2}},
    \qquad
    |\nabla_zb(z,y)|
    \lesssim
    \frac{1}{\|z-y\|^{d-1}}.
\end{equation}

Choose a $\mathcal C^\infty$ domain $\Omega''$ such that
\[
    \BB_{13/16}^+
    \subset
    \Omega''
    \subset
    \BB_{7/8}^+,
\]
and such that its boundary agrees with $\{z_d=0\}$ in
$\BB_{13/16}$. On the whole boundary of $\Omega''$, define
\[
    \widehat b(z,y)
    :=
    \mathcal B_{\widetilde A}\xi(z)
    =
    \left\langle
        \widetilde A(z)\nabla\xi(z),
        \nu_{\Omega''}(z)
    \right\rangle.
\]
On the flat part, where $\nu_{\Omega''}=-e_d$, one has
$\widehat b=-b$. On the artificial part of the boundary, and in a
fixed neighborhood of the junction between the flat and artificial
parts, the distance to $y\in\BB_{3/4}^+$ is bounded from below.
Hence,
\cref{Proposition:green-function-qualitative} and the explicit
formula for $\Gamma_{\widetilde A(y)}$ show that $\widehat b$ is
uniformly bounded in $\mathcal C^1$ there. On the flat part, \eqref{eq:growth-control-a-b} and scaling in
$\mathbb R^{d-1}$ give
\[
    \|b(\cdot,y)\|_
    {W^{1-1/p,p}
    (\partial\Omega''\cap\{z_d=0\})}
    \leq C,
    \qquad
    1<p<\frac{d}{d-1}.
\]
Using boundary charts and a smooth partition of unity, we therefore obtain
\begin{equation}\label{eq:a-b-global-bound}
    \|a(\cdot,y)\|_{L^p(\Omega'')}
    +
    \|\widehat b(\cdot,y)\|_
    {W^{1-1/p,p}(\partial\Omega'')}
    \leq C,
    \qquad
    1<p<\frac{d}{d-1},
\end{equation}
uniformly in $y$. We next record the weak identity on $\Omega''$ that incorporates
both the distributional equation and the boundary condition. Let
$p'=p/(p-1)$. We claim that
\begin{equation}\label{eq:weak-xi-Neumann}
\int_{\Omega''}
\xi\,\mathcal L_{\widetilde A}\rho\,\rd z
=
\int_{\Omega''}
a(\cdot,y)\rho\,\rd z
-
\int_{\partial\Omega''}
\widehat b(\cdot,y)\rho\,\rd S
\end{equation}
for every $\rho\in W^{2,p'}(\Omega'')$ satisfying
\[
    \mathcal B_{\widetilde A}\rho=0
    \qquad\text{on }\partial\Omega''.
\]
To prove \eqref{eq:weak-xi-Neumann}, choose
$\chi\in\mathcal C_c^\infty(\Omega'')$ such that $\chi=1$ in a
neighborhood of $y$, and write
\[
    \rho=\rho_0+\rho_1,
    \qquad
    \rho_0:=\chi\rho,
    \qquad
    \rho_1:=(1-\chi)\rho.
\]
Since $\rho_0$ has compact support in $\Omega''$,
\eqref{eq:xi-distribution-equation} gives
\begin{equation}\label{eq:weak-xi-rho0}
    \int_{\Omega''}
    \xi\,\mathcal L_{\widetilde A}\rho_0\,\rd z
    =
    \int_{\Omega''}
    a(\cdot,y)\rho_0\,\rd z.
\end{equation}
This follows first for smooth compactly supported functions and
then for $\rho_0\in W^{2,p'}_0(\Omega'')$ by approximation, since
$\xi,a(\cdot,y)\in L^p(\Omega'')$.

The function $\rho_1$ vanishes in a neighborhood of $y$. By
\cref{Proposition:green-function-qualitative} and the explicit
formula for $\Gamma_{\widetilde A(y)}$, the function $\xi$ is
$\mathcal C^{2,\alpha}$ on a neighborhood of the support of
$\rho_1$. The second Green identity therefore gives
\[
\begin{aligned}
\int_{\Omega''}
\xi\,\mathcal L_{\widetilde A}\rho_1\,\rd z
&=
\int_{\Omega''}
\rho_1\,\mathcal L_{\widetilde A}\xi\,\rd z+
\int_{\partial\Omega''}
\left\{
    \xi\mathcal B_{\widetilde A}\rho_1
    -
    \rho_1\mathcal B_{\widetilde A}\xi
\right\}\rd S.
\end{aligned}
\]
The identity for $\rho_1\in W^{2,p'}$ follows by smooth
approximation and continuity of the trace maps. Because $\chi$ is
compactly supported in $\Omega''$, it vanishes near
$\partial\Omega''$. Hence
\[
    \rho_1=\rho,
    \qquad
    \mathcal B_{\widetilde A}\rho_1
    =
    \mathcal B_{\widetilde A}\rho
    =0
    \quad\text{on }\partial\Omega''.
\]
Since
\[
    \mathcal L_{\widetilde A}\xi=a(\cdot,y)
    \quad\text{away from }y,
    \qquad
    \mathcal B_{\widetilde A}\xi=\widehat b(\cdot,y),
\]
we obtain
\begin{equation}\label{eq:weak-xi-rho1}
    \int_{\Omega''}
    \xi\,\mathcal L_{\widetilde A}\rho_1\,\rd z
    =
    \int_{\Omega''}
    a(\cdot,y)\rho_1\,\rd z
    -
    \int_{\partial\Omega''}
    \widehat b(\cdot,y)\rho\,\rd S.
\end{equation}
Adding \eqref{eq:weak-xi-rho0} and
\eqref{eq:weak-xi-rho1} proves
\eqref{eq:weak-xi-Neumann}.

Taking $\rho=1$ in \eqref{eq:weak-xi-Neumann} gives the
compatibility relation
\begin{equation}\label{eq:compatibility-xi}
    \int_{\Omega''}a(\cdot,y)\,\rd z
    =
    \int_{\partial\Omega''}
    \widehat b(\cdot,y)\,\rd S,
\end{equation}
because $\mathcal L_{\widetilde A}1=0
 $ and $
    \mathcal B_{\widetilde A}1=0.$
We now remove the nonhomogeneous boundary condition. Uniform
ellipticity implies that $ \left\langle
        \widetilde A\nu_{\Omega''},\nu_{\Omega''}
    \right\rangle$ 
is bounded away from zero. Hence
\[
    \frac{\widehat b}
    {\langle
        \widetilde A\nu_{\Omega''},\nu_{\Omega''}
    \rangle}
    \in
    W^{1-1/p,p}(\partial\Omega'')
\]
with uniformly bounded norm. By the higher-order trace theorem
\cite[Theorem~7.8]{Taira.2024.book}, there exists
$h\in W^{2,p}(\Omega'')$ such that
\[
    h=0,
    \qquad
    \partial_{\nu_{\Omega''}}h
    =
    \frac{\widehat b}
    {\langle
        \widetilde A\nu_{\Omega''},\nu_{\Omega''}
    \rangle}
    \quad\text{on }\partial\Omega'',
\]
and $\|h\|_{W^{2,p}(\Omega'')}\leq C.$ 
Since the trace of $h$ vanishes, its tangential gradient vanishes
on $\partial\Omega''$. Thus
\[
    \nabla h
    =
    \partial_{\nu_{\Omega''}}h\,\nu_{\Omega''}
    \qquad\text{on }\partial\Omega'',
\]
and consequently
\[
    \mathcal B_{\widetilde A}h
    =
    \widehat b
    \qquad\text{on }\partial\Omega''.
\]
By \eqref{eq:compatibility-xi} and the divergence theorem applied
to $h$,
\[
\begin{aligned}
\int_{\Omega''}
\left(
    a-\mathcal L_{\widetilde A}h
\right)\rd z
&=
\int_{\partial\Omega''}
\left(
    \widehat b-\mathcal B_{\widetilde A}h
\right)\rd S
=0.
\end{aligned}
\]
Therefore, by \cref{theorem:solvable-lp}, there exists a unique
mean-zero function $v\in W^{2,p}(\Omega'')$ satisfying
\[
\begin{cases}
\mathcal L_{\widetilde A}v
=
a-\mathcal L_{\widetilde A}h
&\text{in }\Omega'',\\
\mathcal B_{\widetilde A}v=0
&\text{on }\partial\Omega'',
\end{cases}
\]
and $ \|v\|_{W^{2,p}(\Omega'')}\leq C.$
Set $    w:=h+v.$ 
Then $w\in W^{2,p}(\Omega'')$ and
\[
\begin{cases}
\mathcal L_{\widetilde A}w=a(\cdot,y)
&\text{in }\Omega'',\\
\mathcal B_{\widetilde A}w=\widehat b(\cdot,y)
&\text{on }\partial\Omega''.
\end{cases}
\]
For every
$\rho\in W^{2,p'}(\Omega'')$ satisfying
$\mathcal B_{\widetilde A}\rho=0$, the second Green identity gives
\begin{equation}\label{eq:weak-w-Neumann}
    \int_{\Omega''}
    w\,\mathcal L_{\widetilde A}\rho\,\rd z
    =
    \int_{\Omega''}
    a(\cdot,y)\rho\,\rd z
    -
    \int_{\partial\Omega''}
    \widehat b(\cdot,y)\rho\,\rd S.
\end{equation}
Subtracting \eqref{eq:weak-w-Neumann} from
\eqref{eq:weak-xi-Neumann}, we obtain
\begin{equation}\label{eq:xi-w-annihilates-range}
    \int_{\Omega''}
    (\xi-w)\mathcal L_{\widetilde A}\rho\,\rd z
    =0
\end{equation}
for every
$\rho\in W^{2,p'}(\Omega'')$ satisfying
$\mathcal B_{\widetilde A}\rho=0$. Let now $g\in L^{p'}_0(\Omega'')$. By
\cref{theorem:solvable-lp}, there exists a mean-zero
$\rho\in W^{2,p'}(\Omega'')$ satisfying
\[
\begin{cases}
\mathcal L_{\widetilde A}\rho=g
&\text{in }\Omega'',\\
\mathcal B_{\widetilde A}\rho=0
&\text{on }\partial\Omega''.
\end{cases}
\]
Using this function in
\eqref{eq:xi-w-annihilates-range} yields
\[
    \int_{\Omega''}(\xi-w)g\,\rd z=0
    \qquad
    \text{for every }g\in L^{p'}_0(\Omega'').
\]
Thus $\xi-w$ belongs to the annihilator of
$L^{p'}_0(\Omega'')$ in $L^p(\Omega'')$, which is the space of
constant functions. Hence there exists $c\in\mathbb R$ such that
\[
    \xi-h-v=c
    \qquad\text{a.e. in }\Omega''.
\]
Consequently,
\begin{equation}\label{eq:bounded:LP-boundary}
\begin{aligned}
    \inf_{c\in\mathbb R}
    \|\xi-c\|_{W^{2,p}(\Omega'')}
    &\leq
    \|h+v\|_{W^{2,p}(\Omega'')}
    \\
    &\leq
    \|h\|_{W^{2,p}(\Omega'')}
    +
    \|v\|_{W^{2,p}(\Omega'')}
    \lesssim1,
\end{aligned}
\end{equation}
for every $    1<p<\frac{d}{d-1}.$ 
We consider the following two regimes separately:
\[
    {\rm case\ 1}:\ z_d^*>\frac14R
    \qquad\text{and}\qquad
    {\rm case\ 2}:\ z_d^*\leq\frac14R.
\]

\textit{Case 1.}
Since $z^*\in\BB_{1/2}^+$ and $R\leq1/4$, it follows that
\[
    \BB_{R/4}(z^*)
    \subset\BB_{9/16}^+
    \subset\Omega''.
\]
Moreover, if $z\in\BB_{R/8}(z^*)$, then
\begin{equation}\label{eq:boundary-interior-away-pole}
    \|z-y\|\geq\frac{7R}{8},
    \qquad
    \|z-y^\ast\|\geq cR,
\end{equation}
where the second inequality follows from the definition of the
$A_0$-reflection and uniform ellipticity. Indeed, reflection across
$\{z_d=0\}$ is orthogonal for the metric induced by $A_0^{-1}$, and
the Euclidean and $A_0^{-1}$ norms are uniformly equivalent. 
Thus both poles of
\[
    \Gamma_{\widetilde A(y)}(\cdot,y)
    =
    \Phi_{\widetilde A(y)}(\cdot,y)
    +
    \Phi_{\widetilde A(y)}(\cdot,y^\ast)
\]
are away from $\BB_{R/8}(z^*)$. By Schauder's interior estimates
\cite[Corollary~6.3]{GilbargTrudinger.Book}, for all $i,j\in[d]$,
\begin{equation}\label{eq:boundary-interior-C2alpha}
\begin{aligned}
&R\|\partial_i{\xi}\|_
{\infty,\BB_{R/16}(z^*)}
+
R^2\|\partial_{ij}\xi\|_
{\infty,\BB_{R/16}(z^*)}
\\
&\qquad\lesssim
\inf_{c\in\R}
\|\xi-c\|_{\infty,\BB_{R/8}(z^*)}
+
R^{2+\alpha}
\|a(\cdot,y)\|_{0,\alpha,\BB_{R/8}(z^*)}.
\end{aligned}
\end{equation}
Furthermore, using the weak local maximum principle
\cite[Theorem~9.20]{GilbargTrudinger.Book}, followed by
\eqref{eq:growth-control-a-b} and
\eqref{eq:boundary-interior-away-pole}, we derive
\begin{equation}\label{eq:boundary-interior-infinity}
\begin{aligned}
\inf_{c\in\R}
\|\xi-c\|_{\infty,\BB_{R/8}(z^*)}
&\lesssim
\frac{
\inf_{c\in\R}
\|\xi-c\|_{L^q(\BB_{R/4}(z^*))}
}
{R^{d/q}}
+
R\|a(\cdot,y)\|_{L^d(\BB_{R/8}(z^*))}
\\
&\lesssim
\frac{
\inf_{c\in\R}
\|\xi-c\|_{L^q(\BB_{R/4}(z^*))}
}
{R^{d/q}}
+
\frac{1}{R^{d-3}}.
\end{aligned}
\end{equation}
Fix $c\in\R$. By the Sobolev embedding theorem
\cite[Corollary~9.15]{brezis2011functional}, for $d\geq3$,
\begin{equation}\label{eq:boundary-interior-Sobolev-embedding}
    \|\xi-c\|_{L^q(\BB_{R/4}(z^*))}
    \lesssim
    \|\xi-c\|_{W^{2,p}(\BB_{R/4}(z^*))},
    \qquad
    q<\frac{dp}{d-2p},
\end{equation}
with the convention $1/0=\infty$. Taking
$p<d/(d-1)$ sufficiently close to $d/(d-1)$,
\eqref{eq:bounded:LP-boundary} yields
\begin{equation}\label{eq:xi-q}
    \inf_{c\in\R}
    \|\xi-c\|_{L^q(\BB_{R/4}(z^*))}
    \lesssim1,
    \qquad
    q<\frac{d}{d-3},
    \quad d\geq3.
\end{equation}
Recall that $\alpha\in(0,1)$. Taking
\[
    q=\frac{d}{d-2-\alpha}
\]
in~\eqref{eq:boundary-interior-infinity} gives
\[
    \inf_{c\in\R}
    \|\xi-c\|_{\infty,\BB_{R/8}(z^*)}
    \lesssim
    \frac{1}{R^{d-2-\alpha}}
    +
    \frac{1}{R^{d-3}}
    \lesssim
    \frac{1}{R^{d-2-\alpha}}.
\]
Together with~\eqref{eq:boundary-interior-C2alpha}, this shows
\[
    \|\partial_{ij}\xi\|_
    {\infty,\BB_{R/16}(z^*)}
    \lesssim
    \frac{1}{R^2}
    \frac{1}{R^{d-2-\alpha}}
    +
    \frac{1}{R^{d-1}}
    \lesssim
    \frac{1}{R^{d-\alpha}},
\]
because
\[
    \|a(\cdot,y)\|_
    {0,\alpha,\BB_{R/8}(z^*)}
    \lesssim
    \frac{1}{R^{d-1+\alpha}}.
\]
By the same estimate,
\[
    \|\partial_i\xi\|_
    {\infty,\BB_{R/16}(z^*)}
    \lesssim
    \frac{1}{R^{d-1-\alpha}}.
\]
This concludes the first case when $d\geq3$. 
When $d=2$, choose
\[
    \frac{2}{2-\alpha}<p<2.
\]
By~\eqref{eq:bounded:LP-boundary} and Morrey's inequality,
\[
    [\xi]_{0,\alpha;\Omega''}\lesssim1.
\]
Consequently,
\[
    \inf_{c\in\R}
    \|\xi-c\|_{\infty,\BB_{R/8}(z^*)}
    \lesssim R^\alpha.
\]
Using this estimate in
\eqref{eq:boundary-interior-C2alpha} gives
\[
    \|D^2\xi\|_{\infty,\BB_{R/16}(z^*)}
    \lesssim R^{-2+\alpha},
    \qquad
    \|\nabla\xi\|_{\infty,\BB_{R/16}(z^*)}
    \lesssim R^{-1+\alpha},
\]
which are precisely the desired estimates for $d=2$. 

\textit{Case 2.}
We now focus on the second case. For ease of notation, replace
$\xi$ by $\xi-c_0$, where $c_0$ realizes the minimum in
\eqref{eq:bounded:LP-boundary}. Let $r=R/2$. Since
\[
    \|z^*-\bar z\|
    =
    z_d^*
    \leq
    \frac{R}{4}
    =
    \frac r2,
\]
we have
\[
    z^*\in\BB_{r/2}^+(\bar z).
\]
Moreover,
\[
    \|y-\bar z\|
    \geq
    \|y-z^*\|-\|z^*-\bar z\|
    \geq
    R-\frac R4
    =
    \frac{3R}{4},
\]
and therefore
\begin{equation}\label{eq:case2-distance}
    {\rm dist}
    \left(
        y,\BB_{3r/4}^+(\bar z)
    \right)
    \geq
    \frac{3R}{4}-\frac{3R}{8}
    =
    \frac{3R}{8}.
\end{equation}
Notice also that
$\BB_{3r/4}^+(\bar z)\subset\BB_{5/8}^+$.

Set
\[
    w(\zeta)
    :=
    \xi(\bar z+r\zeta).
\]
Then
\[
    \nabla w(\zeta)
    =
    r(\nabla\xi)(\bar z+r\zeta),
    \qquad
    \nabla^2w(\zeta)
    =
    r^2(\nabla^2\xi)(\bar z+r\zeta).
\]
Moreover,
\[
\left\langle
    \widetilde A(\bar z+r\zeta)
    \nabla w(\zeta),
    e_d
\right\rangle
=
r\,b(\bar z+r\zeta,y)
=:\widetilde b(\zeta,y)
\quad\text{on }T_1
\]
and
\[
{\rm div}
\left(
    \widetilde A(\bar z+r\zeta)
    \nabla w(\zeta)
\right)
=
r^2a(\bar z+r\zeta,y)
=:\widetilde a(\zeta,y)
\quad\text{in }\BB_1^+.
\]
Using the boundary Schauder estimate
\cite[Lemma~6.29]{GilbargTrudinger.Book} from $\BB_{5/8}^+$ to
$\BB_{1/2}^+$, we obtain
\begin{align}
&\|\nabla w\|_{\infty,\BB_{1/2}^+}
+
\|D^2w\|_{\infty,\BB_{1/2}^+}
\notag\\
&\qquad\lesssim
\inf_{c\in\R}
\|w-c\|_{\infty,\BB_{5/8}^+}
+
\|\widetilde a(\cdot,y)\|_
{0,\alpha,\BB_{5/8}^+}
+
\|\widetilde b(\cdot,y)\|_
{1,\alpha,T_{5/8}}.
\label{eq:w-2-alpha}
\end{align}
By~\eqref{eq:case2-distance} and
\eqref{eq:growth-control-a-b},
\[
    \|\widetilde a(\cdot,y)\|_
    {0,\alpha,\BB_{5/8}^+}
    +
    \|\widetilde b(\cdot,y)\|_
    {1,\alpha,T_{5/8}}
    \lesssim
    \frac{1}{R^{d-3}}.
\]

It remains to control
$\inf_c\|w-c\|_{\infty,\BB_{5/8}^+}$.
Set
\[
    \rho_m
    :=
    \frac58+2^{-m-3},
    \qquad m\geq0,
\]
so that $\rho_0=3/4$ and $\rho_m\downarrow5/8$.
Lemma~19.1 in \cite{Taira.2024.book}, followed by the trace theorem
\cite[Theorem~7.7]{Taira.2024.book}, gives, for every $p>1$,
\begin{equation}\label{eq:W2p-w}
\begin{aligned}
\|w\|_{W^{2,p}(\BB_{\rho_{m+1}}^+)}
\lesssim{}&
\|w\|_{L^p(\BB_{\rho_m}^+)}
+
\|\widetilde a(\cdot,y)\|_
{L^p(\BB_{\rho_m}^+)}
\\
&+
\|\widetilde b(\cdot,y)\|_
{W^{1-1/p,p}(T_{\rho_m})}.
\end{aligned}
\end{equation}
All the physical half-balls
$\bar z+r\BB_{\rho_m}^+$ remain at distance comparable to $R$ from
the pole. Hence,
\begin{equation}\label{eq:scaled-data-bound}
    \|\widetilde a(\cdot,y)\|_
    {L^p(\BB_{\rho_m}^+)}
    +
    \|\widetilde b(\cdot,y)\|_
    {W^{1-1/p,p}(T_{\rho_m})}
    \lesssim
    \frac{1}{R^{d-3}}.
\end{equation}

Assume first that $d\geq3$ and put
\[
    p_0
    :=
    \frac{d}{d-2-\alpha}.
\]
By~\eqref{eq:bounded:LP-boundary} and the Sobolev embedding used in
\eqref{eq:xi-q},
\[
    \|\xi\|_{L^{p_0}(\Omega'')}
    \lesssim1.
\]
Consequently,
\[
\begin{aligned}
    \|w\|_{L^{p_0}(\BB_{3/4}^+)}
    &=
    r^{-d/p_0}
    \|\xi\|_{L^{p_0}
    (\BB_{3r/4}^+(\bar z))}
    \\
    &\lesssim
    \frac{1}{R^{d-2-\alpha}}.
\end{aligned}
\]
Apply~\eqref{eq:W2p-w} with $p=p_0$. If $p_0>d/2$, Morrey's
inequality gives the required $L^\infty$ bound on a smaller
half-ball. Otherwise, define
\[
    p_{m+1}
    =
    \frac{dp_m}{d-2p_m}
\]
as long as $p_m<d/2$, choosing any $p_{m+1}>d/2$ when
$p_m=d/2$. At each step, the Sobolev embedding followed by
\eqref{eq:W2p-w} gives
\[
    \|w\|_{L^{p_{m+1}}
    (\BB_{\rho_{m+1}}^+)}
    \lesssim
    \frac{1}{R^{d-2-\alpha}},
\]
where we used
\[
    \frac{1}{R^{d-3}}
    \leq
    \frac{1}{R^{d-2-\alpha}},
    \qquad R\leq1.
\]
After finitely many steps, $p_m>d/2$, and Morrey's inequality yields
\begin{equation}\label{eq:w-infinity}
    \inf_{c\in\R}
    \|w-c\|_{\infty,\BB_{5/8}^+}
    \lesssim
    \frac{1}{R^{d-2-\alpha}}.
\end{equation}
Combining~\eqref{eq:w-2-alpha} and~\eqref{eq:w-infinity}, we obtain
\[
    \|\nabla w\|_{\infty,\BB_{1/2}^+}
    +
    \|D^2w\|_{\infty,\BB_{1/2}^+}
    \lesssim
    \frac{1}{R^{d-2-\alpha}}
    +
    \frac{1}{R^{d-3}}.
\]
Since $r=R/2$,
\begin{align*}
\|D^2\xi\|_{\infty,\BB_{r/2}^+(\bar z)}
&=
r^{-2}
\|D^2w\|_{\infty,\BB_{1/2}^+}
\\
&\lesssim
\frac{1}{R^2}
\left(
    \frac{1}{R^{d-2-\alpha}}
    +
    \frac{1}{R^{d-3}}
\right)
\lesssim
\frac{1}{R^{d-\alpha}},
\end{align*}
and
\begin{align*}
\|\nabla\xi\|_{\infty,\BB_{r/2}^+(\bar z)}
&=
r^{-1}
\|\nabla w\|_{\infty,\BB_{1/2}^+}
\\
&\lesssim
\frac{1}{R}
\left(
    \frac{1}{R^{d-2-\alpha}}
    +
    \frac{1}{R^{d-3}}
\right)
\lesssim
\frac{1}{R^{d-1-\alpha}}.
\end{align*}
Since $z^*\in\BB_{r/2}^+(\bar z)$, this proves the result for
$d\geq3$.

When $d=2$, choose again
\[
    \frac{2}{2-\alpha}<p<2.
\]
By~\eqref{eq:bounded:LP-boundary} and Morrey's inequality,
$[\xi]_{0,\alpha;\Omega''}\lesssim1$. Therefore,
\[
    \inf_{c\in\R}
    \|w-c\|_{\infty,\BB_{5/8}^+}
    \lesssim R^\alpha.
\]
Furthermore, the two data terms in
\eqref{eq:w-2-alpha} are bounded by $CR$, which is at most
$CR^\alpha$ for $0<R\leq1$. It follows that
\[
    \|\nabla w\|_{\infty,\BB_{1/2}^+}
    +
    \|D^2w\|_{\infty,\BB_{1/2}^+}
    \lesssim R^\alpha.
\]
Scaling back gives
\[
    \|D^2\xi\|_{\infty,\BB_{r/2}^+(\bar z)}
    \lesssim
    R^{-2+\alpha},
    \qquad
    \|\nabla\xi\|_{\infty,\BB_{r/2}^+(\bar z)}
    \lesssim
    R^{-1+\alpha}.
\]
This concludes the proof.
\end{proof}

Combining~\cref{lemma:interior-comparison} and
\cref{lemma:boundary-comparison}, we arrive at the following key
result providing the global estimates of the Green function.

\begin{lemma}\label{lemma:global-estimate-green-function}
Under the setting of~\cref{theorem-Solvable}, fix
$x_0\neq y\in\Omega$ and let $R:=\|x_0-y\|$. Then, for any
$x\in\BB_{R/2}(x_0)$,
\[
|\partial_{ij,x}G(x,y)-\partial_{ij,x}G(x_0,y)|
\lesssim
\frac{\|x-x_0\|}{\|x-y\|^{d+1}}
+
\frac{1}{\|x-y\|^{d-\alpha}}
+
\frac{1}{\|x_0-y\|^{d-\alpha}}.
\]
\end{lemma}

\begin{proof}
First observe that, since $x\in\BB_{R/2}(x_0)$,
\begin{equation}\label{eq:distance-comparability-global-G}
    \frac{R}{2}
    \leq
    \|x-y\|
    \leq
    \frac{3R}{2}.
\end{equation}
Moreover, for every point
$x_t:=x_0+t(x-x_0)$, $t\in[0,1]$, we have
\[
    \|x_t-y\|
    \geq
    \|x_0-y\|-\|x_t-x_0\|
    \geq
    R-\frac R2
    =
    \frac R2.
\]

We first reduce to the small-scale case. Fix $R_*>0$ sufficiently
small, depending only on $\Omega$, so that all boundary charts used
in~\cref{lemma:boundary-comparison} are available on balls of radius
comparable to $R_*$. If $R\geq R_*$, then
$\|x-y\|\geq R_*/2$. Hence $x$ and $x_0$ stay a fixed positive
distance away from the pole $y$. By the uniform local regularity in
\cref{Proposition:green-function-qualitative}, together with a finite
covering of $\overline\Omega$, we get
\[
    |\partial_{ij,x}G(x,y)|
    +
    |\partial_{ij,x}G(x_0,y)|
    \lesssim1
    \lesssim
    \frac{1}{\|x-y\|^{d-\alpha}}
    +
    \frac{1}{\|x_0-y\|^{d-\alpha}}.
\]
Henceforth, assume that $0<R<R_*$.

We split the proof into two cases.

\textit{Case 1: the interior case.}
Assume
\[
    R\leq\frac14\operatorname{dist}(y,\partial\Omega).
\]
Then $\BB_{4R}(y)\subset\Omega$, and the points $x_0,x$ satisfy
$\|x_0-y\|,\|x-y\|\leq2R$. Applying
\cref{lemma:interior-comparison}, we obtain, for $z=x,x_0$,
\begin{equation}\label{eq:interior-comparison-error-global}
    \left|
        \partial_{ij,z}G(z,y)
        +
        \partial_{ij,z}\Phi_{A(y)}(z,y)
    \right|
    \lesssim
    \frac{1}{\|z-y\|^{d-\alpha}}.
\end{equation}
The fundamental solution satisfies
\[
    |D_z^3\Phi_{A(y)}(z,y)|
    \lesssim
    \frac{1}{\|z-y\|^{d+1}}.
\]
Therefore, by the mean-value theorem,
\begin{align*}
&|\partial_{ij,x}\Phi_{A(y)}(x,y)
-\partial_{ij,x}\Phi_{A(y)}(x_0,y)|
\\
&\qquad\leq
\|x-x_0\|
\sup_{t\in[0,1]}
|D_z^3\Phi_{A(y)}(x_t,y)|
\\
&\qquad\lesssim
\frac{\|x-x_0\|}{R^{d+1}}
\lesssim
\frac{\|x-x_0\|}{\|x-y\|^{d+1}}.
\end{align*}
Combining this estimate with
\eqref{eq:interior-comparison-error-global} gives
\[
\begin{aligned}
|\partial_{ij,x}G(x,y)-\partial_{ij,x}G(x_0,y)|
\lesssim{}&
\frac{\|x-x_0\|}{\|x-y\|^{d+1}}
\\
&+
\frac{1}{\|x-y\|^{d-\alpha}}
+
\frac{1}{\|x_0-y\|^{d-\alpha}}.
\end{aligned}
\]

\textit{Case 2: the boundary case.}
Assume now
\[
    R>\frac14\operatorname{dist}(y,\partial\Omega).
\]
Choose $\bar y\in\partial\Omega$ such that
$\|y-\bar y\|=\operatorname{dist}(y,\partial\Omega)$. Then
\[
    \|x_0-\bar y\|
    \leq
    \|x_0-y\|+\|y-\bar y\|
    \leq
    5R,
\]
and, using~\eqref{eq:distance-comparability-global-G},
\[
    \|x-\bar y\|
    \leq
    \|x-y\|+\|y-\bar y\|
    \leq
    \frac{3R}{2}+4R
    \leq
    6R.
\]
Since $R<R_*$, after choosing $R_*$ sufficiently small, the points
$x,x_0$ and $y$ all lie in a single boundary coordinate chart. 
Put
\[
    z=\Psi(x),
    \qquad
    z_0=\Psi(x_0),
    \qquad
    \eta=\Psi(y),
\]
and write
\[
    \Gamma(\cdot,\eta)
    :=
    \Gamma_{\widetilde A(\eta)}(\cdot,\eta).
\]
Define 
\[
\begin{aligned}
\mathcal K_{ij}(x,y)
:={}&
\sum_{k,\ell=1}^d
\partial_{k\ell,z}\Gamma(\Psi(x),\eta)
\partial_i\Psi_k(x)\partial_j\Psi_\ell(x)
\\
&+
\sum_{k=1}^d
\partial_{k,z}\Gamma(\Psi(x),\eta)
\partial_{ij}\Psi_k(x).
\end{aligned}
\]
Applying~\cref{lemma:boundary-comparison} at both $x$ and $x_0$
gives
\begin{align}
|\partial_{ij,x}G(x,y)+\mathcal K_{ij}(x,y)|
&\lesssim
\frac{1}{\|x-y\|^{d-\alpha}},
\label{eq:chain-error-bound-global-x}
\\
|\partial_{ij,x}G(x_0,y)+\mathcal K_{ij}(x_0,y)|
&\lesssim
\frac{1}{\|x_0-y\|^{d-\alpha}}.
\label{eq:chain-error-bound-global-x0}
\end{align}
The explicit formula for $\Gamma$ gives
\[
    |D_z^m\Gamma(z,\eta)|
    \lesssim
    \|z-\eta\|^{2-d-m},
    \qquad m=1,2,3.
\]
Using these estimates, the mean-value theorem for the terms
containing $D\Psi$, and the $\alpha$-Hölder continuity of
$D^2\Psi$, we obtain
\[
    |\mathcal K_{ij}(x,y)-\mathcal K_{ij}(x_0,y)|
    \lesssim
    \frac{\|x-x_0\|}{R^{d+1}}
    +
    \frac{1}{R^{d-\alpha}}.
\]
Combining this estimate with
\eqref{eq:chain-error-bound-global-x} and
\eqref{eq:chain-error-bound-global-x0}, and using
\eqref{eq:distance-comparability-global-G}, proves the claim.
The constants are uniform because $\partial\Omega$ is covered by
finitely many boundary charts with uniformly bounded
$\cC^{2,\alpha}$ norms.
\end{proof}
We also record the following consequence, which will be used below.
The explicit estimates for $\Phi$ and $\Gamma$, together with
\cref{lemma:interior-comparison},
\cref{lemma:boundary-comparison} and the uniform regularity away
from the diagonal, give
\begin{equation}\label{eq:green-kernel-size}
    |D_x^\beta G(x,y)|
    \leq
    C\left(
        1+\|x-y\|^{2-d-|\beta|}
    \right),
    \qquad
    |\beta|\in\{1,2\},
    \quad x\neq y.
\end{equation}

\subsubsection{Proof of \cref{theorem-Solvable}}

Since $f\in L_0^\infty(\Omega)$, we know that
$f\in L_0^p(\Omega)$. Therefore, from
\cref{theorem:solvable-lp}, there exists
$u\in W^{2,p}(\Omega)$ for all $p\in(1,\infty)$ with
$\int_\Omega u(x)\dd x=0$ such that
\begin{equation}\label{eq:W2p-estimate}
    \|u\|_{W^{2,p}(\Omega)}
    \leq
    C_p\|f\|_{L^p(\Omega)}.
\end{equation}
Furthermore,
\cref{Proposition:green-function-characterization} yields
\[
    u(\cdot)
    =
    \int_\Omega f(y)G(\cdot,y)\dd y,
\]
where $G$ is the Green function associated with $\cL_A$.

To conclude the statement, we aim to show that
$u\in W^{2,BMO}(\Omega)$. In particular, we will prove that
\begin{equation}\label{eq:BMO}
    \sup_{x_0\in\Omega}
    \sup_{0<r\leq\operatorname{diam}(\Omega)}
    \min_{a\in\R}
    \frac{1}{r^d}
    \int_{\BB_r(x_0)\cap\Omega}
    |\partial_{ij}u(x)-a|\dd x
    \leq
    C\|f\|_{\infty,\Omega}.
\end{equation}
Fix $x_0\in\Omega$ and let
$\BB_r=\BB_r(x_0)$. Put
\[ 
    E:=\BB_{2r}(x_0)\cap\Omega,
    \qquad
    c_r:=\frac{1}{|\Omega|}\int_Ef,
\]
and write $f=f_1+f_2$, where
\[
    f_1:=f{\bf1}_E-c_r,
    \qquad
    f_2:=f-f_1.
\]
Then $f_1,f_2\in L_0^\infty(\Omega)$ and $f_2=c_r$ on $E$.
Writing
\[
    u
    =
    \cL_A^{-1}(f_1)+\cL_A^{-1}(f_2)
    =:u_1+u_2,
\]
we obtain
\begin{align*}
I(x_0,r)
&:=
\min_{a\in\R}
\frac{1}{r^d}
\int_{\BB_r\cap\Omega}
|\partial_{ij}u(x)-a|\dd x
\\
&\leq
\underbrace{
\min_{a\in\R}
\frac{1}{r^d}
\int_{\BB_r\cap\Omega}
|\partial_{ij}u_1(x)-a|\dd x
}_{=:I_1(x_0,r)}
\\
&\qquad+
\underbrace{
\min_{a\in\R}
\frac{1}{r^d}
\int_{\BB_r\cap\Omega}
|\partial_{ij}u_2(x)-a|\dd x
}_{=:I_2(x_0,r)}.
\end{align*}

On the one hand, applying
\cref{theorem:solvable-lp} to
$u_1=\cL_A^{-1}(f_1)$ with $p=2$, we get
\begin{align*}
I_1(x_0,r)
&\leq
\frac{1}{r^d}
|\BB_r\cap\Omega|^{1/2}
\left(
    \int_{\BB_r\cap\Omega}
    |\partial_{ij}u_1(x)|^2\dd x
\right)^{1/2}
\\
&\lesssim
r^{-d/2}\|f_1\|_{L^2(\Omega)}.
\end{align*}
Moreover,
\begin{align*}
\|f_1\|_{L^2(\Omega)}
&\leq
|E|^{1/2}\|f\|_{\infty,\Omega}
+
|c_r||\Omega|^{1/2}
\\
&\leq
|E|^{1/2}\|f\|_{\infty,\Omega}
+
\frac{|E|}{|\Omega|^{1/2}}
\|f\|_{\infty,\Omega}
\\
&\lesssim
r^{d/2}\|f\|_{\infty,\Omega}.
\end{align*}
Consequently,
\begin{equation}\label{eq:I1}
    I_1(x_0,r)
    \lesssim
    \|f\|_{\infty,\Omega}.
\end{equation}

On the other hand, since
$\int_\Omega G(x,y)\dd y=0$, for any
$x\in\BB_r\cap\Omega$,
\begin{align*}
u_2(x)
&=
\int_\Omega G(x,y)f_2(y)\dd y
\\
&=
\int_{E^c}G(x,y)f(y)\dd y
+
c_r\int_\Omega G(x,y)\dd y
\\
&=
\int_{\BB_{2r}^c\cap\Omega}
G(x,y)f(y)\dd y.
\end{align*}
If $x\in\BB_r(x_0)$ and
$y\in\BB_{2r}^c(x_0)$, then
$\|x-y\|\geq r$. Therefore,
\eqref{eq:green-kernel-size} gives, for fixed $r$,
\[
    |\partial_{ij,x}G(x,y)|
    \leq
    C(1+r^{-d}),
\]
which is integrable with respect to $y$ on the bounded domain
$\Omega$. Dominated convergence therefore justifies differentiating
twice under the integral, and
\begin{equation}\label{eq:derivative-u2}
    \partial_{ij}u_2(x)
    =
    \int_{\BB_{2r}^c\cap\Omega}
    \partial_{ij,x}G(x,y)f(y)\dd y.
\end{equation}

From~\cref{lemma:global-estimate-green-function}, for any
$x\in\BB_r\cap\Omega$ and
$y\in\BB_{2r}^c\cap\Omega$,
\begin{align*}
&|\partial_{ij,x}G(x,y)
-\partial_{ij,x}G(x_0,y)|
\\
&\qquad\lesssim
\frac{\|x-x_0\|}{\|x-y\|^{d+1}}
+
\frac{1}{\|x-y\|^{d-\alpha}}
+
\frac{1}{\|x_0-y\|^{d-\alpha}}.
\end{align*}
The last two terms are integrable with respect to $y$. Indeed,
letting $D=\operatorname{diam}(\Omega)$,
\begin{align*}
\sup_{x\in\Omega}
\int_\Omega
\frac{1}{\|x-y\|^{d-\alpha}}\dd y
&\leq
C_d\int_0^D s^{\alpha-1}\dd s
\\
&=
\frac{C_d}{\alpha}D^\alpha
<\infty.
\end{align*}
Denote this finite constant by $C_\Omega$. Taking
$a=\partial_{ij}u_2(x_0)$ and using
\eqref{eq:derivative-u2}, we obtain
\begin{align*}
I_2(x_0,r)
&\leq
\frac{1}{r^d}
\int_{\BB_r\cap\Omega}
|\partial_{ij}u_2(x)
-\partial_{ij}u_2(x_0)|
\dd x
\\
&\leq
\frac{\|f\|_{\infty,\Omega}}{r^d}
\int_{\BB_r\cap\Omega}
\int_{\BB_{2r}^c\cap\Omega}
|\partial_{ij,x}G(x,y)
-\partial_{ij,x}G(x_0,y)|
\dd y\dd x
\\
&\lesssim
\frac{\|f\|_{\infty,\Omega}}{r^d}
\int_{\BB_r\cap\Omega}
\left(
    2C_\Omega
    +
    \|x-x_0\|
    \int_{\BB_{2r}^c\cap\Omega}
    \frac{1}{\|x-y\|^{d+1}}\dd y
\right)
\dd x.
\end{align*}
For $x\in\BB_r(x_0)$,
\[
\begin{aligned}
\int_{\BB_{2r}^c\cap\Omega}
\frac{1}{\|x-y\|^{d+1}}\dd y
&\leq
\int_{\{\|x-y\|\geq r\}}
\frac{1}{\|x-y\|^{d+1}}\dd y
\\
&\lesssim
r^{-1}.
\end{aligned}
\]
Therefore,
\begin{equation}\label{eq:I2}
\begin{aligned}
I_2(x_0,r)
&\lesssim
\frac{\|f\|_{\infty,\Omega}}{r^d}
\int_{\BB_r\cap\Omega}
\left(
    1+\|x-x_0\|r^{-1}
\right)
\dd x
\\
&\lesssim
\|f\|_{\infty,\Omega}.
\end{aligned}
\end{equation}
Combining~\eqref{eq:I1} and~\eqref{eq:I2} gives
\[
    I(x_0,r)
    \lesssim
    \|f\|_{\infty,\Omega}.
\]
Taking the supremum over $x_0$ and
$0<r\leq\operatorname{diam}(\Omega)$ proves
\eqref{eq:BMO}.

Finally, since $\Omega$ is a bounded $\cC^{2,\alpha}$ domain, it is
a bounded Lipschitz domain and satisfies the uniform volume-density
property
\[
    c_\Omega r^d
    \leq
    |\BB_r(x_0)\cap\Omega|
    \leq
    C_dr^d,
    \qquad
    x_0\in\Omega,
    \quad
    0<r\leq\operatorname{diam}(\Omega).
\]
Thus~\eqref{eq:BMO} is equivalent, up to constants, to the
relative-ball BMO seminorm. Together with the $W^{2,2}$ estimate in
\eqref{eq:W2p-estimate}, this yields
\[
    \|[u]\|_{\mathcal X}
    \lesssim
    \|f\|_{\infty,\Omega},
\]
and concludes the proof of~\cref{theorem-Solvable}.
\section{Technical Lemmas}\label{appendix-tecnical-lemmas}

\begin{lemma}\label{lemma:BMO-inclusion}
Let $\Omega$ be bounded with $\cC^{2,\alpha}$ boundary, then 
\[
W^{2,BMO}(\Omega) \hookrightarrow \mathcal{C}^{1,{\rm LogLip}}(\Omega).
\]
\end{lemma}

\begin{proof}
Let $u\in W^{2,BMO}(\Omega)$. 
Assume first that $d=1$. Then $\Omega=(a,b)$. Put
$D:=\operatorname{diam}(\Omega)$ and $v:=u''$. For an interval
$I\subset\Omega$, denote by $v_I$ the average of $v$ on $I$. Choose
nested intervals
\[
I=I_0\subset I_1\subset\cdots\subset I_N=\Omega
\]
such that $|I_{k+1}|\leq2|I_k|$ and
\[
N\leq C\left(1+\log\left(\frac{D}{|I|}\right)\right).
\]
By the definition of the BMO norm,
\[
|v_{I_k}-v_{I_{k+1}}|
\leq
\frac{1}{|I_k|}
\int_{I_k}|v-v_{I_{k+1}}|
\leq
2\|v\|_{\mathrm{BMO}(\Omega)}.
\]
Consequently,
\begin{align*}
\int_I|v|
&\leq
\int_I|v-v_I|
+
|I|\,|v_I-v_\Omega|
+
|I|\,|v_\Omega|
\\
&\leq
C|I|
\left(1+\log\left(\frac{D}{|I|}\right)\right)
\|v\|_{\mathrm{BMO}(\Omega)}
+
\frac{|I|}{D}\|v\|_{L^1(\Omega)}
\\
&\leq
C|I|
\left(1+\log\left(\frac{D}{|I|}\right)\right)
\|u\|_{W^{2,BMO}(\Omega)}.
\end{align*}
Since $u'\in W^{1,2}(\Omega)$ has an absolutely continuous
representative, for $x<y$ with $|x-y|\leq1$, applying the previous
estimate to $I=(x,y)$ gives
\begin{align*}
|u'(y)-u'(x)|
&\leq
\int_x^y|u''(t)|\dd t
\\
&\leq
C\|u\|_{W^{2,BMO}(\Omega)}
|x-y|\bigl(1+|\log|x-y||\bigr).
\end{align*}
The remaining terms in the $\cC^{1,{\rm LogLip}}$ norm are controlled
by the one-dimensional Sobolev embedding and the $W^{2,2}$ part of
the $W^{2,BMO}$ norm. This proves the result when $d=1$.

Assume now that $d\geq2$. In view of John--Nirenberg inequality~\cite[Lemma~1]{JohnNirenberg1961}, there exist universal constants $c_1,c_2>0$ so that
\[
\Bigg|
\Bigg\{
x\in Q:
\left|
\partial_{ij}u(x)
-
\frac{1}{|Q|}\int_Q\partial_{ij}u
\right|
>\lambda
\Bigg\}
\Bigg|
\leq
c_1|Q|
\exp\left(
-\frac{c_2\lambda}
{\|u\|_{W^{2,BMO}(\Omega)}}
\right).
\]
Therefore, $\partial_i u\in W^{1,A}(\Omega)$ for
$A(t)=e^t-t-1$. Now, as $\Omega$ has $\cC^{2,\alpha}$ boundary,
\cite[Theorem~3]{cianchi1996continuity} implies
\[
|\partial_i u(x)-\partial_i u(x')|
\leq
C\|\partial_i u\|_{W^{1,A}(\Omega)}
H^{-1}(\|x-x'\|^{-d}),
\quad
\text{a.e. in }\Omega,
\]
where
\[
H(s)=\bigl(s\,\Theta^{-1}(s^{d'})\bigr)^{d'}
\quad {\rm and}\quad
\Theta(r)=d'\int_r^\infty
\frac{\widetilde A(t)}{t^{1+d'}}\dd t.
\]
Here, $d'$ denotes H\"older conjugate of $d$. Since
$A(t)=e^t-t-1$,
\[
\widetilde A(s)
=
\sup_{t\geq0}\{st-A(t)\}
=
(1+s)\log(1+s)-s,
\qquad s\geq0,
\]
and thus
\[
H^{-1}(\rho)
\leq
C\rho^{-1/d}(1+\log\rho),
\qquad
\rho\geq1,
\]
for some $C>0$. Consequently, for
$0<\|x-x'\|\leq1$,
\[
|\partial_i u(x)-\partial_i u(x')|
\leq
C\|u\|_{W^{2,BMO}(\Omega)}
\|x-x'\|
\bigl(1+|\log\|x-x'\||\bigr).
\]
Summing over $i=1,\dots,d$, we obtain
\[
\|\nabla u(x)-\nabla u(x')\|
\leq
C\|u\|_{W^{2,BMO}(\Omega)}
\|x-x'\|
\bigl(1+|\log\|x-x'\||\bigr),
\quad
\text{a.e. in }\Omega.
\]
The same Orlicz--Sobolev embedding controls
$\|\nabla u\|_{\infty,\Omega}$, while the lower-order terms are
controlled by the $W^{2,2}$ part of the norm. Moreover, the preceding
estimate provides a continuous representative of $\nabla u$. 
Our desired statement follows.
\end{proof}

Let us also record the following compactness fact, which shows that
weak-$\ast$ convergence of measures can be lifted to strong
convergence in the dual of the mean-zero Neumann space defined in \eqref{eq:zero-traze-space}.

\begin{lemma}\label{lemma:strong-convergence} 
Fix $p\in(1,\infty)$ with $p>\frac d2$. Let
$(\mu_n)_{n\in\N}\subset\cM_0(\Omega)$ be such that
$\mu_n\overset{\ast}{\rightharpoonup}\mu\in\cM_0(\Omega)$. Then
\[
\mu_n\longrightarrow\mu
\qquad\text{strongly in }E_p^*,
\]
and
\[
(\cL_A^*)^{-1}(\mu_n)
\longrightarrow
(\cL_A^*)^{-1}(\mu)
\qquad\text{strongly in }L_0^{p'}(\Omega).
\]
\end{lemma}

\begin{proof}
We shall prove a more general result: let $E,F$ be two Banach spaces
such that
$E\Subset F$. If
$x_n\overset{\ast}{\rightharpoonup}x$ in
$\sigma(F^*,F)$, the weak-$\ast$ topology on $F^*$, then
$x_n\to x$ strongly in $E^*$.

Fix $\vae>0$. As $E\Subset F$, the unit ball of $E$ is relatively
compact in $F$. Hence, there exist
$\{g_i\}_{i=1}^N\subset E$, with $\|g_i\|_E\leq1$, such that for every
$g\in E$ with $\|g\|_E\leq1$, there is an $i\in\{1,\dots,N\}$ for
which
\[
\|g-g_i\|_F\leq\vae.
\]
Therefore,
\[
\begin{aligned}
\|x_n-x\|_{E^*}
&=
\sup_{\|g\|_E\leq1}
|\langle x_n-x,g\rangle|
\\
&\leq
\max_{1\leq i\leq N}
|\langle x_n-x,g_i\rangle|
+
\vae\|x_n-x\|_{F^*}.
\end{aligned}
\] 
In view of
$x_n\overset{\ast}{\rightharpoonup}x$ in $\sigma(F^*,F)$,
\[
\sup_{n\in\N}\|x_n-x\|_{F^*}<\infty,
\]
and, since the collection $\{g_i\}_{i=1}^N$ is finite,
\[
\max_{1\leq i\leq N}
|\langle x_n-x,g_i\rangle|
\longrightarrow0.
\]
Thus,
\[
\limsup_{n\to\infty}\|x_n-x\|_{E^*}
\leq
\vae\sup_{n\in\N}\|x_n-x\|_{F^*}.
\] 
As $\vae>0$ is arbitrary, we conclude that
$x_n\to x$ strongly in $E^*$. 
We apply this result with
\[
E=E_p,
\qquad
F=\cC(\overline\Omega).
\]
By Morrey's inequality and the Rellich compactness theorem, as
recorded in~\eqref{eq:contentionW-C-embedding},
\[
E_p\Subset\cC(\overline\Omega).
\]
The weak-$\ast$ convergence of $\mu_n$ to $\mu$ therefore implies
\[
\|\mu_n-\mu\|_{E_p^*}\longrightarrow0,
\]
which proves the first claim. Finally, since
\[
(\cL_A^*)^{-1}:E_p^*\longrightarrow L_0^{p'}(\Omega)
\]
is bounded,
\begin{align*}
\|(\cL_A^*)^{-1}(\mu_n)
-(\cL_A^*)^{-1}(\mu)\|_{L^{p'}(\Omega)}
&\leq
C\|\mu_n-\mu\|_{E_p^*}
\longrightarrow0.
\end{align*}
This proves the second claim. 
\end{proof}
\section{Omitted proofs}\label{sect:omitted-proofs}

\begin{proof}[Proof of \cref{Lemma:Control-of-entropy}] We start with an auxiliary lemma. 
\begin{lemma}[Complexity of $C^1$ functions with log-Lip gradients in the
$L^\infty$ metric]
\label{lemma:covering-LogLip}
Let $\Omega\subset\mathbb R^d$ be bounded with
$\mathcal C^{2,\alpha}$ boundary. Then, for every $M>0$ and every
$0<\eps<M/2$,
\[
\log N\left(
\eps,
\cC_M^{1,\mathrm{LogLip}}(\Omega),
\|\cdot\|_{\infty,\Omega}
\right)
\lesssim
\left(\frac{M}{\eps}\right)^{d/2}
\left(\log(M/\eps)\right)^{(d+2)/2},
\]
where
\[
\cC_M^{1,\mathrm{LogLip}}(\Omega)
=
\left\{
f\in\cC^{1,\mathrm{LogLip}}(\Omega):
\|f\|_{1,\mathrm{LogLip},\Omega}\leq M
\right\}.
\]
The implicit constant depends only on $d$ and $\Omega$.
\end{lemma}

\begin{proof}[Proof of \Cref{lemma:covering-LogLip}]
Choose $R>0$ such that $\overline\Omega\subset\mathbb B_R$. By the
$C^{1,\omega}$ Whitney extension theorem
\citep[generalized sharp Whitney theorem, pp.~320--321]{Fefferman2009Extension},
applied with
\[
\omega(t)=t\bigl(1+|\log t|\bigr),
\qquad 0<t\leq1,
\]
every $f\in\mathcal C^{1,\mathrm{LogLip}}(\Omega)$ admits an extension
$\widetilde f\in\mathcal C^{1,\mathrm{LogLip}}(\mathbb B_R)$ satisfying
\[
\widetilde f=f\quad\text{on }\Omega,
\qquad
\|\widetilde f\|_{1,\mathrm{LogLip},\mathbb B_R}
\leq
C_\Omega
\|f\|_{1,\mathrm{LogLip},\Omega}.
\]
Consequently,
\[
N\left(
\eps,
\cC_M^{1,\mathrm{LogLip}}(\Omega),
\|\cdot\|_{\infty,\Omega}
\right)
\leq
N\left(
\eps,
\cC_{C_\Omega M}^{1,\mathrm{LogLip}}(\BB_R),
\|\cdot\|_{\infty,\BB_R}
\right).
\]
It therefore suffices to prove the claim on $\BB_R$. Without loss of generality, take $M=1$. It is enough to consider
$\eps>0$ sufficiently small. Let $\{x_k\}_{k=1}^K$ be a finite
$\delta$-net of $\BB_R$, with
\[
K\leq C_R\delta^{-d},
\]
and let $\mathcal Q=\{Q_k\}_{k=1}^K$ be the associated Voronoi
partition, assigning boundary points arbitrarily in case of ties.
Then, for every $x\in Q_k$,
\[
\|x-x_k\|\leq\delta.
\]
Moreover, since $\BB_R$ is convex, the segment joining $x_k$ and $x$
is contained in $\BB_R$.

Fix $\eps,\eta>0$. For any
$f\in\cC^{1,\mathrm{LogLip}}(\BB_R)$, define
\[
\widetilde f(x)
:=
\sum_{k=1}^K
\left(
\eps
\left\lfloor\frac{f(x_k)}{\eps}\right\rfloor
+
\eta
\left\langle
\left\lfloor\frac{\nabla f(x_k)}{\eta}\right\rfloor,
x-x_k
\right\rangle
\right)
1_{\{x\in Q_k\}},
\]
where the floor operation in
$\lfloor\nabla f(x_k)/\eta\rfloor$ is understood coordinatewise.
We regard the resulting piecewise-affine functions as centers of an
external cover; in particular, they need not belong to
$\cC^{1,\mathrm{LogLip}}(\BB_R)$.

The fact that $f$ has a bounded log-Lipschitz gradient means that
\[
f(x)
=
f(y)+\langle\nabla f(y),x-y\rangle+R_y(x),
\]
where, if $\|x-y\|\leq1$,
\begin{align*}
|R_y(x)|
&=
\left|
\int_0^1
\left\langle
\nabla f\bigl(y+t(x-y)\bigr)-\nabla f(y),
x-y
\right\rangle
\diff t
\right|
\\
&\leq
\|x-y\|^2
\int_0^1
t\left(1+\left|\log(t\|x-y\|)\right|\right)\diff t
\\
&\leq
\frac12
\|x-y\|^2
\left(\frac32-\log\|x-y\|\right).
\end{align*}

We now evaluate the approximation error. For every $x\in Q_k$,
\begin{align*}
|f(x)-\widetilde f(x)|
&=
\Bigg|
f(x_k)
+
\langle\nabla f(x_k),x-x_k\rangle
+
R_{x_k}(x)
\\
&\qquad
-
\eps
\left\lfloor\eps^{-1}f(x_k)\right\rfloor
-
\eta
\left\langle
\left\lfloor\eta^{-1}\nabla f(x_k)\right\rfloor,
x-x_k
\right\rangle
\Bigg|
\\
&\leq
\eps+\sqrt d\,\eta\delta
+
\frac12\delta^2
\left(\frac32-\log\delta\right).
\end{align*}
To choose $\delta$ such that
$\delta^2\log(1/\delta)\lesssim\eps$, take
\[
\delta
=
\sqrt{
\frac{2\eps}
{\log\bigl(1/(2\eps)\bigr)}
}.
\]
Indeed,
\[
\begin{aligned}
\delta^2\log(1/\delta)
&=
\frac{2\eps}{\log\bigl(1/(2\eps)\bigr)}
\frac12
\log\left(
\frac{\log\bigl(1/(2\eps)\bigr)}{2\eps}
\right)
\\
&=
\eps
+
\eps
\frac{
\log\log\bigl((2\eps)^{-1}\bigr)
}{
\log\bigl((2\eps)^{-1}\bigr)
}
\leq
2\eps
\end{aligned}
\]
for $\eps$ sufficiently small. Fixing
$\eta=\frac{\eps}{\sqrt d\,\delta},$
there exists a constant $C>0$, depending only on $d$ and $R$, such
that $\|f-\widetilde f\|_{\infty,\BB_R}
\leq C\eps.$ 

Each $\widetilde f$ can be encoded by a finite number of coefficients.
At each point $x_k$, the quantized value of $f(x_k)$ has
$O(\eps^{-1})$ possible values, while each quantized
partial-derivative coefficient has $O(\eta^{-1})$ possible values.
Thus, the number of possible coefficient vectors at each point is
bounded by $C\eps^{-1}\eta^{-d}.$ Consequently, the total number of possible centers is at most $\left(C\eps^{-1}\eta^{-d}\right)^K.$ 
Denoting the external covering number by $N_{\rm ext}$, we obtain
\begin{align*}
\log N_{\rm ext}
\left(
C\eps,
\cC^{1,\mathrm{LogLip}}_1(\BB_R),
\|\cdot\|_{\infty,\BB_R}
\right)
&\leq
K\log\left(C\eps^{-1}\eta^{-d}\right)
\\
&\lesssim
\left(
\frac{\log(1/(2\eps))}{\eps}
\right)^{d/2}
\log\left(C\eps^{-1}\eta^{-d}\right)
\\
&\lesssim
\left(\frac1\eps\right)^{d/2}
\left(\log(1/\eps)\right)^{(d+2)/2}.
\end{align*}
Since every nonempty ball of an external cover can be recentered at
an element of the class at the cost of multiplying its radius by
two, the same estimate holds for the usual covering number.
Replacing $\eps$ by a fixed multiple of $\eps$ does not affect the
claimed order.

The result for general $M$ follows by applying the estimate to
$f/M$. Finally, applying the estimate on $\BB_R$ with
$C_\Omega M$ in place of $M$ and restricting the covering functions
to $\Omega$ gives
\[
\begin{aligned}
\log N\left(
\eps,
\cC_M^{1,\mathrm{LogLip}}(\Omega),
\|\cdot\|_{\infty,\Omega}
\right)
&\lesssim
\left(\frac{C_\Omega M}{\eps}\right)^{d/2}
\left(
\log\frac{C_\Omega M}{\eps}
\right)^{(d+2)/2}
\\
&\lesssim
\left(\frac{M}{\eps}\right)^{d/2}
\left(\log(M/\eps)\right)^{(d+2)/2},
\end{aligned}
\]
where the last implicit constant depends only on $d$ and $\Omega$.
\end{proof}

With this lemma at hand, we can turn to the proof of \cref{Lemma:Control-of-entropy}.  First, by definition 
\[
\E[|\nu - \widehat{\nu}_n| _{1,\mathrm{LogLip}}']
=
\E \sup_{\|f\|_{1,\mathrm{LogLip},{\BB_R}}\leq 1}
|(\nu - \widehat{\nu}_n)(f)|.
\] 
Notice that the class
$\cC_1^{1,\mathrm{LogLip}}(\BB_R)$ has envelope bounded by $1$.
Moreover, for every $f,g\in\cC_1^{1,\mathrm{LogLip}}(\BB_R)$,
\[
\left(
\frac1n\sum_{i=1}^n|f(X_i)-g(X_i)|^2
\right)^{1/2}
\leq \|f-g\|_{\infty,\BB_R}.
\]
Therefore, a combination of Proposition~4.11 in
\citet{wainwright2019high} with Theorem~16 in
\citet{luxburg2004distance} yields
\[
\E \left[
\sup_{\|f\|_{1,\mathrm{LogLip},\BB_R}\leq 1}
|(\nu-\widehat{\nu}_n)(f)|
\right]
\leq
C\inf_{0<\delta\leq1}
\left\{
\delta+
\frac1{\sqrt n}
\int_{\delta/4}^1
\sqrt{
\log N\left(
\eps,
\cC_1^{1,\mathrm{LogLip}}(\BB_R),
\|\cdot\|_\infty
\right)}
\diff\eps
\right\}.
\] 
By \Cref{lemma:covering-LogLip}, up to an additive constant, the entropy
integral is upper bounded by
\[
\int_{\delta/4}^1
\left(\frac1{\eps}\right)^{d/4}
\left(\log(1/\eps)\right)^{(d+2)/4}
\diff\eps.
\] 

Let us start with the case $d<4$. The integral is upper bounded, up to a multiplicative constant, by
\[
\left(\frac{4}{4-d}\right)^{\frac{d+6}{4}}
\gamma\left(
\frac{d+6}{4},
\frac{4-d}{4}\log\left(\frac4\delta\right)
\right),
\]
where $\gamma$ is the lower incomplete Gamma function. 
This quantity remains bounded as $\delta\downarrow0$. Thus, taking
$\delta=n^{-1}$, we obtain
\[
\E[|\nu-\widehat{\nu}_n|_{1,\mathrm{LogLip}}']
\lesssim n^{-1/2},
\]
yielding the first part of the claim on the rates. 

Turning to $d=4$, the entropy integral in that case reads $C\log^{5/2}(4/\delta),$ 
and the claim follows by taking $\delta=n^{-1/2}$.

For $d>4$, as $\max_{\eps\in[\delta/4,1]}\log(1/\eps)
=
\log(4/\delta),$ 
we have
\begin{align*}
&\int_{\delta/4}^1
\left(\frac1{\eps}\right)^{d/4}
\left(\log(1/\eps)\right)^{(d+2)/4}
\diff\eps
\\
&\qquad\leq
\left(\log(4/\delta)\right)^{\frac{d+2}{4}}
\frac{4}{d-4}
\left\{
\left(\frac{\delta}{4}\right)^{1-d/4}-1
\right\}\lesssim
\delta^{1-d/4}
\left(\log(4/\delta)\right)^{\frac{d+2}{4}}.
\end{align*}
Consequently, Dudley's bound gives
\[
\E[|\nu-\widehat{\nu}_n|_{1,\mathrm{LogLip}}']
\lesssim
\inf_{0<\delta\leq1}
\left\{
\delta+
n^{-1/2}\delta^{1-d/4}
\left(\log(4/\delta)\right)^{\frac{d+2}{4}}
\right\}.
\]
For $n$ large enough, take
\[
\delta=\delta_n
:=
n^{-2/d}(\log n)^{(d+2)/d}.
\]
Since $\log(4/\delta_n)\asymp\log n$, we have
\[
n^{-1/2}\delta_n^{1-d/4}
\left(\log(4/\delta_n)\right)^{\frac{d+2}{4}}
\lesssim
n^{-2/d}(\log n)^{(d+2)/d}
=
\delta_n.
\]
We conclude that
\[
\E[|\nu-\widehat{\nu}_n|_{1,\mathrm{LogLip}}']
\lesssim
n^{-2/d}(\log n)^{(d+2)/d},
\]
which proves the result for $d>4$. Increasing the implicit constant
takes care of the finitely many remaining values of $n$. 
\end{proof}
\begin{proof}[Proof of \Cref{lem:empirical-sobolev-clt}]
Set $\gamma_p=1-\frac dp>0$ 
and let $\mathcal F_p$ denote the mean-zero representatives of the unit
ball of $W^{2,p}(D)/\langle1\rangle$. By Poincar\'e's inequality and
Morrey's inequality,
\[
\mathcal F_p
\subset
\left\{
f\in \cC^{1,\gamma_p}(D):
\|f\|_{1,\gamma_p;D}\leq C
\right\}.
\]
Consequently, the covering entropy of $\mathcal F_p$ in the uniform
metric satisfies (see \cite[Theorem~2.7.1]{van1996weak})
\[
\log N\bigl(\varepsilon,\mathcal F_p,\|\cdot\|_\infty\bigr)
\lesssim
\varepsilon^{-d/(1+\gamma_p)},
\qquad 0<\varepsilon<1.
\]
If \(f_1,\ldots,f_N\) form an \(\varepsilon\)-cover in the
uniform norm, then
\[
[f_k-\varepsilon,f_k+\varepsilon],
\qquad k=1,\ldots,N,
\]
form \(L^2(\nu)\)-brackets of width \(2\varepsilon\).
Consequently,
\[
N_{[]}(2\varepsilon,\mathcal F_p,L^2(\nu))
\leq
N(\varepsilon,\mathcal F_p,\|\cdot\|_\infty).
\] Therefore,
\[
\log N_{[]}\bigl(\varepsilon,\mathcal F_p,L^2(\nu)\bigr)
\lesssim
\varepsilon^{-d/(1+\gamma_p)}.
\]
It follows that
\[
\int_0^1
\sqrt{
\log N_{[]}\bigl(
\varepsilon,\mathcal F_p,L^2(\nu)
\bigr)}
\,\rd\varepsilon
<\infty
\]
whenever $\frac{d}{1+\gamma_p}<2,$ 
or, equivalently, $p>\frac{2d}{4-d}.$ 
Moreover, $\mathcal F_p$ is pointwise measurable, since
$W^{2,p}(D)$ is separable and its inclusion into $\cC(D)$ is
continuous. Hence, by the central limit theorem for classes with
finite bracketing entropy
\citep[Theorem~2.5.2]{van1996weak},
\[
\sqrt n(\nu-\widehat\nu_n)
\rightsquigarrow
\mathbb G_\nu
\qquad\text{in }\ell^\infty(\mathcal F_p).
\]
The canonical map
\[
(W^{2,p}(D)/\langle1\rangle)^*\longrightarrow\ell^\infty(\mathcal F_p),
\qquad
\gamma\longmapsto\bigl(f\mapsto\gamma(f)\bigr),
\]
is an isometric embedding with closed range. Since
$\sqrt n(\nu-\widehat\nu_n)$ belongs to this closed subspace for every
$n$, its weak limit also belongs to it almost surely. Thus
$\mathbb G_\nu$ is a tight Gaussian random element of $(W^{2,p}(D)/\langle1\rangle)^*$. Its covariance
is the covariance of the empirical process, namely
\[
\E[\mathbb G_\nu(f)\mathbb G_\nu(g)]
=
\operatorname{Cov}_\nu(f,g).
\] 
\end{proof}
\begin{proof}[Proof of \cref{lem:local-affine-lower}]
Let $X,Y$ be independent and uniform on $K$, put $M=(X+Y)/2$,
and set $h=f-\ell$.  Strong convexity gives
\[
\frac{\kappa}{8}\|X-Y\|^2
\leq\frac{h(X)+h(Y)}2-h(M).
\]
The density of $M$ is bounded by $2^d/|K|$.  Consequently,
\[
\frac{\kappa}{8}\E\|X-Y\|^2
\leq\frac{1+2^d}{|K|}\int_C|h(x)|\,\rd x.
\]
If $\bar x=\E X$, then
$\E\|X-Y\|^2=2\E\|X-\bar x\|^2$.  By the rearrangement principle,
the ball minimizes the second moment among sets of fixed volume, so
\[
\E\|X-\bar x\|^2
\geq\frac{d}{d+2}\omega_d^{-2/d}|K|^{2/d}.
\]
Combining the displays proves the claim.
\end{proof}
\begin{proof}[Proof of \cref{prop:polyhedral-lower}]
For fixed $c$, the active cells of the affine pieces of $g+c$ form,
up to null boundaries, a partition of $\Omega$ into at most $n$
convex sets $C_j$.  Applying
\cref{lem:local-affine-lower} on every cell and then Jensen's
inequality gives
\[
\|g-f-c\|_{L^1(\mu)}
\geq\lambda c_d\kappa\sum_j|C_j|^{1+2/d}
\geq\lambda c_d\kappa|\Omega|^{1+2/d}n^{-2/d}.
\]
The bound is uniform in $c$.
\end{proof}
\begin{proof}[Proof of \cref{lem:bootstrap-W2}]
Conditional on the sample
$X_{1:n}=(X_1,\ldots,X_n)$, let
$\pi_n\in\Pi(\widehat\nu_n,\nu)$ be a measurable
choice of an optimal coupling, and let
$(X_i^*,Y_i)_{i=1}^n$ be conditionally i.i.d.~with common distribution
$\pi_n$. 
Define
\[
\widetilde\nu_n
=
\frac1n\sum_{i=1}^n\delta_{Y_i}.
\]
Conditional on $X_{1:n}$, the variables $X_1^*,\dots,X_n^*$ are
i.i.d.~with distribution $\widehat\nu_n$, whereas
$Y_1,\dots,Y_n$ are i.i.d.~with distribution $\nu$. Thus
$\widehat\nu_n^*=n^{-1}\sum_{i=1}^n\delta_{X_i^*}$ is the bootstrap
empirical measure, and the conditional distribution of
$\widetilde\nu_n$ is the same as the unconditional distribution of
$\widehat\nu_n$.

Given a random variable $Z$, write
$\E^*[Z]=\E[Z\mid X_{1:n}]$. Since
\[
\frac1n\sum_{i=1}^n\delta_{(X_i^*,Y_i)}
\]
is a coupling between $\widehat\nu_n^*$ and $\widetilde\nu_n$, and
in view of the optimality of $\pi_n$, we have
\[
\E^*\!\left[
\mathcal W_2^2(\widehat\nu_n^*,\widetilde\nu_n)
\right]
\leq
\E^*\!\left[
\frac1n\sum_{i=1}^n\|X_i^*-Y_i\|^2
\right]
=
\mathcal W_2^2(\widehat\nu_n,\nu).
\]
Moreover,
\[
\E^*\!\left[
\mathcal W_2^2(\widetilde\nu_n,\nu)
\right]
=
\E\!\left[
\mathcal W_2^2(\widehat\nu_n,\nu)
\right].
\]
Therefore,
\[
\begin{aligned}
\E^*\!\left[
\mathcal W_2^2(\widehat\nu_n^*,\nu)
\right]
&\leq
2\left(
\E^*\!\left[
\mathcal W_2^2(\widehat\nu_n^*,\widetilde\nu_n)
\right]
+
\E^*\!\left[
\mathcal W_2^2(\widetilde\nu_n,\nu)
\right]
\right)
\\
&\leq
2\mathcal W_2^2(\widehat\nu_n,\nu)
+
2\E\!\left[
\mathcal W_2^2(\widehat\nu_n,\nu)
\right].
\end{aligned}
\]
Taking expectations yields
\[
\E\!\left[
\mathcal W_2^2(\widehat\nu_n^*,\nu)
\right]
\leq
4\E\!\left[
\mathcal W_2^2(\widehat\nu_n,\nu)
\right].
\]
The second inequality in the statement follows from
\eqref{eq:W2}.
\end{proof}
\begin{proof}[Proof of \cref{Thm:rates-and-CLT-bootstrap}]
Set
\[
\gamma_p=1-\frac dp
\qquad\text{and}\qquad
\mathcal K
=
(\mathcal L_A^{-1})^*
\circ
\mathcal T_{\nabla\varphi}.
\]
By \cref{lem:empirical-sobolev-clt}, the unit ball of
$W^{2,p}(\Omega')/\langle1\rangle$ is a $\nu$-Donsker class.
Consequently, the bootstrap central limit theorem for Donsker
classes \citep[Section~3.6]{van1996weak} yields
$\mathbb G_n^*
\rightsquigarrow_{\mathbb P}
\mathbb G_\nu$ in
$\bigl(W^{2,p}(\Omega')/\langle1\rangle\bigr)^*$. Note that
$\mathbb G_\nu$ is centered Gaussian and hence symmetric. Therefore,
$-\mathbb G_n^*
\rightsquigarrow_{\mathbb P}
\mathbb G_\nu$.
Since $\mathcal K$ is bounded, the bootstrap continuous mapping theorem
\citep[Theorem~10.8]{Kosorok.2008.book} gives
\[
-\mathcal K(\mathbb G_n^*)
\rightsquigarrow_{\mathbb P}
\mathcal K(\mathbb G_\nu)
\qquad\text{in}\qquad
L^{p'}(\Omega)/\langle1\rangle.
\]
It remains to control the nonlinear remainder. Fix
$g\in L^p_0(\Omega)$ and set $u_g=\mathcal L_A^{-1}g$.
Applying the linearization estimate used in the proof of
\cref{Thm:rates-and-CLT} (see
\cref{Prop:bound-in-terms-calpha}) first to
$(\widehat\varphi_{n,*},\widehat\nu_n^*)$ and then to
$(\widehat\varphi_n,\widehat\nu_n)$, write
\[
\begin{aligned}
R_{n,*}(g)
&:=
(\widehat\nu_n^*-\nu)
\bigl(u_g\circ\nabla\varphi^*\bigr)
+
\int_\Omega
g(\widehat\varphi_{n,*}-\varphi)\,\rd x,
\\
R_n(g)
&:=
(\widehat\nu_n-\nu)
\bigl(u_g\circ\nabla\varphi^*\bigr)
+
\int_\Omega
g(\widehat\varphi_n-\varphi)\,\rd x.
\end{aligned}
\]
The two linearization estimates give
\[
\begin{aligned}
|R_{n,*}(g)|
&\lesssim
\|g\|_{L^p(\Omega)}
\mathcal W_2(\widehat\nu_n^*,\nu)^{1+\gamma_p},
\\
|R_n(g)|
&\lesssim
\|g\|_{L^p(\Omega)}
\mathcal W_2(\widehat\nu_n,\nu)^{1+\gamma_p}.
\end{aligned}
\]
Subtracting the two identities defining $R_{n,*}(g)$ and $R_n(g)$,
and then applying the triangle inequality, gives 
\[
\begin{aligned}
&
\left|
(\widehat\nu_n^*-\widehat\nu_n)
\bigl(u_g\circ\nabla\varphi^*\bigr)
+
\int_\Omega
g(\widehat\varphi_{n,*}-\widehat\varphi_n)\,\rd x
\right|
\\
&\qquad\lesssim
\|g\|_{L^p(\Omega)}
\left\{
\mathcal W_2(\widehat\nu_n^*,\nu)^{1+\gamma_p}
+
\mathcal W_2(\widehat\nu_n,\nu)^{1+\gamma_p}
\right\}.
\end{aligned}
\]
Here we used
\[
\int_\Omega
\left\langle
\mu[\nabla^2\varphi]^{-1}\nabla u_g,
\nabla(\widehat\varphi_{n,*}-\widehat\varphi_n)
\right\rangle
\rd x
=
-
\int_\Omega
g(\widehat\varphi_{n,*}-\widehat\varphi_n)\,\rd x,
\]
which follows by integration by parts and the homogeneous conormal
boundary condition satisfied by $u_g$. Taking the supremum over
$\{g\in L^p_0(\Omega):\|g\|_{L^p(\Omega)}\leq1\}$ yields
\[
\begin{aligned}
&
\left\|
\sqrt n
(\widehat\varphi_{n,*}-\widehat\varphi_n)
+
\mathcal K(\mathbb G_n^*)
\right\|_{L^{p'}(\Omega)/\langle1\rangle}
\\
&\qquad\lesssim
\sqrt n
\left\{
\mathcal W_2(\widehat\nu_n^*,\nu)^{2-d/p}
+
\mathcal W_2(\widehat\nu_n,\nu)^{2-d/p}
\right\}
=:\Delta_n^*.
\end{aligned}
\]
By Jensen's inequality,
\cref{lem:bootstrap-W2}, and \eqref{eq:W2},
\[
\begin{aligned}
\sqrt n\,
\E\!\left[
\mathcal W_2(\widehat\nu_n^*,\nu)^{2-d/p}
\right]
&\leq
\sqrt n
\left\{
\E\!\left[
\mathcal W_2^2(\widehat\nu_n^*,\nu)
\right]
\right\}^{1-\frac{d}{2p}}
\\
&\lesssim
\sqrt n\,\alpha(n,d)^{1-\frac{d}{2p}}
\longrightarrow0.
\end{aligned}
\]
The same argument gives
\[
\sqrt n\,
\E\!\left[
\mathcal W_2(\widehat\nu_n,\nu)^{2-\frac{d}{p}}
\right]
\longrightarrow0.
\]
Indeed,
$\sqrt n\,\alpha(n,d)^{1-\frac{d}{2p}}\to0$
for $d=1,2,3$ precisely under the condition
$p>\max\left\{d,\frac{2d}{4-d}\right\}$.
The above bounds imply $\E[\Delta_n^*]\to0$. Consequently,
$\E^*[\Delta_n^*]\overset{\mathbb P}{\longrightarrow}0$,
and conditional Markov's inequality gives, for every
$\varepsilon>0$,
\[
\mathbb P\!\left(
\Delta_n^*>\varepsilon
\,\middle|\,
X_{1:n}
\right)
\overset{\mathbb P}{\longrightarrow}0.
\]
Thus
\[
\sqrt n
(\widehat\varphi_{n,*}-\widehat\varphi_n)
=
-\mathcal K(\mathbb G_n^*)
+
o_{\mathbb P^*}(1)
\]
in $L^{p'}(\Omega)/\langle1\rangle$, in probability with respect to
the original sample. The conclusion follows from the conditional
Slutsky theorem.
\end{proof}
\begin{proof}[Proof of \Cref{prop:shadow-premium}]
By Kantorovich duality,
\[
\rho_{\mu_t}(\nu)
=
\inf_{\psi\ {\rm convex}}
\left[
\int\psi\,\rd\mu+\int\psi^*\,\rd\nu
+t\int\psi w_{A,B}\,\rd\mu
\right].
\]
The conclusion follows from the envelope theorem.  More explicitly,
the upper and lower directional bounds follow by inserting,
respectively, the optimizer at $t=0$ and the optimizer at $t$ into
the two dual objectives, and then using uniqueness modulo constants
and stability of normalized potentials.  The normalization is
immaterial because $\int w_{A,B}\,\rd\mu=0$.
\end{proof}
\end{document}